\documentclass[11pt]{amsart}

\usepackage{graphicx,latexsym,pinlabel,amssymb,arydshln,yhmath}

\usepackage[all]{xy}
\SelectTips{cm}{}

\theoremstyle{plain}
\newtheorem{theorem}{Theorem}[section]
\newtheorem{lemma}[theorem]{Lemma}
\newtheorem{proposition}[theorem]{Proposition}
\newtheorem{claim}[theorem]{Claim}
\newtheorem{corollary}[theorem]{Corollary}

\theoremstyle{definition}

\newtheorem{example}[theorem]{Example}
\newtheorem{remark}[theorem]{Remark}

\newtheorem{definition}[theorem]{Definition}

\numberwithin{equation}{section}
\numberwithin{figure}{section}

\def \ie{i.e$.$ }

\def \mod{\ \operatorname{mod}\ }

\def \Z{\mathbb{Z}}
\def \Q{\mathbb{Q}}
\def \R{\mathbb{R}}

\def \rhoLMO{\varrho^{\Ztilde}}
\def \thetaLMO{\theta^{\Ztilde}}
\def \tauLMO{\tau^{\Ztilde}}

\def \Ztilde{\widetilde{Z}}
\def \ZtildeY{\widetilde{Z}^Y}
\def \ZtildeYt{\widetilde{Z}^{Y,t}}

\def \A{\mathcal{A}}
\def \tsA{{}^{ts}\!\!\mathcal{A}}
\def \homotopy{\mathcal{H}}

\def \Lie{\mathfrak{L}}
\def \Liehat{\widehat{\mathfrak{L}}}

\def \MGroup{\mathsf{M}}
\def \MLie{\mathfrak{m}}

\def \T{\mathcal{T}}
\def \Filt{\mathcal{F}}

\def \bottomK{\mathcal{B}}
\def \stringK{\mathcal{S}}
\def \cyl{\mathcal{IC}}
\def \Cob{\mathcal{C}ob}
\def \I{\mathcal{I}}
\def \LCob{\mathcal{LC}ob}

\def \Aut{\operatorname{Aut}}

\def \can{\operatorname{canonical}}
\def \comm{\operatorname{comm}}
\def \col{\operatorname{col}}
\def \deg{\operatorname{deg}}
\def \D{\operatorname{D}}
\def \Der{\operatorname{Der}}

\def \GLike{\operatorname{GLike}}
\def \Gr{\operatorname{Gr}}
\def \Hom{\operatorname{Hom}}
\def \Id{\operatorname{Id}}
\def \ideg{\operatorname{i-deg}}
\def \interior{\operatorname{int}}
\def \K{\operatorname{K}}
\def \Ker{\operatorname{Ker}}
\def \IAut{\operatorname{IAut}}
\def \Img{\operatorname{Im}}
\def \IO{\operatorname{IO}}
\def \P{\operatorname{P}}
\def \Prim{\operatorname{Prim}}
\def \Tens{\operatorname{T}}
\def \Tenshat{\widehat{\operatorname{T}}}
\def \Tor{\operatorname{Tor}}
\def \tore{\operatorname{t}}
\def \U{\operatorname{U}}
\def \Uhat{\widehat{\operatorname{U}}}

\def \longrightleftarrows{\begin{array}{c}\longrightarrow \\[-0.3cm] \longleftarrow \end{array}}

\newcommand{\set}[1]{\lfloor #1\rceil}

\newcommand{\gp}[1]{\mathsf{#1}}

\newcommand{\strutgraph}[2]
{ \! \! \begin{array}{c} 
\phantom{.}\\[-10pt]
\labellist \small \hair 2pt 
\pinlabel {\scriptsize $#1$} [l] at 37 18
\pinlabel {\scriptsize $#2$} [l] at 37 170
\endlabellist
\includegraphics[scale=0.1]{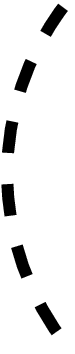} 
\end{array} \  \ }

\newcommand{\Ygraphbottoptop}[3]
{  \begin{array}{c} 
\hphantom{.}\\
\labellist \small \hair 2pt 
\pinlabel {\scriptsize $#1$} [t] at 56 0
\pinlabel {\scriptsize $#2$} [b] at 0 173
\pinlabel {\scriptsize $#3$} [b] at 122 173
\endlabellist
\includegraphics[scale=0.1]{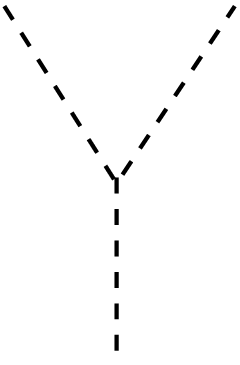}\\
\hphantom{.}
\end{array}  }

\begin{document}

\title[]{Infinitesimal Morita homomorphisms\\and the tree-level of the LMO invariant}

\date{December 13, 2011}

\author[]{Gw\'ena\"el Massuyeau}
\address{Institut de Recherche Math\'ematique Avanc\'ee, Universit\'e de Strasbourg \& CNRS,
7 rue Ren\'e Descartes, 67084 Strasbourg, France}
\email{massuyeau@math.unistra.fr}



\thanks{Partially supported by the French ANR research project ANR-08-JCJC-0114-01.}

\begin{abstract}
Let $\Sigma$ be a compact connected oriented surface with one boundary component,
and let $\pi$ be the fundamental group of $\Sigma$.
The Johnson filtration is a decreasing sequence of subgroups of the Torelli group of $\Sigma$,
whose $k$-th term consists of the self-homeomorphisms of $\Sigma$ 
that act trivially at the level of the $k$-th nilpotent quotient of $\pi$.
Morita defined a homomorphism from the $k$-th term of the Johnson filtration 
to the third homology group of the $k$-th nilpotent quotient of $\pi$.

In this paper, we replace groups by their Malcev Lie algebras 
and we study the ``infinitesimal'' version of the $k$-th Morita homomorphism,
which is shown to correspond to the original version by a canonical isomorphism.
We provide a diagrammatic description of the $k$-th infinitesimal Morita homomorphism
and, given an expansion of the free group $\pi$ that is ``symplectic'' in some sense,
we show how to compute it from Kawazumi's ``total Johnson map''. 

Besides, we give a topological interpretation of the full tree-reduction of the LMO homomorphism,
which is a diagrammatic representation of the Torelli group derived from the Le--Murakami--Ohtsuki invariant of $3$-manifolds.
More precisely, a  symplectic expansion of $\pi$ is constructed from the LMO invariant,
and it is shown  that the tree-level of the LMO homomorphism
is equivalent to the total Johnson map induced by this specific expansion.
It follows that the $k$-th infinitesimal Morita homomorphism coincides 
with the degree $[k,2k[$ part of the tree-reduction of the LMO homomorphism.  
Our results also apply to the monoid of homology cylinders over $\Sigma$.
\end{abstract}

\maketitle


%
%
%
\section*{Introduction}

Nilpotent homotopy types of  $3$-manifolds have been introduced by Turaev \cite{Turaev}.
They are defined by elementary tools from algebraic topology as follows. 
We fix an integer $k\geq 1$ and an abstract group $G$ of nilpotency class $k$, 
which means that commutators of length $(k+1)$ are trivial in $G$.
Let $M$ be a closed connected oriented $3$-manifold, 
whose $k$-th nilpotent quotient of the fundamental group is parametrized by the group $G$:
$$
\psi: G \stackrel{\simeq}{\longrightarrow} \pi_1(M)/\Gamma_{k+1} \pi_1(M).
$$
Then, the \emph{$k$-th nilpotent homotopy type} of the pair $(M,\psi)$ 
is the homology class
$$
\mu_k(M,\psi) :=  f^\psi_*\left([M]\right)  \in H_3(G;\Z)
$$
where $f^\psi: M \to \K(G,1)$ induces the composition 
$\pi_1(M) \to \pi_1(M)/ \Gamma_{k+1} \pi_1(M) \stackrel{\psi^{-1}}{\to} G$ at the level of fundamental groups.
For example, for $k=1$, we are considering the \emph{abelian homotopy type} of $3$-manifolds
which, by the work of Cochran, Gerges and Orr \cite{CGO}, 
is very well understood: $\mu_1(M,\psi)$ determines the cohomology ring of $M$ 
together with its linking pairing, and vice versa. 

As suggested to the author by Turaev, one way to study the invariant $\mu_k$ for higher $k$ is to study its behaviour under surgery. 
This method particularly applies if one wishes to understand nilpotent homotopy types 
from the point of view of finite-type invariants, which was our initial motivation.
Note that,  to compare the $k$-th nilpotent homotopy type of a manifold after surgery
with that of the manifold before surgery, we can only admit surgeries that preserve 
the $k$-th nilpotent quotient of the fundamental group (up to isomorphism).
The following type of surgery  is admissible in that  sense.
We consider a compact connected oriented surface
$S \subset M$ with one boundary component, and a homeomorphism 
$s: S \to S$ whose restriction to $\partial S$ is the identity
and which acts trivially at the level of $\pi_1(S)/ \Gamma_{k+1} \pi_1(S)$.
Then, we can ``twist'' $M$ along $S$ by $s$ to obtain the new manifold
$$
M_S := \left(M \setminus \interior(S \times [-1,1])\right) 
\cup_{(s\times 1) \cup (\Id \times (-1))} S \times [-1,1].
$$
The Seifert--Van Kampen theorem shows the existence of a canonical isomorphism
$$
\pi_1(M)/\Gamma_{k+1} \pi_1(M) \stackrel{\simeq}{\longrightarrow}
\pi_1(M_S)/\Gamma_{k+1} \pi_1(M_S) 
$$
which is defined by the following commutative diagram:
$$
\xymatrix{
&{\frac{\pi_1\left(M \setminus \interior(S \times [-1,1])\right)}
{\Gamma_{k+1}\pi_1\left(M \setminus \interior(S \times [-1,1])\right)}} 
\ar@{->>}[ld] \ar@{->>}[rd] & \\
{\frac{\pi_1(M)}{\Gamma_{k+1} \pi_1(M)}} \ar@{-->}[rr]_-\simeq^-{\exists !}
& & {\frac{\pi_1(M_S)}{\Gamma_{k+1} \pi_1(M_S)}}.
}
$$
By composing it with $\psi$, we obtain a parametrization
$$
\psi_S: G \stackrel{\simeq}{\longrightarrow} \pi_1(M_S)/\Gamma_{k+1} \pi_1(M_S)
$$
of the $k$-th nilpotent quotient of $\pi_1(M_S)$.
In order to compare $\mu_k(M,\psi)$ with $\mu_k(M_S,\psi_S)$, 
we consider the mapping torus of $s$
$$
\tore(s) := (S \times [-1,1]\ /\! \sim) \cup \left(S^1 \times D^2\right).
$$
Here, the equivalence relation $\sim$ identifies $s(x) \times 1$ with $x \times (-1)$,
and the meridian $1 \times \partial D^2$ of the solid torus $S^1 \times D^2$ 
is glued along the circle $* \times [-1,1]\ /\! \sim$ (where $*\in \partial S$) 
while the longitude $S^1 \times 1$ is glued along $\partial S \times 1$. 
We note that $\tore(s)$ is a closed connected oriented $3$-manifold,
and that the inclusion $S= S \times 1 \subset \tore(s)$ defines an isomorphism
$$
\varphi_s: \pi_1(S)/\Gamma_{k+1} \pi_1(S)
\stackrel{\simeq}{\longrightarrow}  \pi_1(\tore(s))/ \Gamma_{k+1} \pi_1(\tore(s)). 
$$
Besides, the inclusion $S \subset M$ defines a homomorphism
$$
i: \pi_1(S)/\Gamma_{k+1} \pi_1(S) \longrightarrow \pi_1(M)/\Gamma_{k+1} \pi_1(M).
$$
Then, $\mu_k$ varies as follows\footnote{This can be proved by a simple homological computation
in a singular $3$-manifold that contains the three of $M$, $M_S$ and $\tore(s)$. 
Similar formulas are shown in \cite[Theorem 2]{GL} and \cite[Theorem 5.2]{Heap} by cobordism arguments.} 
under the surgery $M \leadsto M_S$:
\begin{equation}
\label{eq:gluing_formula}
\mu_k(M_S,\psi_S) - \mu_k(M,\psi) = \psi_*^{-1} i_*\big( \mu_k(\tore(s),\varphi_s)\big) \ \in H_3(G;\Z).
\end{equation}

This variation formula suggests the following construction,
relative to a compact connected oriented surface $\Sigma$ with one boundary component.
Let $\I(\Sigma)$ be the Torelli group of $\Sigma$ 
and let $\pi:= \pi_1(\Sigma,*)$ be the fundamental group of $\Sigma$, where $* \in \partial \Sigma$.
Let
$$
\I(\Sigma)= \I(\Sigma)[1] \supset \I(\Sigma)[2] \supset \I(\Sigma)[3] \supset \cdots
$$
be the \emph{Johnson filtration} of $\I(\Sigma)$,  whose $k$-th subgroup $\I(\Sigma)[k]$ 
consists of (the isotopy classes of) the homeomorphisms $s:\Sigma \to \Sigma$ 
that act trivially at the level of $\pi/\Gamma_{k+1} \pi$.
The previous discussion shows that the map
$$
M_k: \I(\Sigma)[k] \longrightarrow H_3\left(\pi/\Gamma_{k+1} \pi;\Z\right),
\ s \longmapsto \mu_k(\tore(s),\varphi_s)
$$
plays a crucial role in the study of nilpotent homotopy types,
and formula (\ref{eq:gluing_formula}) shows that it is a group homomorphism.
The homomorphism $M_k$ has been studied by Heap in \cite{Heap}.
By considering the simplicial model of $\K\left(\pi/\Gamma_{k+1} \pi,1\right)$,
he proves that $M_k$ is equal to Morita's refinement of the $k$-th Johnson homomorphism,
whose definition is purely algebraic and involves the bar complex of a group \cite{Morita}.
Thus, in the sequel, we will refer to $M_k$ as the \emph{$k$-th Morita homomorphism}.

Since Lie algebra homology is simpler than group homology,
one would like to replace the group $\pi/\Gamma_{k+1}\pi$ by its Malcev Lie algebra 
$\MLie(\pi/\Gamma_{k+1}\pi)$ in the above discussion.
Thus, one defines a Lie analogue of the $k$-th Morita homomorphism
$$
m_k: \I(\Sigma)[k] \longrightarrow 
H_3\big(\MLie\left(\pi/\Gamma_{k+1} \pi\right);\Q\big)
$$
by imitating Morita's original definition of $M_k$ \cite{Morita}, 
the bar complex of a group being simply replaced by the Koszul complex of its Malcev Lie algebra.
The homomorphism $m_k$ and, in particular its relationship with the theory of finite-type invariants,
is the main subject of this paper whose contents we now describe.\\

First of all, let us recall that the Lie algebra $\MLie\left(\pi/\Gamma_{k+1} \pi\right)$
is free nilpotent of class $k$. More precisely, if we set $H:= H_1(\Sigma;\Q)$ and if we denote
by $\Lie(H)$ the free Lie algebra generated by $H$, then we have a (non-canonical) isomorphism
\begin{equation}
\label{eq:non-canonical}
\Lie(H)/\Gamma_{k+1} \Lie(H) \simeq \MLie\left(\pi/\Gamma_{k+1} \pi\right).
\end{equation}
In \S \ref{sec:fission}, we start by describing the third homology group
of a free nilpotent Lie algebra in terms of Jacobi diagrams, 
which are commonly encountered in the theory of finite-type invariants \cite{Ohtsuki}.
To be more explicit, let 
$$
\T(H) = \bigoplus_{d=1}^{+ \infty} \T_d(H) 
$$ 
be the graded vector space of Jacobi diagrams that are tree-shaped, connected, $H$-colored, 
and subject to the usual AS, IHX and multilinearity relations. 
The \emph{internal degree} $d\geq 1$ of such a diagram is the number of trivalent vertices. 
We define an explicit linear map
$$
\Phi: \bigoplus_{d=k}^{2k-1} \T_d(H) \longrightarrow 
H_3\left(\frac{\Lie(H)}{\Gamma_{k+1} \Lie(H)};\Q\right)
$$
and, thanks to prior computations of Igusa and Orr \cite{IO}, 
we show that the map $\Phi$ is an isomorphism (Theorem \ref{th:fission_iso}).

Next, \S \ref{sec:expansions} is mainly expositional and deals with \emph{expansions} of the free group $\pi$.
These are generalizations of the classical ``Magnus expansion'' of $\pi$ and, essentially, they are algebra isomorphisms
\begin{equation}
\label{eq:algebra_iso}
\widehat{\Q}[\pi] \simeq \Tenshat(H)
\end{equation}
between the $I$-adic completion of the group algebra of $\pi$ and the complete tensor algebra over $H$. 
Expansions have been studied by Lin in the context of Vassiliev invariants 
and Milnor's $\mu$ invariants \cite{Lin}, and by Kawazumi in connection with Johnson homomorphisms and
the cohomology of mapping class groups \cite{Kawazumi}. 
If the identification (\ref{eq:algebra_iso}) is required to be a Hopf algebra isomorphism,
then the expansion is said to be \emph{group-like} and it induces a Lie algebra isomorphism
\begin{equation}
\label{eq:Lie_algebras}
\MLie(\pi) \simeq \Liehat(H)
\end{equation}
between the Malcev Lie algebra of $\pi$ and the complete Lie algebra over $H$.
Then, each group-like expansion induces an isomorphism (\ref{eq:non-canonical}) for all $k\geq 1$.
We also introduce \emph{symplectic} expansions which relate the boundary of $\Sigma$ 
to the symplectic form of $H$ defined by the intersection pairing.

In \S \ref{sec:total_Johnson}, we review Johnson homomorphisms and
the ``total Johnson map'' defined by Kawazumi in \cite{Kawazumi}.
This is a way of encoding the \emph{Dehn--Nielsen representation} of the Torelli group
$$
\rho: \I(\Sigma) \longrightarrow \Aut(\pi), \ h \longmapsto h_*
$$
which depends on the choice of a group-like expansion $\theta$. 
More precisely, by passing to the Malcev Lie algebra of $\pi$ and by using the identification (\ref{eq:Lie_algebras}) induced by $\theta$, 
$\rho$ translates into a group homomorphism 
$$
\varrho^\theta: \I(\Sigma) \longrightarrow \Aut(\Liehat(H)).
$$ 
Since $h$ acts trivially in homology, there is no loss of information in defining
$$
\tau^\theta(h):= \varrho^\theta(h)|_{H} - \Id_H \ \in \Hom(H,\Liehat_{\geq 2}(H))
$$
where $\Liehat_{\geq 2}(H)$ denotes the degree $\geq 2$ part of $\Liehat(H)$.
Then, the \emph{total Johnson map} relative to $\theta$ is the map
\vspace{-0.4cm}
$$
\tau^\theta:  \I(\Sigma) \longrightarrow \Hom(H,\Liehat_{\geq 2}(H)) 
\simeq H^* \otimes \Liehat_{\geq 2}(H) 
\stackrel{\begin{array}{c} \hbox{\scriptsize Poincar\'e}\\[-0.1cm] \hbox{\scriptsize duality}\end{array}}{\simeq} 
H \otimes \Liehat_{\geq 2}(H).
$$
By degree truncation and by restriction, one obtains a map
$$
\tau^\theta_{[k,2k[}: \I(\Sigma)[k] \longrightarrow \bigoplus_{d=k}^{2k-1} H \otimes \Lie_{d+1}(H)
$$
whose degree $k$ part $\I(\Sigma)[k] \to H \otimes \Lie_{k+1}(H)$ 
is the $k$-th Johnson homomorphism \cite{Kawazumi}.
We observe two properties for this restriction of the total Johnson map. 
First, $\tau^\theta_{[k,2k[}$ is a group homomorphism (Proposition \ref{prop:Kawazumi_homomorphism}).
Second, the values of $\tau^\theta_{[k,2k[}$ can be expressed in terms of Jacobi diagrams,
provided the expansion $\theta$ is symplectic (Proposition \ref{prop:Kawazumi_diagrammatic}):
$$
\xymatrix{
{\I(\Sigma)[k]} \ar[r]^-{\tau^\theta_{[k,2k[}} \ar@{-->}[d]
& {\displaystyle \bigoplus_{d=k}^{2k-1} H \otimes \Lie_{d+1}(H)} \\
{\displaystyle \bigoplus_{d=k}^{2k-1} \T_d(H)} \ar@{>->}[ur]_-{\eta}& 
}
$$
Here, $\eta$ is the usual map that gives rise to diagrammatic descriptions 
of Milnor's $\mu$ invariants \cite{HM} and Johnson homomorphisms \cite{GL}.

The infinitesimal versions of Morita's homomorphisms are introduced and studied in \S \ref{sec:Morita}.
As already evoked, the precise definition of $m_k$  is obtained from the original definition of $M_k$ \cite{Morita} 
by replacing each bar complex of a group by the Koszul complex of its Malcev Lie algebra.
A similar passing from groups to Malcev Lie algebras already appears in Day's work \cite{Day},
where $M_k$ (with real coefficients) is extended to a crossed homomorphism on the full mapping class group of $\Sigma$. 
Whereas his construction uses methods of differential topology, our definition of $m_k$ is purely algebraic.
Yet, the two approaches should be connected since Pickel's isomorphism \cite{Pickel} 
$$
\P: H_3\left(\pi/\Gamma_{k+1}\pi;\Q\right) \stackrel{\simeq}{\longrightarrow} 
H_3\big(\MLie(\pi/\Gamma_{k+1}\pi);\Q\big)
$$
connects $M_k$ (with rational coefficients) to $m_k$, as shown in Proposition \ref{prop:original_to_infinitesimal}.
The main result of \S \ref{sec:Morita} is Theorem \ref{th:Kawazumi_to_Morita},
which asserts that $m_k$ coincides (up to a sign) 
with the degree $[k,2k[$ truncation of Kawazumi's total Johnson map $\tau^\theta$
relative to a symplectic expansion $\theta$. This relation between $m_k$ and $\tau^\theta_{[k,2k[}$
needs their diagrammatic descriptions, which are given by the maps $\Phi$ and $\eta$ respectively.
Then, we recover two properties for $M_k$ by proving them for $m_k$: 
first, $M_k$ determines the $k$-th Johnson homomorphism \cite{Morita}
and, second, its kernel coincides with the $2k$-th term of the Johnson filtration \cite{Heap}.

In \S \ref{sec:LMO}, we come back to our initial motivation
which was connecting nilpotent homotopy types of $3$-manifolds to their finite-type invariants. 
Let us recall that Le, Murakami and Ohtsuki  have constructed
in \cite{LMO} a universal finite-type invariant of homology $3$-spheres.
More recently, the LMO invariant has been extended by Cheptea, Habiro and the author 
to a functor $\Ztilde$ from a category of cobordisms to a category of diagrams \cite{CHM}.
In particular, the functor $\Ztilde$ defines a monoid homomorphism
whose source is the Torelli group and whose target is a certain algebra of Jacobi diagrams. 
This diagrammatic representation of the Torelli group is called the \emph{LMO homomorphism}.
The main result of \S \ref{sec:LMO} is a topological interpretation of its \emph{tree-reduction},
which is obtained by killing all Jacobi diagrams that are looped.
More precisely, we start by showing 
that the functor $\Ztilde$ defines a symplectic expansion of $\pi$, which we denote by $\thetaLMO$.
Then, we show that the total Johnson map relative to $\thetaLMO$ is determined by 
the tree-reduction of the LMO homomorphism (Theorem \ref{th:LMO_to_Kawazumi}) and vice versa (Theorem \ref{th:Kawazumi_to_LMO}).
Theorem \ref{th:LMO_to_Kawazumi} is inspired by the ``global formula''  of Habegger and Masbaum giving 
\emph{all} Milnor's $\mu$ invariants of a pure braid from the tree-reduction of its Kontsevich integral \cite{HM}. 
Finally, we conclude that the degree $[k,2k[$ part of the tree-reduction of the LMO homomorphism
coincides with $m_k$ through the isomorphisms
$$
\bigoplus_{d=k}^{2k-1} \T_d(H) \stackrel{\Phi}{\simeq} 
H_3\left(\frac{\Lie(H)}{\Gamma_{k+1} \Lie(H)};\Q\right)
\stackrel{\thetaLMO}{\simeq} H_3\big(\MLie(\pi/\Gamma_{k+1} \pi);\Q\big).
$$
Consequently $M_k$ splits as a sum of $k$ finite-type invariants, whose degrees range from $k$ to $2k-1$.
In the lowest degree, this decomposition of $M_k$ interprets the $k$-th Johnson homomorphism
as a finite-type invariant of degree $k$ which is  already known \cite{Habiro, GL, Habegger}.

To close this introduction, let us recall that the Torelli group of $\Sigma$ embeds 
into the monoid $\cyl(\Sigma)$ of homology cylinders over $\Sigma$.
Here a \emph{homology cylinder} over $\Sigma$ is a compact oriented $3$-manifold $C$, 
whose boundary is parametrized by an orientation-preserving homeomorphism
$c: \partial(\Sigma \times [-1,1]) \to\partial C$
in such a way that $(C,c)$ has the same homology type as $(\Sigma \times [-1,1],\Id)$.
Homology cylinders play a key role in the works of Goussarov and Habiro 
on finite-type invariants \cite{Goussarov_note,Habiro}.
All the constructions that we have previously mentioned for the group $\I(\Sigma)$ 
can be extended to the monoid $\cyl(\Sigma)$.
For example, Sakasai considers Morita's homomorphisms for homology cylinders in \cite{Sakasai}. 
Our results are proved in this general framework.\\[-0.1cm]

\noindent
\textbf{Acknowledgements.}
The author thanks Benjamin Enriquez for helpful and stimulating discussions,
and he is grateful to Alex Bene, Nariya Kawazumi and Yusuke Kuno for their comments.
This work has been finalized during a stay at the CTQM (Aarhus University) whose hospitality is gratefully acknowledged.\\

\noindent
\textbf{Conventions.}
In the sequel, all vector spaces, Lie algebras, homology groups, etc. 
are considered with rational coefficients. 
Equivalence classes are always denoted by curly brackets $\{-\}$, 
except for homology classes which are denoted by straight brackets $[-]$.
The lower central series of a group $G$ (respectively of a Lie algebra $\mathfrak{g}$) is denoted by 
$$
G=\Gamma_1 G \supset  \Gamma_2 G \supset \Gamma_3 G \supset \cdots 
\quad (\hbox{respectively } \mathfrak{g}=\Gamma_1 \mathfrak{g} \supset  
\Gamma_2 \mathfrak{g} \supset  \Gamma_3 \mathfrak{g} \supset \cdots ),
$$
and our notation for  commutators in $G$ is $[\gp{x},\gp{y}] := \gp{x}\gp{y} \gp{x}^{-1}\gp{y}^{-1}$.

\vspace{0.5cm}

\section{The third homology group of a free nilpotent Lie algebra}

\label{sec:fission}

We define in this section, for each nilpotent Lie algebra $\mathfrak{g}$,
a linear map from a space of tree-shaped $\mathfrak{g}$-colored Jacobi diagrams to $H_3(\mathfrak{g})$.
When $\mathfrak{g}$ is free nilpotent and finite-dimensional,  
we deduce from the computation of $H_3(\mathfrak{g})$  by Igusa and Orr \cite{IO}
that our diagrammatic map is an isomorphism.

\subsection{The fission map}

\label{subsec:fission}

Let $V$ be a vector space.
A \emph{Jacobi diagram} is a unitrivalent finite graph whose trivalent vertices are oriented
(\ie edges are cyclically ordered around each trivalent vertex).
The \emph{internal degree} of such a diagram is the number of trivalent vertices, and is denoted by $\ideg$.
A Jacobi diagram is said to be \emph{$V$-colored} if it comes with a map
from the set of its univalent vertices to $V$. For example, 

$$
{\labellist \small \hair 2pt
\pinlabel {$v_1$} [tr] at 0 0
\pinlabel {$v_2$} [br] at 0 59
\pinlabel {$v_3$} [b] at 69 76
\pinlabel {$v_4$} [bl] at 138 63
\pinlabel {$v_5$} [tl] at 135 0
\endlabellist}
\includegraphics[scale=0.3]{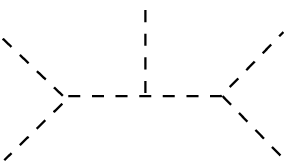}
\quad \quad \quad \quad  (\hbox{with } v_1,\dots,v_5 \in V)
$$
\vspace{-0.3cm}

\noindent
is a $V$-colored Jacobi diagram of internal degree $3$ where, 
by convention, vertex orientations are given by the trigonometric orientation of the plan.
Let
$$
\T(V) = \bigoplus_{d=1}^{+ \infty} \T_d(V) 
$$ 
be the graded vector space of Jacobi diagrams which are tree-shaped, connected and $V$-colored, 
modulo the AS, IHX and multilinearity relations: \\[0.2cm]

\begin{center}
\labellist \small \hair 2pt
\pinlabel {AS} [t] at 102 -5
\pinlabel {IHX} [t] at 543 -5
\pinlabel {multinearity} [t] at 1036 -5
\pinlabel {$= \ -$}  at 102 46
\pinlabel {$-$} at 484 46
\pinlabel {$+$} at 606 46
\pinlabel {$=0$} at 721 46 
\pinlabel {$+$} at 1106 46
\pinlabel {$=$} at 961 46
\pinlabel{$v_1+v_2$} [b] at 881 89
\pinlabel{$v_1$} [b] at 1042 89
\pinlabel{$v_2$} [b] at 1170 89
\endlabellist
\centering
\includegraphics[scale=0.32]{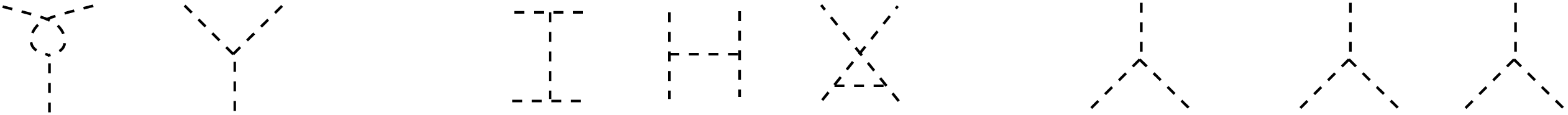}
\end{center}
\vspace{0.5cm}

\noindent
In the sequel, we assume that $V$ is a Lie algebra $\mathfrak{g}$.

Let $T$ be a connected tree-shaped $\mathfrak{g}$-colored Jacobi diagram. For each trivalent vertex $r$, 
$T$ can be seen as the union of three trees ``rooted'' at $r$, 
which are denoted by $T_r^{(1)}$, $T_r^{(2)}$ and $T_r^{(3)}$ 
so that the numbering $1,2,3$ gives the vertex-orientation around $r$.
Any connected tree-shaped Jacobi diagram $A$, 
all of whose univalent vertices are $\mathfrak{g}$-colored
apart from one which is denoted by $r$, defines an element $\comm(A)$ of $\mathfrak{g}$. For example, we have
\begin{equation}
\label{eq:tree_to_commutator}
\comm\Bigg(\begin{array}{c}
\labellist \small \hair 2pt
\pinlabel {$g_1$} [b] at 2 162
\pinlabel {$g_2$} [b] at 111 162
\pinlabel {$g_3$} [b] at 167 162
\pinlabel {$g_4$} [b] at 222 162
\pinlabel {$r$} [t] at 108 0
\endlabellist
\includegraphics[scale=0.3]{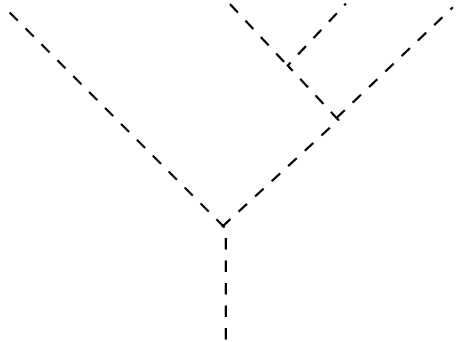}
\end{array}
\Bigg)= [g_1,[[g_2,g_3],g_4]].
\end{equation}

\noindent
Then, the \emph{fission} of $T$ is defined by
\begin{equation}
\label{eq:fission}
\phi(T) :=  \sum_{r} \comm(T_r^{(3)}) \wedge \comm(T_r^{(2)}) \wedge \comm(T_r^{(1)}) 
\ \in \Lambda^3 \mathfrak{g},
\end{equation}
where the sum is indexed by the trivalent vertices $r$ of $T$.
For example, we have 
$$
\phi\Bigg(\quad \begin{array}{c}
{\labellist \small \hair 2pt
\pinlabel {$g_1$} [tr] at 0 0
\pinlabel {$g_2$} [br] at 0 59
\pinlabel {$g_3$} [b] at 69 74
\pinlabel {$g_4$} [bl] at 138 63
\pinlabel {$g_5$} [tl] at 135 0
\endlabellist}
\includegraphics[scale=0.25]{fission} 
\end{array} \quad \Bigg) =
g_1 \wedge g_2 \wedge [g_3,[g_4,g_5]] + g_3 \wedge [g_4,g_5] \wedge [g_1,g_2] 
+ g_5 \wedge [[g_1,g_2],g_3] \wedge  g_4.
$$

We recall that the \emph{Koszul complex} of $\mathfrak{g}$ (with trivial coefficients)
is the chain complex $\left(\Lambda \mathfrak{g}, \partial\right)$ 
whose boundary operator $\partial_n: \Lambda^n \mathfrak{g} \to \Lambda^{n-1} \mathfrak{g}$ is defined by
$$
\partial_n(g_1 \wedge \cdots \wedge g_n) = \sum_{i<j} 
(-1)^{i+j}\cdot [g_i,g_j] \wedge g_1\wedge \cdots \widehat{g_i} \cdots \widehat{g_j} \cdots \wedge g_n.
$$
Its homology gives the homology of the Lie algebra $\mathfrak{g}$ (with trivial coefficients).

\begin{lemma}
Fission of tree diagrams induces a linear map
$$
\Phi: \T(\mathfrak{g}) \longrightarrow \Lambda^3 \mathfrak{g}/ \Img(\partial_4).
$$
\end{lemma}

\begin{proof}
The fission map $\phi$ can be extended by linearity to linear combinations of Jacobi diagrams.
Then, the map $\phi$ vanishes on the AS relations by the antisymmetry of the Lie bracket
of $\mathfrak{g}$ and by the antisymmetry in $\Lambda^3 \mathfrak{g}$. 
A similar argument applies to the multilinearity relations. 
Next,  a straightforward computation shows that

$$
\phi\left( \quad \begin{array}{c}
{\labellist \small \hair 2pt
\pinlabel {$g$} [r] at 0 0
\pinlabel {$h$} [r] at 0 47
\pinlabel {$k$} [l] at 63 47
\pinlabel {$l$} [l] at 63 0 
\pinlabel {$g$} [t] at 104 2
\pinlabel {$h$} [b] at 104 49
\pinlabel {$k$} [b] at 148 49
\pinlabel {$l$} [t] at 149 2
\pinlabel {$g$} [rt] at 191 0
\pinlabel {$h$} [rb] at 192 49
\pinlabel {$k$} [lb] at 241 48
\pinlabel {$l$} [lt] at 241 0
\pinlabel {$-$}  at 85 25
\pinlabel {$+$} at 173 25
\endlabellist}
\includegraphics[scale=0.7]{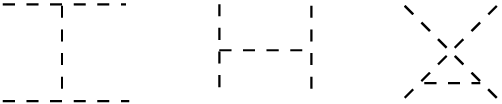} 
\end{array} \quad \right) 
= \partial_4(g\wedge h \wedge  k \wedge l).
$$
\vspace{0cm}

\noindent
It follows from this identity and the Jacobi relation in $\mathfrak{g}$,
that $\phi$ sends the IHX relations to the subspace $\Img(\partial_4)$ of $\Lambda^3 \mathfrak{g}$.
So, $\phi$ induces a linear map $\Phi:\T(\mathfrak{g}) \to \Lambda^3 \mathfrak{g}/ \Img(\partial_4)$.
\end{proof}

\begin{lemma}
\label{lem:d_3}
Let $T$ be a connected tree-shaped $\mathfrak{g}$-colored Jacobi diagram. 
Then, we have 
$$
\partial_3(\phi(T))= \sum_{v} \col(v) \wedge \comm(T_v) \ \in \Lambda^2 \mathfrak{g}
$$
where the sum is over all univalent vertices $v$ of $T$ 
with color $\col(v)$, and where $T_v$ is the tree $T$ ``rooted'' at $v$.
\end{lemma}

\begin{proof}
We proceed by induction on the internal degree $d$ of $T$. 
If $d=1$, then $T$ is a $Y$-shaped diagram whose univalent vertices 
are denoted by $v_1,v_2,v_3$ in accordance with the cyclic order:
\begin{eqnarray*}
&&\partial_3(\phi(T))\\
&=& \partial_3( \col(v_3) \wedge \col(v_2) \wedge \col(v_1)) \\
&=& -[\col(v_3),\col(v_2)] \wedge \col(v_1) + [\col(v_3),\col(v_1)] \wedge \col(v_2)
- [\col(v_2),\col(v_1)] \wedge \col(v_3)\\
&=&  \col(v_1) \wedge \comm(T_{v_1}) + \col(v_2) \wedge \comm(T_{v_2}) + \col(v_3) \wedge \comm(T_{v_3}).
\end{eqnarray*}
Assume now that $T$ has internal degree $d+1$, 
and let $r$ be a trivalent vertex of $T$ which is adjacent to two univalent vertices.
Then, $T$ is the union of a tree $A$ ``rooted'' at $r$ with two ``radicals'' colored by $g$ and $h$ respectively: 
$$
{\labellist \small \hair 2pt
\pinlabel {$T=\quad $} [r] at 0 45
\pinlabel {$A$}  at  45 45
\pinlabel {$r$} [br] at 148 51
\pinlabel {$g$} [bl] at 198 92 
\pinlabel {$h$} [lt] at 198 0
\endlabellist}
\includegraphics[scale=0.5]{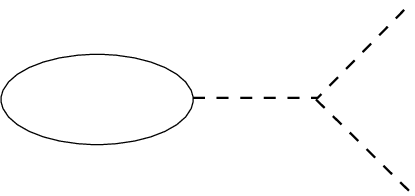}
$$
We denote by $A'$ the $\mathfrak{g}$-colored Jacobi diagram obtained from $A$ 
by coloring its root $r$ with $[g,h]$. Then, we have
\begin{eqnarray*}
\partial_3\left(\phi(T)\right)&=& \partial_3\left( \Phi(A') +  g \wedge h \wedge \comm(A)\right)\\
&=&  \sum_{a\neq r} \col(a) \wedge \comm(A'_{a}) + [g,h]\wedge \comm(A) +\\
&& \left( -[g,h]\wedge \comm(A) + [g,\comm(A)]\wedge h  - [h,\comm(A)]\wedge g \right).\\
&=& \sum_{a\neq r} \col(a) \wedge \comm(A'_{a}) + g \wedge [h,\comm(A)] + h \wedge [\comm(A),g],
\end{eqnarray*}
where the sums range over all univalent vertices $a\neq r$ of the tree $A$. This proves the inductive step.
\end{proof}

\begin{proposition}
\label{prop:fission}
Assume that the Lie algebra $\mathfrak{g}$ is nilpotent of class $k$, i.e$.$ $\Gamma_{k+1} \mathfrak{g}=\{0\}$.
Then, the map $\Phi$ restricts to a linear map
$$
\Phi:\bigoplus_{d=k}^{+ \infty} \T_d( \mathfrak{g}) \longrightarrow H_3(\mathfrak{g}).
$$
Moreover, this map is trivial in degree $d \geq 2k$.
\end{proposition}

\begin{proof}
Let $T$ be a connected $\mathfrak{g}$-colored tree-shaped Jacobi diagram of internal degree $d$.
Then, for each univalent vertex $v$ of $T$, the rooted tree $T_v$ has $d+1$ leaves, 
so that the length of the commutator $\comm(T_d)$  is $(d+1)$. 
So, Lemma \ref{lem:d_3} implies that $\phi(T)$ is a $3$-cycle if $d \geq k$.

Assume now that $d \geq 2k$. By the IHX relation, 
we can assume that each trivalent vertex $r$ of $T$ is adjacent to, at least, one univalent vertex.
Let $T_r^{(1)}$, $T_r^{(2)}$ and $T_r^{(3)}$ be the three subtrees of $T$ rooted at $r$. 
There exists $j\in \{1,2,3\}$ such that $T_r^{(j)}$ has at least $k$ trivalent vertices, 
so that $\comm(T_r^{(j)})$ belongs to $\Gamma_{k+1} \mathfrak{g}=\{0\}$.
Then, formula (\ref{eq:fission}) shows that the $3$-chain $\phi(T)$ is trivial.
\end{proof}

\subsection{The case of a free nilpotent Lie algebra}

\label{subsec:free_nilpotent_Lie}

Let $H$ be a finite-dimensional vector space, and let $\Lie(H)$ be the free Lie algebra generated by $H$.
The length of commutators defines a grading on $\Lie(H)$: 
$$
\Lie(H) = \bigoplus_{k=1}^{+ \infty} \Lie_k(H).
$$
To simplify the notation, we will often write $\Lie$ instead of $\Lie(H)$.
Since $\Lie_{\geq k+1}$ coincides with $\Gamma_{k+1}\Lie$,
the Lie algebra $\Lie/\Lie_{\geq k+1}$ is the free nilpotent Lie algebra generated by $H$ of nilpotency class $k$. 
Being graded, its Koszul complex $\Lambda(\Lie/\Lie_{\geq k+1})$ has a grading, and so has its homology:
$$
H_*\left(\Lie/\Lie_{\geq k+1}\right) 
= \bigoplus_{d=1}^{+ \infty} H_{*}\left(\Lie/\Lie_{\geq k+1}\right)_d.
$$
Those homology groups can be computed for low homological degrees as follows.
First of all, the isomorphism
$$
H \stackrel{\simeq}{\longrightarrow} H_1\left(\Lie/\Lie_{\geq k+1}\right), \
h \longmapsto [\{h\}]
$$
shows that $H_1\left(\Lie/\Lie_{\geq k+1}\right)$ is concentrated in degree $1$.
Next, by Hopf's theorem, we have an isomorphism
\begin{equation}
\label{eq:Hopf_theorem}
\Lie_{k+1} \stackrel{\simeq}\longrightarrow H_2\left(\Lie/\Lie_{\geq k+1}\right), \
[a,b] \longmapsto [ \{a\} \wedge \{b\} ]
\end{equation}
where $a$ and $b$ belong to $\Lie_i$ and $\Lie_j$, respectively, for some $i,j$ such that $i+j=k+1$. 
Thus, $H_2\left(\Lie/\Lie_{\geq k+1}\right)$ is concentrated in degree $k+1$.

The computation of $H_3\left(\Lie/\Lie_{\geq k+1}\right)$ is done by Igusa and Orr in \cite[\S 5]{IO}.
They apply the Hochschild--Serre spectral sequence to the central extension of graded Lie algebras
\begin{equation}
\label{eq:extension}
0 \to \Lie_{k} \to \Lie/\Lie_{\geq k+1} \to  \Lie/\Lie_{\geq k} \to 1,
\end{equation}
which gives
\begin{equation}
\label{eq:HS}
E^r_{p,q} \stackrel{r\to + \infty}{\longrightarrow} H_{p+q}(\Lie/\Lie_{\geq k+1} )
\quad \hbox{where} \quad
E^2_{p,q}  \simeq H_p\left(\Lie/\Lie_{\geq k} \right) \otimes \Lambda^q  \Lie_{k}.
\end{equation}
Their result can be summarized as follows in terms of the kernel of the bracket map
$$
\D_k(H) := \Ker\left([-,-]:H\otimes \Lie_{k-1}(H) \longrightarrow \Lie_{k}(H) \right),
$$
which we also simply denote by $\D_k$.

\begin{theorem}[Igusa--Orr]
\label{th:Igusa-Orr}
There is an isomorphism of graded vector spaces
$$
\IO: H_3(\Lie/\Lie_{\geq k+1}) \stackrel{\simeq}{\longrightarrow}
\bigoplus_{d=k+2}^{2k+1} \D_d
$$
such that, for all $m\geq k$, the diagram 
$$
\xymatrix{
{H_3(\Lie/\Lie_{\geq m+1}) } \ar[rr]^-{\IO}_-\simeq \ar[d]_-{\can} & &{\bigoplus_{d=m+2}^{2m+1} \D_d} \ar[d]^-{\can}\\
{H_3(\Lie/\Lie_{\geq k+1}) } \ar[rr]_-{\IO}^-\simeq & &{\bigoplus_{d=k+2}^{2k+1} \D_d}
}
$$
commutes, and such that the composition
$$
\xymatrix{
{H_3(\Lie/\Lie_{\geq k})} \ar@{->>}[rr]^-\can & & {H_{3}(\Lie/\Lie_{\geq k})_{k+1}}
\ar[rr]^-{\IO_{k+1}}_-\simeq & & {\D_{k+1} \subset H \otimes \Lie_{k}}
}
$$
coincides with the differential $d^2_{3,0}:{E^2_{3,0}} \to {E^2_{1,1}}$
of the spectral sequence (\ref{eq:HS}).
\end{theorem}

Let us recall that the space $\D_k$ can be described in terms of tree diagrams,
which leads to diagrammatic descriptions for Milnor's $\mu$ invariants of string links \cite{HM} 
or, similarly, for Johnson homomorphisms of homology cylinders \cite{GL}. 
Indeed, there is a linear map
$$
\eta_k: \T_k(H) \stackrel{\simeq}{\longrightarrow}  \D_{k+2}(H) \subset H\otimes \Lie_{k+1}(H)
$$
defined, for all connected tree-shaped $H$-colored Jacobi diagrams $T$, by
\begin{equation}
\label{eq:eta}
\eta_k(T) := \sum_{v} \col(v) \otimes \comm(T_v)
\end{equation}
where the sum is over all univalent vertices $v$ of $T$. 
See \cite{Levine_addendum} for more details.
Thus, one can expect from Theorem \ref{th:Igusa-Orr} a diagrammatic description of  $H_3(\Lie/\Lie_{\geq k+1})$,
and this is achieved by the fission map.

\begin{theorem}
\label{th:fission_iso}
Fission of tree diagrams defines a linear isomorphism
$$
\Phi: \bigoplus_{d=k}^{2k-1} \T_d(H) \stackrel{\simeq}{\longrightarrow} H_3(\Lie/\Lie_{\geq k+1})
$$
which shifts the degree by $+2$.
\end{theorem}
\begin{proof}
Since $\T(H)$ embeds into $\T(\Lie/\Lie_{\geq k+1})$, we can take the map $\Phi$  given by Proposition \ref{prop:fission}.
It shifts the degree by $+2$ because a connected tree-shaped Jacobi diagram 
of internal degree $d$ has $d+2$ univalent vertices.
\begin{claim}
\label{claim:Phi_eta}
For all  $m\geq 1$, the following diagram commutes
$$
\xymatrix{
{\T_m(H)} \ar[r]^-{\Phi_m} \ar[rd]_-{\eta_m}^-\simeq & H_{3}(\Lie/\Lie_{\geq m+1})_{m+2}\ar[d]^-{\IO_{m+2}}_-\simeq \\
& {\D_{m+2}}
}
$$
\end{claim}

\noindent
It follows from Claim  \ref{claim:Phi_eta} that  the map $\Phi$ is an isomorphism in the lowest degree. 
Then, for all $d=k,\dots, 2k-1$, the commutative diagram
$$
\xymatrix{
{H_{3}(\Lie/\Lie_{\geq k+1})_{d+2} } & {H_{3}(\Lie/\Lie_{\geq k+2})_{d+2}  } \ar[l] &
{\cdots}  \ar[l]  & {H_{3}(\Lie/\Lie_{\geq d+1})_{d+2} , } \ar[l] \\
&&&\\
\T_{d}(H) \ar[uu]^-{\Phi_d}  \ar[ruu]^-{\Phi_d}  \ar[rrruu]^-{\Phi_d}_-\simeq  &&&
}
$$
(whose horizontal maps are isomorphisms by Theorem \ref{th:Igusa-Orr})
shows that the map $\Phi$ is bijective in degree $d$. 
Thus, it is enough to prove Claim \ref{claim:Phi_eta}, 
i.e$.$ to check the commutativity of the diagram
\begin{equation}
\label{eq:eta_Phi}
\xymatrix{
{\T_m(H)} \ar[r]^-{\Phi_m} \ar[rd]_-{\eta_m} & H_{3}(\Lie/\Lie_{\geq m+1}) \ar[d]^-{d^2_{3,0}} \\
& {H\otimes \Lie_{m+1}.}
}
\end{equation}

For this, we recall that the Hochschild--Serre spectral sequence associated to a central extension of Lie algebras
$$
0 \to \mathfrak{h} \to \mathfrak{g} \to \mathfrak{g}/\mathfrak{h} \to 1
$$
is the spectral sequence associated to the chain complex $C:= (\Lambda \mathfrak{g},\partial)$ filtered by 
$$
 C_n = \Filt_n C_n \supset \Filt_{n-1} C_n \supset \cdots \supset \Filt_0 C_n \supset \Filt_{-1} C_n =\{0\}
$$
where
$$
\Filt_p C_n := \Img\left( \mathfrak{h}^{\otimes(n-p)} \otimes \mathfrak{g}^{\otimes p} 
\to \mathfrak{g}^{\otimes n} \to \Lambda^n \mathfrak{g} \right) \subset \Lambda^n \mathfrak{g} =C_n.
$$
At the second stage of this spectral sequence, we have
$$
E^2_{p,q} = \frac{\Filt_p C_{p+q} \cap d^{-1}(\Filt_{p-2}C_{p+q-1})}
{\Filt_p C_{p+q} \cap d(\Filt_{p+1} C_{p+q+1}) + \Filt_{p-1} C_{p+q} \cap d^{-1}( \Filt_{p-2} C_{p+q-1})},
$$
the differential $d^2_{pq}: E^2_{p,q} \to E^2_{p-2,q+1}$ 
is induced by the boundary operator $\partial_{p+q}$ of $C$ and there is an isomorphism
$$
H_p(\mathfrak{g}/\mathfrak{h}) \otimes \Lambda^q \mathfrak{h} 
\stackrel{\simeq}{\longrightarrow} E^2_{p,q}
$$
defined by $[\{x\}] \otimes y \mapsto \{ x \wedge y \}$ for all $y\in \Lambda^q \mathfrak{h}$
and  $x \in \Lambda^p \mathfrak{g}$ such that $\partial_p(x) \equiv 0$ mod $\mathfrak{h}$.
 
We now take $\mathfrak{g} := \Lie/\Lie_{\geq m+2}$, 
$\mathfrak{h} := \Lie_{m+1}$, $p=3$ and $q=0$. 
Let $T$ be a tree-shaped connected $H$-colored Jacobi diagram of internal degree $m$. 
Let $\phi(T) \in \Lambda^3 \Lie$ be the $3$-chain for the free Lie algebra $\Lie$ obtained from $T$ by fission. 
Then, we have
$$
\Phi_m(T) = [ \{ \phi(T) \}] \in  H_{3}(\Lie/\Lie_{\geq m+1})
$$
where the inner $\{-\}$ denotes the reduction $\Lambda^3 \Lie \to \Lambda^3(\Lie/\Lie_{\geq m+1})$.
So, $d^2_{3,0}  \Phi_m (T)\in E^2_{1,1}$ is the class of $\partial_3\phi(T) \in \Lambda^2 \Lie$,
for which Lemma \ref{lem:d_3} gives an explicit formula.
By comparing this formula with the definition (\ref{eq:eta}) of $\eta_m$, we see that
$d^2_{3,0}  \Phi_m (T)\in E^2_{1,1} \simeq H \otimes \Lie_{m+1}$ coincides with $\eta_m(T)$. 
\end{proof}

\begin{remark}
Let $H$ be a finitely-generated free $\Z$-module, and let $\Lie(H)$ be the free Lie ring generated by $H$.
Igusa and Orr's result (Theorem \ref{th:Igusa-Orr}) is also valid for integer coefficients \cite{IO} and,
besides, the map
$$
\Phi: \bigoplus_{d=k}^{2k-1} \T_d(H) \stackrel{\simeq}{\longrightarrow} H_3(\Lie(H)/\Lie_{\geq k+1}(H)),
$$
is defined with integer coefficients as well. 
However, $\Phi$ is not bijective since $\eta$ is not\footnote{ 
According to Levine \cite{Levine_addendum,Levine_quasi-Lie}, 
$\eta_{2k}$ is not surjective and $\eta_{2k+1}$ is not injective.} an isomorphism over $\Z$.
\end{remark}

\vspace{0.5cm}

\section{Expansions of the free group}

\label{sec:expansions}

In this section, we review expansions of the free group \cite{Lin,Kawazumi}
and we focus on ``group-like'' expansions which are, essentially,
identifications between the Malcev Lie algebra of the free group and the complete free Lie algebra. 
Finally, we introduce ``symplectic'' expansions of the fundamental group 
of a compact connected oriented surface, with one boundary component.

\subsection{Review of the Malcev Lie algebra of a group}

\label{subsec:Malcev}

Let $G$ be a group. As shown by Jennings \cite{Jennings} and Quillen \cite{Quillen_rht},
the Malcev completion and the Malcev Lie algebra of $G$ 
can be constructed from its group algebra $\Q[G]$. 
The reader is referred to \cite{Quillen_rht} for full details of their construction, which is only outlined below.
We denote by $I$ the augmentation ideal of $\Q[G]$. The $I$-adic completion of $\Q[G]$
$$
\widehat{\Q}[G] := \varprojlim_{k} \Q[G]/ I^k
$$
equipped with the filtration 
$$
\widehat{I^j} := \varprojlim_{k\geq j} I^j/ I^k, \quad \forall j\geq 0
$$
is a complete Hopf algebra in the sense of Quillen \cite{Quillen_rht}. 
Let $\widehat \Delta$ be the coproduct. 

\begin{definition}
The \emph{Malcev completion} of $G$ is the group of group-like elements of $\widehat{\Q}[G]$
$$
\MGroup(G) := \GLike(\widehat{\Q}[G]) = 
\left\{x \in \widehat{\Q}[G]: \widehat\Delta(x)= x \widehat{\otimes} x, x\neq 0 \right\}
$$
equipped with the filtration
$$
\widehat{\Gamma}_j \MGroup(G)  
:= \MGroup(G) \cap \left(1+ \widehat{I^j}\right), \quad  \forall j\geq 1.
$$
The \emph{Malcev Lie algebra} of $G$ is the Lie algebra of primitive elements of $\widehat{\Q}[G]$
$$
\MLie(G) :=  \Prim(\widehat{\Q}[G]) = 
\left\{x \in \widehat{\Q}[G]: \widehat\Delta(x)= x \widehat{\otimes} 1 + 1 \widehat{\otimes} x \right\}
$$
equipped with the filtration
$$
\widehat{\Gamma}_j \MLie(G)  
:= \MLie(G) \cap \widehat{I^j}, \quad \forall  j\geq 1.
$$
\end{definition}

\begin{remark}
Our notation for the filtration of $\MLie(G)$ is justified by the fact
that, for all $j\geq 1$, the $j$-th term of this filtration is the closure of $\Gamma_j \MLie(G)$ 
for the topology that it defines. A similar remark applies to $\MGroup(G)$.
\end{remark}

The Malcev completion and the Malcev Lie algebra of a group $G$ 
are equivalent objects derived from $G$. Indeed, as a general fact in a complete Hopf algebra, 
the primitive and the group-like elements 
are in one-to-one correspondence via the exponential and logarithmic series:
$$
\MGroup(G) \subset 1 + \widehat{I} 
\overset{\log}{\underset{\exp}{\longrightleftarrows}} \widehat{I} \supset \MLie(G).
$$
The inclusion $G \subset \Q[G]$ induces a canonical map
$\iota: G \longrightarrow \MGroup(G)$. 
It is injective if and only if $G$ is residually torsion-free nilpotent
or, equivalently, if and only if the rational lower central series of $G$ has a trivial intersection:
$$
\bigcap_{k\geq 1} \big\{ \gp{g} \in G : \exists n\geq 1, \gp{g}^n \in \Gamma_k G \big\} = \{1\}.
$$
In such a case, we will omit the map $\iota$ to simplify notations. 
For example, free groups and free nilpotent groups are residually torsion-free nilpotent.

Classically, the ``Malcev completion'' of a nilpotent group $N$ 
refers to its uniquely-divisible closure. 
It has been proved by Jennings in the finitely generated case \cite{Jennings}
and by Quillen in general \cite{Quillen_rht} that $\MGroup(N)$ is a realization of this closure.
To be more specific, let us recall that a \emph{uniquely-divisible closure} 
of a nilpotent group $N$ is a pair $(D,i)$, where 
\begin{itemize}
\item $D$ is nilpotent and is uniquely-divisible: 
$\forall \gp{y}\in D, \forall k\geq 1, \exists ! \gp{x}\in D, \gp{x}^k =\gp{y}$,
\item $i:N\to D$ is a group homomorphism whose kernel is the torsion subgroup of $N$,
\item $\forall \gp{x}\in D, \exists k\geq 1, \gp{x}^k \in i(N)$.
\end{itemize}
Malcev proved that the uniquely-divisible closure of a nilpotent group $N$
always exists and is essentially unique. (See \cite{KM} for instance.) 
It is usually denoted by $N\otimes \Q$.

\begin{theorem}[Jennings, Quillen]
\label{th:divisible_closure}
The canonical map $\iota:G \to \MGroup(G)$ induces a group isomorphism
$$
\iota: \varprojlim_{k} \left( \left(G/\Gamma_k G\right) \otimes \Q \right)
\stackrel{\simeq}{\longrightarrow} \MGroup(G).
$$
\end{theorem}

The Malcev Lie algebra of a group $G$ being canonically filtered, 
there is a graded Lie algebra $\Gr \MLie(G)$ associated to it.
This graded Lie algebra has been identified by Quillen.

\begin{theorem}[Quillen]
\label{th:Quillen}
The map $\log \iota: G \to \MLie(G)$ preserves the filtrations 
($G$ being filtered by the lower central series),
and it induces a graded Lie algebra isomorphism:
$$
(\Gr \log \iota)\otimes \Q:  \Gr G \otimes \Q \stackrel{\simeq}{\longrightarrow} \Gr \MLie(G).
$$
\end{theorem}

\noindent
The first statement is an  application of the Baker--Campbell--Hausdorff formula,
and the second statement follows from \cite[Theorem 2.14]{Quillen_rht} 
and  the main result of \cite{Quillen_graded}.

\subsection{Group-like expansions}

Let $F$ be a finitely-generated free group, and let $H$ be the abelianization of $F$ with rational coefficients:
$$
H:= \left( F/ \Gamma_2 F\right) \otimes \Q.
$$
We denote by $\Tens(H)$ the tensor algebra of $H$, and by $\Tenshat(H)$ its degree completion.
If one forgets its addition, one can regard $\Tenshat(H)$ just as a monoid.

\begin{definition}
An \emph{expansion} of the free group $F$ is a monoid map
$\theta: F \to \Tenshat(H)$ such that 
$\theta(\gp{x}) = 1 + \{\gp{x}\} + (\deg \geq 2)$ for all $\gp{x} \in F$.
\end{definition}

Expansions have been studied by Lin in his work on Milnor's $\mu$ invariants \cite{Lin}, 
and by Kawazumi in his study of Johnson homomorphisms \cite{Kawazumi}.

\begin{example}
\label{ex:Magnus}
Assume that $F$ comes with a preferred basis $\gp{b}=(\gp{b}_1,\dots,\gp{b}_n)$.
The \emph{Magnus expansion} of $F$ \emph{relative} to $\gp{b}$ is the unique expansion defined by
$$
\forall i=1,\dots,n, \quad \theta^{\Z}_\gp{b}(\gp{b}_i) := 1 + \{\gp{b}_i\}.
$$
This expansion plays an important role in combinatorial group theory and low-dimensional topology,
and it has the peculiarity to exist with integer coefficients.
\end{example}

It is well-known that there is a canonical graded algebra isomorphism
$$
 \Gr \widehat{\Q}[F] = \bigoplus_{k\geq 0} I^k/I^{k+1} \stackrel{\simeq}{\longrightarrow}  \Tens(H)
$$
defined by $I/I^2 \ni \{\gp{x} -1\} \mapsto \{\gp{x}\} \in H$:
See, for example, \cite{Bourbaki} or \cite{MKS}. 
Thus, we can identify $\Gr \widehat{\Q}[F]$ with $\Tens(H) =\Gr \Tenshat(H)$.

\begin{proposition}[Lin, Kawazumi]
\label{prop:Lin-Kawazumi}
An expansion $\theta$ of $F$  extends to a unique filtered algebra isomorphism
$$
\theta: \widehat{\Q}[F] \stackrel{\simeq}{\longrightarrow} \Tenshat(H)
$$ 
which is the identity at the graded level. 
Conversely, any such isomorphism $\theta$ restricts to an expansion $\theta: F \to \Tenshat(H)$.
\end{proposition}

\noindent
It follows that, for any two expansions $\theta$ and $\theta'$ of $F$, 
there exists a unique filtered algebra automorphism $\psi: \Tenshat(H) \to  \Tenshat(H)$ 
inducing the identity at the graded level and such that $\psi \circ \theta = \theta'$ \cite{Kawazumi,Lin}.  

\begin{proof}[Proof of Proposition \ref{prop:Lin-Kawazumi}]
The monoid homomorphism $\theta: F \to \Tenshat(H)$ induces a unique algebra map
$\theta: \Q[F] \to \Tenshat(H)$ which sends the augmentation ideal $I$ to the ideal $\Tenshat_{\geq 1}(H)$ 
and, so, preserves the filtrations. Hence we have a filtered algebra homomorphism
$$
\theta: \widehat{\Q}[F] \longrightarrow \Tenshat(H).
$$
Clearly, $\Gr \theta$ is the identity in degree $1$ and, so, is the identity in any degree. 
By completeness, it follows that $\theta: \widehat{\Q}[F] \to \Tenshat(H)$ is an isomorphism. 

Conversely, given a  filtered algebra isomorphism $\theta: \widehat{\Q}[F] \to \Tenshat(H)$ 
that induces the identity at the graded level,
we get an expansion $\theta: F \to \Tenshat(H)$ by composition 
with the canonical monomorphism $\iota: F \to  \widehat{\Q}[F]$.
\end{proof}

Among expansions of the free group $F$, we prefer those with the following property.

\begin{definition} 
An expansion $\theta:F \to \Tenshat(H)$ is \emph{group-like} 
if it takes values in the group of group-like elements of $\Tenshat(H)$.
\end{definition}

\noindent
For instance, the Magnus expansion from Example \ref{ex:Magnus} is not group-like.

\begin{example}
\label{ex:basis}
Let $\gp{b}=(\gp{b}_1,\dots,\gp{b}_n)$ be a basis of $F$.
The \emph{group-like expansion} of $F$ \emph{relative} to $\gp{b}$ is the unique expansion defined by
$$
\forall i=1,\dots,n, \quad \theta_\gp{b}(\gp{b}_i) :=  \exp(\{\gp{b}_i \}).
$$
\end{example}

The following analogue of Proposition \ref{prop:Lin-Kawazumi} is proved along the same lines.

\begin{proposition}
\label{prop:group-like}
A group-like expansion $\theta$ of $F$ extends to a unique complete Hopf algebra isomorphism
$$
\theta: \widehat{\Q}[F] \stackrel{\simeq}{\longrightarrow} \Tenshat(H)
$$ 
which is the identity at the graded level. 
Conversely, any such isomorphism $\theta$ restricts to a group-like expansion $\theta: F \to \Tenshat(H)$.
\end{proposition}

\noindent
Let us restate this characterization of group-like expansions in terms of Lie algebras.
For this, we need the canonical isomorphisms
$$
\xymatrix{
{\Lie(H)} \ar@{-->}[rr] \ar[rd]_-\simeq & &{\Gr \MLie(F)}. \\
&{ \Gr F \otimes \Q} \ar[ru]_-\simeq & 
}
$$
The left-hand isomorphism is defined to be the identity in degree $1$ \cite{Bourbaki,MKS}, 
while the right-hand isomorphism is given by Theorem \ref{th:Quillen}.
Thus, we can identify $\Gr \MLie(F)$ with $\Lie(H)=\Gr \Liehat(H)$.

\begin{corollary}
\label{cor:Lie}
A group-like expansion $\theta$ of $F$ induces a unique filtered Lie algebra isomorphism
$$
\theta: \MLie(F) \stackrel{\simeq}{\longrightarrow} \Liehat(H)
$$ 
which is the identity at the graded level. 
Conversely, any such isomorphism $\theta$ induces a unique group-like expansion $\theta: F \to \Tenshat(H)$.
\end{corollary}

\begin{proof}
This follows from Proposition \ref{prop:group-like} since we have
$$
\MLie(F) = \Prim( \widehat{\Q}[F]) \quad \hbox{and} \quad
\Liehat(H) = \Prim (\Tenshat(H))
$$
and since we have, conversely,
$$
\widehat{\Q}[F] = \Uhat(\MLie(F))  \quad \hbox{and} \quad
\Tenshat(H) = \Uhat(\Liehat(H)).
$$
Here,  for $\mathfrak{g}$ a Lie algebra, $\Uhat(\mathfrak{g})$ denotes the $J$-adic completion 
of the universal enveloping algebra  of $\mathfrak{g}$, 
where $J$ is the ideal of  $\U(\mathfrak{g})$ generated by $\mathfrak{g}$.
\end{proof}

Group-like expansions are useful to compute the Malcev Lie algebra of a finitely generated group. 
Indeed, we have the following statement. 

\begin{theorem}
\label{th:finitely_presented_group}
Let $R$ be a normal  subgroup of $F$ and let $\theta$ be a group-like expansion of $F$.
We denote by $\langle \langle \log \theta(R) \rangle \rangle$
the closed ideal of $\Liehat(H)$ generated by $\log \theta(R)$.
Then, $\theta:\MLie(F) \to \Liehat(H)$ 
induces a unique filtered Lie algebra isomorphism
$$
\theta: \MLie(F/R) \stackrel{\simeq}{\longrightarrow} \Liehat(H)/ \langle \langle \log \theta(R) \rangle \rangle.
$$
\end{theorem}

\noindent
In particular, if $R$ is normally generated by some $\gp{r}_1,\dots,\gp{r}_s$,
we can see from the Baker--Campbell--Hausdorff formula that
$$
\langle \langle \log \theta(R) \rangle \rangle = 
\langle \langle \log\theta(\gp{r}_1),\dots,\log \theta(\gp{r}_s)  \rangle \rangle.
$$
Thus, Theorem \ref{th:finitely_presented_group} 
is a recipe to compute the Malcev Lie algebra of a finitely presented group.
This is well-known for the group-like  expansion $\theta$ relative to a basis of $F$ \cite{Papadima}.

\begin{proof}[Proof of Theorem \ref{th:finitely_presented_group}]
The canonical projection $F \to F/R$ induces a filtered Lie algebra homomorphism
$\MLie(F) \to \MLie(F/R)$ which, obviously, vanishes on $\log(\gp{r})$ for all $\gp{r}\in R$.
So, we have a filtered Lie algebra map
$$
p:\MLie(F)/ \langle \langle \log(R) \rangle \rangle \longrightarrow \MLie(F/R).
$$
\begin{claim}
\label{claim:p}
$p$ is an isomorphism at the graded level.
\end{claim}
\noindent
By completeness, it follows that $p$ is an isomorphism.
Besides, the filtered Lie algebra isomorphism $\theta: \MLie(F) \to \Liehat(H)$
given by Corollary \ref{cor:Lie} induces
$$
\overline{\theta}: \MLie(F)/ \langle \langle \log(R) \rangle \rangle
\longrightarrow \Liehat(H)/ \langle \langle \log \theta(R) \rangle \rangle.
$$
Then, the composition $\overline{\theta} \circ p^{-1}: \MLie(F/R) \to 
\Liehat(H)/ \langle \langle \log \theta(R) \rangle \rangle $
is our filtered Lie algebra isomorphism $\theta$.

To prove Claim \ref{claim:p}, we look at the following diagram:
$$
\xymatrix{
{\Gr \MLie(F)} \ar@{->>}[r] & 
{\Gr\left(\MLie(F)/\langle \langle \log(R) \rangle \rangle\right)} \ar[r]^-{\Gr p} & 
{\Gr \MLie(F/R)}\\
{\Gr F \otimes \Q} \ar@{->>}[rr] \ar[u]^-\simeq & & 
{\Gr(F/R) \otimes \Q.} \ar[u]_-\simeq \ar[lu]
}
$$
Here, the vertical maps are given by Theorem \ref{th:Quillen}, 
the horizontal maps are induced by canonical projections and the diagonal map
$$
\bigoplus_{n\geq 1} \frac{\Gamma_n F \cdot R}{\Gamma_{n+1} F \cdot R} \otimes \Q 
= \Gr \frac{F}{R} \otimes \Q \longrightarrow 
\Gr\frac{\MLie(F)}{\langle \langle \log(R) \rangle \rangle}=
\bigoplus_{n\geq 1}\frac{\widehat{\Gamma}_{n} \MLie(F) + \langle \langle \log(R) \rangle \rangle}
{\widehat{\Gamma}_{n+1} \MLie(F) + \langle \langle \log(R) \rangle \rangle}
$$
is defined by $\{\gp{x}\} \otimes 1 \mapsto \{\log(\gp{x})\}$ 
for all $\gp{x} \in \Gamma_n F \cdot R$. 
We deduce from that diagram that $\Gr p$ is an isomorphism.
\end{proof}

We can deduce a nilpotent version of Corollary \ref{cor:Lie}.

\begin{corollary}
\label{cor:Lie_nilpotent}
Let $m\geq 1$ and let $\theta$ be a group-like expansion of $F$.
Then, the isomorphism $\theta: \MLie(F) \to \Liehat(H)$ 
induces a unique filtered Lie algebra isomorphism
$$
\theta: \MLie(F/\Gamma_{m+1} F) \stackrel{\simeq}{\longrightarrow} \Lie(H)/\Lie_{\geq m+1}(H).
$$ 
\end{corollary}

\begin{proof}
Let $(\gp{b}_1,\dots,\gp{b}_n)$ be a basis of $F$. By Example \ref{ex:basis}, 
$\MLie(F)$ is the complete free Lie algebra generated by $({b}_1,\dots,{b}_n)$ where $b_i:= \log(\gp{b}_i)$.
The Baker--Campbell--Hausdorff formula  implies that, for all $r\geq 1$ and
for all $i_1,\dots,i_r \in \{1,\dots,n\}$, 
$$
\log\left([\gp{b}_{i_1}, [\gp{b}_{i_2},[ \dots, \gp{b}_{i_r}] \cdots ]] \right) \equiv
[b_{i_1}, [b_{i_2},[ \dots, b_{i_r}] \cdots ]]
\mod \widehat{\Gamma}_{r+1} \MLie(F).
$$
We deduce that 
$\langle \langle \log(\Gamma_{m+1}F) \rangle\rangle = \widehat{\Gamma}_{m+1} \MLie(F)$
and we conclude by applying Theorem \ref{th:finitely_presented_group} to $R=\Gamma_{m+1}F$.
\end{proof}

\subsection{Symplectic expansions}

Let $\Sigma$ be a compact connected oriented surface of genus $g$ with one boundary component.
The fundamental group of $\Sigma$ relative to a point $* \in \partial \Sigma$
$$
\pi := \pi_1(\Sigma,*)
$$
is a free group of rank $2g$. The oriented boundary curve defines a special element 
$$
\zeta \in \pi.
$$
The first homology group of $\Sigma$ with rational coefficients
$$
H:=H_1(\Sigma)
$$
is a vector space of dimension $2g$. The intersection pairing of $\Sigma$ defines
a symplectic form $\omega$ on $H$ and, so, it gives a duality
$H \stackrel{\simeq}{\to} H^*$ defined by $h \mapsto \omega(h,-)$.
The bivector dual to $\omega\in \Lambda^2H^*$ is still denoted by
$$
\omega \in \Lambda^2 H \simeq \Lie_2(H).
$$

Every system of meridians and parallels $(\alpha_1, \dots, \alpha_g, \beta_1,\dots,\beta_g)$
on the surface $\Sigma$, as shown on Figure \ref{fig:surface}, 
defines a basis $(\gp{a}_1,\dots, \gp{a}_g,\gp{b}_1, \dots,\gp{b}_g)$ of $\pi$,
as well as a basis $(a_1,\dots,a_g,b_1,\dots,b_g)$ of $H$. 
In terms of these basis, $\zeta^{-1}$ and $\omega$ write
$$
\zeta^{-1} = \prod_{i=1}^g \left[\gp{b}_i^{-1},\gp{a}_i\right] \ \in \pi
\quad \quad \hbox{and} \quad \quad 
\omega = \sum_{i=1}^g [a_i, b_i]  \in \Lie_2(H).
$$

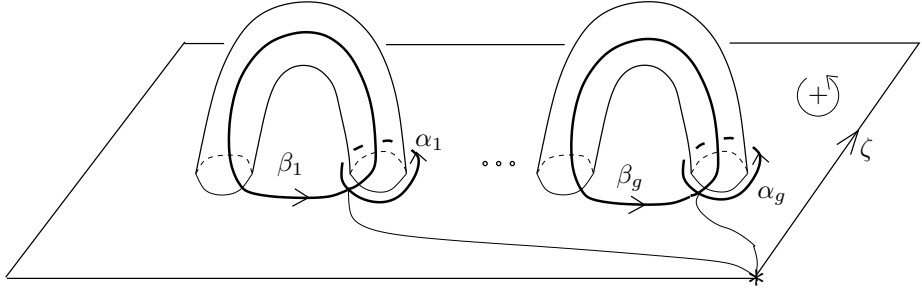
\begin{figure}[h]
\begin{center}
{\labellist \small \hair 0pt 
\pinlabel {\Large $*$} at 567 3
\pinlabel {$\zeta$} [l] at 645 100
\pinlabel {$\alpha_1$} [l] at 310 105
\pinlabel {$\alpha_g$} [l] at 568 65
\pinlabel {$\beta_1$} [b] at 216 80
\pinlabel {$\beta_g$} [b] at 472 69
\pinlabel {$+$} at 614 141
\endlabellist}
\includegraphics[scale=0.5]{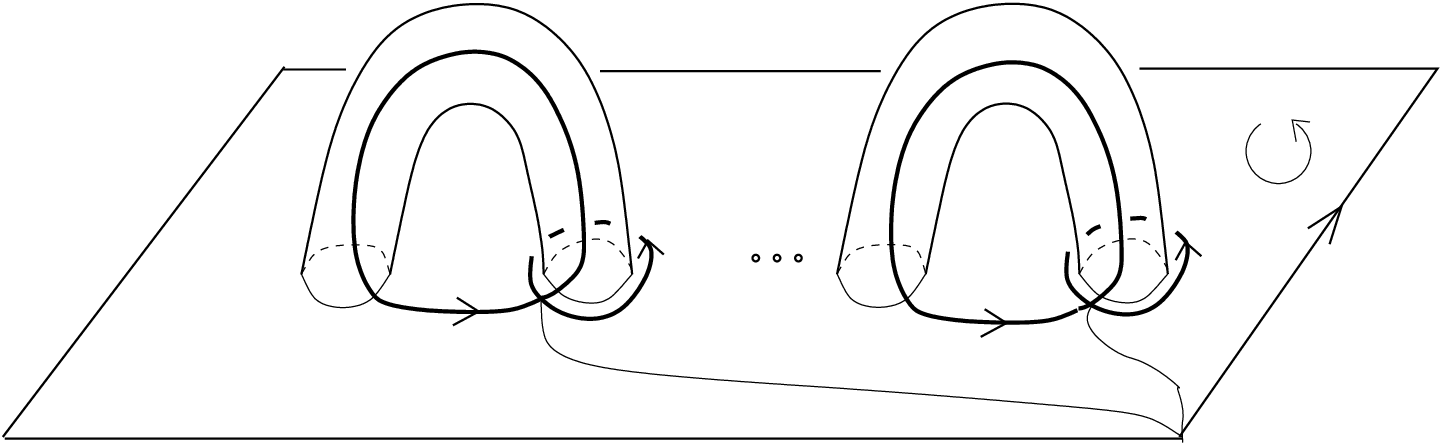}
\end{center}
\caption{The surface $\Sigma_{g,1}$ and a system of meridians and parallels $(\alpha,\beta)$.}
\label{fig:surface}
\end{figure}
 
\begin{definition}
An expansion $\theta: \pi \to \Tenshat(H)$ is \emph{symplectic} if it is group-like and if it sends
$\zeta^{-1}$ to $\exp(\omega)$.
\end{definition}

Bene, Kawazumi and Penner show in \cite{BKP} how to build a group-like expansion of  $\pi$
from any ``fatgraph presentation'' of the surface $\Sigma$, but this kind of expansion does not seem to be symplectic \cite[\S 6]{BKP}.
The group-like expansion $\theta_{(\gp{a},\gp{b})}$ relative to the basis $(\gp{a},\gp{b})$ of $\pi$ 
(see Example \ref{ex:basis}) is not either. Nevertheless, $\theta_{(\gp{a},\gp{b})}$ can be ``deformed'' to a symplectic expansion as the next proof shows.

\begin{lemma}
\label{lem:existence}
Symplectic expansions do exist.
\end{lemma}

\begin{proof}
By Corollary \ref{cor:Lie}, proving the existence of symplectic expansions is equivalent 
to proving the existence of a filtration-preserving isomorphism 
$\theta: \MLie(\pi) \to \Liehat(H)$ which induces the identity at the graded level 
and satisfies $\theta(\log(\zeta^{-1}))=  \omega$.
The isomorphism $\theta_{(\gp{a},\gp{b})}: \MLie(\pi) \to \Liehat(H)$ 
induced by the expansion $\theta_{(\gp{a},\gp{b})}$ of $\pi$ satisfies
$$
\theta_{(\gp{a},\gp{b})}(\log(\zeta^{-1})) = 
\log(\theta_{(\gp{a},\gp{b})}(\zeta^{-1})) = \widetilde{\omega}
$$
where we set
$$
\widetilde{\omega} := 
\log\left(\prod_{i=1}^g \exp(-b_i)  \otimes \exp(a_i) \otimes \exp(b_i) \otimes \exp(-a_i)  \right)
\ \in \Liehat(H).
$$
So, it is enough to show that there exists a filtration-preserving Lie algebra automorphism 
$\psi: \Liehat(H) \to \Liehat(H)$ which is the identity at the graded level and sends $\omega$ to $\widetilde{\omega}$:
then, the Lie algebra isomorphism $\theta:= \psi^{-1} \circ \theta_{(\gp{a},\gp{b})}$ will have the desired properties.

\begin{claim}
\label{claim:Lie_words}
Let $n\geq 1$. For all $i=1,\dots,g$ and for all $j=2,\dots,n$, there exist
some $u_i^{(j)},v_i^{(j)} \in \Lie_{j}$ such that
$$
\widetilde{\omega} \equiv 
 \sum_{i=1}^g \left[ a_i + \sum_{j=2}^n u_i^{(j)}, b_i + \sum_{j=2}^n v_i^{(j)}\right]
\mod \Liehat_{\geq n+2}.
$$
\end{claim}

This statement is proved by induction on $n\geq 1$. For $n=1$, Claim \ref{claim:Lie_words} holds because
the Baker--Campbell--Hausdorff formula shows that
$$
\widetilde{\omega} \equiv \sum_{i=1}^g [a_i,b_i] \mod \Liehat_{\geq 3}.
$$
If Claim \ref{claim:Lie_words} holds at step $n$, then the Lie series
$$
d:= \widetilde{\omega} - \sum_{i=1}^g \left[ a_i + \sum_{j=2}^n u_i^{(j)}, b_i + \sum_{j=2}^n v_i^{(j)}\right] 
$$
starts in degree $n+2$. Using the Jacobi identity, we can write the degree $n+2$ part of  the series $d$ as
$$
\sum_{i=1}^g \left( \left[a_i,v_i^{(n+1)}\right] + \left[u_i^{(n+1)},b_i\right] \right) \in \Lie_{n+2}
$$
for some $u_i^{(n+1)}, v_i^{(n+1)} \in \Lie_{n+1}$. We then have
\begin{eqnarray*}
&& \widetilde{\omega} - \sum_{i=1}^g \left[ a_i + \sum_{j=2}^{n+1} u_i^{(j)}, b_i + \sum_{j=2}^{n+1} v_i^{(j)}\right]\\
&\equiv & d - \sum_{i=1}^g \left( \left[a_i,v_i^{(n+1)}\right] + \left[u_i^{(n+1)},b_i\right] \right) 
\quad \equiv\quad  0   \mod \Liehat_{\geq n+3},
\end{eqnarray*}
which proves Claim \ref{claim:Lie_words} at step $n+1$.

Of course, those Lie words $u_i^{(j)},v_i^{(j)}$ are not unique 
but, as the above induction shows, we can choose those words at step $n+1$ 
in a way compatible with those chosen at step $n$. 
Then, we define a filtration-preserving Lie algebra endomorphism $\psi: \Liehat(H) \to \Liehat(H)$ 
by the formulas
$$
\psi(a_i) := a_i + \sum_{j\geq 2} u_i^{(j)} 
\quad \hbox{and} \quad
\psi(b_i) := b_i + \sum_{j\geq 2} v_i^{(j)}. 
$$
Clearly, $\psi$ induces the identity at the graded level and so, by completeness, $\psi$ is an isomorphism. 
Moreover, it satisfies 
$$
\psi(\omega) = \sum_{i=1}^g \left[\psi(a_i),\psi(b_i)\right] \equiv
\sum_{i=1}^g \left[ a_i + \sum_{j=2}^n u_i^{(j)}, b_i + \sum_{j=2}^n v_i^{(j)}\right]
\equiv \widetilde{\omega} \mod \Liehat_{\geq n+2}
$$
for all $n\geq 1$, so that we have $\psi(\omega)= \widetilde{\omega}$.
\end{proof}

If one allows coefficients to be in $\R$ rather than in $\Q$, 
then the ``harmonic expansions'' considered by Kawazumi in \cite{Kawazumi_harmonic} are symplectic.
Symplectic expansions with real coefficients also appear implicitely in \cite{Lin},
where the following proposition is proved using Chen's iterated integrals.
 
\begin{proposition}
Let $\wideparen{\Sigma}$ be the closed connected oriented surface of genus $g$, 
which is obtained from $\Sigma$ by gluing a $2$-disk along its boundary.
Then, any symplectic expansion $\theta$ of $\pi$ 
induces an isomorphism of filtered Lie algebras
$$
\theta: \MLie\left(\pi_1\left(\wideparen{\Sigma},*\right)\right)
 \stackrel{\simeq}{\longrightarrow} \Liehat(H)/ {\langle \langle \omega \rangle \rangle}
$$
where ${\langle \langle \omega \rangle \rangle}$ denotes the closed ideal of $\Liehat(H)$ generated by $\omega$.
\end{proposition}
 
\begin{proof}
The fundamental group of $\wideparen{\Sigma}$ has the following presentation: 
$$
\pi_1\left(\wideparen{\Sigma},*\right) = \left\langle \gp{a}_1,\dots,\gp{a}_g,\gp{b}_1,\dots,\gp{b}_g \left| \zeta \right. \right\rangle. 
$$
So, according to Theorem \ref{th:finitely_presented_group}, 
any symplectic expansion $\theta$ of $\pi$ induces an isomorphism
$\theta: \MLie(\pi_1(\wideparen{\Sigma},*)) \to \Liehat(H)/ {\langle \langle \omega \rangle\rangle}$.
\end{proof}

It does not seem easy to describe by a closed formula an instance of a symplectic expansion.
Nevertheless, we can still figure out a symplectic expansion up to some finite low degree with the help of a computer.

\begin{example}
There is a symplectic expansion $\theta$ which, in degree $\leq 4$,  is given by
\begin{eqnarray*}
\log \theta(\gp{a_i}) & =&  a_i - \frac{1}{2}[a_i,b_i] + \frac{1}{12} [[a_i,b_i],b_i] - \frac{1}{2} \sum_{j<i}  [[a_j,b_j],a_i] \\
&&- \frac{1}{24} [a_i,[a_i,[a_i,b_i]]] + \frac{1}{4} \sum_{j<i} [[a_j,b_j],[a_i,b_i]] + (\deg \geq 5),\\[0.1cm]
\log \theta(\gp{b_i}) & =&  b_i - \frac{1}{2}[a_i,b_i] + \frac{1}{12} [a_i,[a_i,b_i]] + \frac{1}{4} [[a_i,b_i],b_i] 
 + \frac{1}{2} \sum_{j<i} [b_i,[a_j,b_j]] \\
&& -\frac{1}{24} [[[a_i,b_i],b_i],b_i]  + \frac{1}{4} \sum_{j<i} [[a_j,b_j],[a_i,b_i]] +  (\deg \geq 5).
\end{eqnarray*}
This example has been found using the computer algebra software \texttt{Axiom}.
More precisely, we have written a small program which delivers this in genus $g=3$.
To conclude that an expansion having that form in degree $\leq 4$ is symplectic up to order $5$ for any genus $g\geq1$,
it  has been enough to write another small program which checks this assertion in genus $g=5$.
(The \texttt{.input} files are available on the author's webpage.)
\end{example}

\vspace{0.5cm}

\section{The total Johnson map}

\label{sec:total_Johnson}

Let $\Sigma$ be a compact connected oriented surface of genus $g$, with one boundary component.
The Torelli group of $\Sigma$ is denoted by
$$
\I := \I(\Sigma).
$$ 
In this section, we review the total Johnson map relative to a group-like expansion, 
which contains all Johnson homomorphisms and has been introduced by Kawazumi in \cite{Kawazumi}.
The total Johnson map is originally defined on $\I$, but it is easily extended to the monoid
$$
\cyl := \cyl(\Sigma)
$$
of homology cylinders over $\Sigma$.
Next, we consider certain truncations of this map which define homomorphisms on subgroups of the Johnson filtration,
and have diagrammatic descriptions if the choosen expansion is symplectic.

\subsection{The Dehn--Nielsen representation}

The canonical action of  the Torelli group of $\Sigma$ on its fundamental group  defines a group homomorphism
$$
\rho: \I \longrightarrow \Aut(\pi)
$$
which, by a classical result of Dehn and Nielsen, allows one to regard $\I$ 
as a subgroup of the automorphism group of  a free group.

\begin{theorem}[Dehn--Nielsen]
\label{th:Dehn-Nielsen}
The homomorphism $\rho$ is injective, and its image is the group 
$$
\IAut_\zeta(\pi)
$$
of automorphisms of $\pi$ that fix $\zeta=[\partial \Sigma]$ and induce the identity at the abelianized level.
\end{theorem}

We wish to consider an ``infinitesimal'' version of the Dehn--Nielsen representation $\rho$. For this and further purposes, 
let us recall how, in general,  an action 
$$
A: \mathcal{M} \longrightarrow \Aut(G)
$$
of a monoid $\mathcal{M}$ on a group $G$ can be transported to an action of $\mathcal{M}$ on the Malcev Lie algebra $\MLie(G)$:
$$
a: \mathcal{M} \longrightarrow \Aut\left(\MLie(G)\right).
$$
Each automorphism $\Psi$ of $G$ induces an automorphism $\widehat{\Q}[\Psi]$ of the complete Hopf algebra $\widehat{\Q}[G]$. 
So, by restricting to the primitive part, we get a filtered Lie algebra isomorphism
$\MLie(\Psi): \MLie(G) \to\MLie(G)$.
Thus, we obtain a group homomorphism
$$
\MLie: \Aut(G) \longrightarrow \Aut(\MLie(G))
$$
with values in the group of filtration-preserving automorphisms of $\MLie(G)$.
Then, we define the ``infinitesimal'' version of $A$ to be $a:=\MLie \circ A$.

\begin{lemma}
\label{lem:MLie_functorial}
If $G$ is residually torsion-free nilpotent, then the map
$$
\MLie: \Aut(G) \longrightarrow \Aut(\MLie(G))
$$
is an isomorphism onto the subgroup 
$$
\Aut_{\log(G)}(\MLie(G)) := \{\psi \in \Aut(\MLie(G)): \psi(\log(G)) = \log(G)\}.
$$ 
\end{lemma}

\begin{proof}
By assumption on $G$, the canonical map $\iota: G \to \MGroup(G)$
defines an inclusion $G \subset \MGroup(G)$.
For all $\Psi \in \Aut(G)$, we have
\begin{equation}
\label{eq:MLie_f}
\forall \gp{g} \in G, \ \MLie(\Psi)(\log(\gp{g}))= \log \Psi(\gp{g})
\end{equation}
so that $\MLie(\Psi)$ belongs to $\Aut_{\log (G)}(\MLie(G))$.
Conversely, given $\psi \in \Aut_{\log (G)}(\MLie(G))$,
we define a map $\Psi:G \to G$ by
$$
\forall \gp{g} \in G, \  \Psi(\gp{g}) := \exp \psi \log  (\gp{g}).
$$
A map $\Psi':G \to G$ is defined similarly from $\psi^{-1}$.
It is easily checked that $\Psi \Psi'= \Psi' \Psi = \Id_G$
and, using the Baker--Campbell--Hausdorff formula, that
$\Psi$ is a group homomorphism. 
So, $\Psi$ is a group automorphism of $G$ which satisfies $\MLie(\Psi) = \psi$.
Thus, the surjectivity of $\MLie$ onto $\Aut_{\log(G)}(\MLie(G))$ is proved.
Its injectivity follows from (\ref{eq:MLie_f}).
\end{proof}

We come back to the fundamental group $\pi$ of $\Sigma$.
By the previous discussion, we obtain an ``infinitesimal'' version of the Dehn--Nielsen representation
\begin{equation}
\label{eq:infinitesimal_Dehn-Nielsen}
\varrho: \I \longrightarrow \IAut_{\log(\zeta)}\left(\MLie(\pi)\right)
\end{equation}
defined by $\varrho:= \MLie \circ \rho$ and with values in the group of filtration-preserving automorphisms of $\MLie(\pi)$
that fix $\log(\zeta)$ and induce the identity at the graded level.
As an application of Lemma \ref{lem:MLie_functorial}, 
we obtain an ``infinitesimal'' formulation of Theorem \ref{th:Dehn-Nielsen}.

\begin{theorem}
The map $\varrho: \I \to \IAut_{\log(\zeta)}\left(\MLie(\pi)\right)$ is injective, and its image is
$$
\varrho(\I) =\big\{\psi \in \IAut_{\log(\zeta)}(\MLie(\pi)) : \psi(\log(\pi)) = \log(\pi)\big\}.
$$
\end{theorem}

Finally, assume that we are given a group-like expansion $\theta$ of $\pi$. Then, we denote by
$$
\IAut_{\theta \log(\zeta)}(\Liehat)
$$ 
the group of filtration-preserving automorphisms 
of $\Liehat = \Liehat(H)$ that induce the identity at the graded level and fix the element $\theta \log(\zeta)$.
The Dehn--Nielsen representation of the Torelli group is equivalent to the group homomorphism
$$
\varrho^\theta: \I \longrightarrow \IAut_{\theta \log(\zeta)}(\Liehat)
$$
defined by 
\begin{equation}
\label{eq:rho_theta}
\varrho^\theta(f) := \theta \circ \varrho(f) \circ \theta^{-1} =\theta \circ \MLie(f_*) \circ \theta^{-1}
\end{equation}
where $\theta: \MLie(\pi) \to \Liehat$ is the Lie algebra isomorphism given by Corollary \ref{cor:Lie}.

\subsection{Johnson homomorphisms and the total Johnson map}

Let us now  review Johnson homomorphisms in a few lines, the reader being referred to \cite{Johnson, Morita} for details.
For each integer $k\geq 1$, the Dehn--Nielsen representation $\rho$ induces a group homomorphism
$$
\rho_k: \I \longrightarrow \IAut_{\{\zeta\}}(\pi/\Gamma_{k+1} \pi)
$$
with values in the group of automorphisms of $\pi/\Gamma_{k+1} \pi$
that fix $\{\zeta\}$ and induce the identity at the abelianized level.
Let $\I[k]$ be the kernel of $\rho_k$. The sequence of subgroups
$$
\I = \I[1] \supset \I[2] \supset \I[3] \supset \cdots 
$$
is called the \emph{Johnson filtration} of $\I$. 
To define Johnson homomorphisms, one considers the short exact sequence
$$
1 \to \Hom\left(\pi/\Gamma_2 \pi, \Gamma_{k+1} \pi/ \Gamma_{k+2} \pi\right)
\to \Aut(\pi/\Gamma_{k+2} \pi) \to \Aut(\pi/\Gamma_{k+1} \pi)
$$
where  a group homomorphism $t: \pi/\Gamma_2 \pi \to \Gamma_{k+1}\pi/\Gamma_{k+2} \pi$
goes to the automorphism of $\pi/\Gamma_{k+2} \pi$ defined by $\{\gp{x}\} \mapsto \{\gp{x} \cdot t(\{\gp{x}\})\}$.
Thus, the map $\rho_{k+1}$ restricts to a  homomorphism
$$
\tau_k: \I[k] \longrightarrow \Hom(\pi/\Gamma_2 \pi,\Gamma_{k+1} \pi/\Gamma_{k+2} \pi) \otimes \Q
\simeq  \Hom(H, \Lie_{k+1}(H)) \simeq H \otimes \Lie_{k+1}(H).
$$
Here, the first isomorphism comes from the canonical identification between
$\Lie_n(H)$ and $\left(\Gamma_n \pi/\Gamma_{n+1} \pi \right)\otimes \Q$
and the second isomorphism is induced by the duality $H^*\simeq H$ defined by the intersection pairing $\omega$.
The map $\tau_k$ is known as the $k$-th \emph{Johnson homomorphism} and is considered, here, with rational coefficients.

Next, following \cite{Kawazumi}, we  give an equivalent description of the infinitesimal Dehn--Nielsen representation $\varrho$.
For this, we fix a group-like expansion $\theta$ of $\pi$. 

\begin{definition}[Kawazumi]
The \emph{total Johnson map} relative to the group-like expansion $\theta$  is the map 
$$
\tau^\theta: \I \longrightarrow \Hom(H,\Liehat_{\geq 2}) , \  f \longmapsto \left. \varrho^\theta(f)\right|_H - \Id_H.
$$
\end{definition}

The total Johnson map can be decomposed as follows:
$$
\tau^\theta = \sum_{m\geq 1} \tau_m^\theta
\ \in \prod_{m\geq 1} H \otimes  \Lie_{m+1}  \simeq \Hom(H,\Liehat_{\geq 2}).
$$
Such notation and terminology are justified by the following result from \cite{Kawazumi},
whose proof is given here for the sake of completeness.

\begin{theorem}[Kawazumi]
\label{th:Kawazumi_to_Johnson}
The degree $k$ part of the total Johnson map, restricted to the $k$-th term of the Johnson filtration,
coincides with the $k$-th Johnson homomorphism:
$$
\left. \tau^\theta_k\right|_{\I[k]} = \tau_k \ \in \Hom(\I[k], H\otimes \Lie_{k+1}).
$$
\end{theorem}

\begin{proof}
Let $f\in \I[k]$ and let $\{\gp{x}\} \in H$ be represented by $\gp{x}\in \pi$. We set 
$$
\gp{q} := \gp{x}^{-1} \cdot f_*(\gp{x}) \ \in \Gamma_{k+1} \pi
$$ 
so that, by definition of $\tau_k$, we have
\begin{equation}
\label{eq:value_of_Johnson}
\tau_k(f)(\{\gp{x}\}) = \{\gp{q}\} \ \in \Gamma_{k+1}\pi/ \Gamma_{k+2}\pi\otimes \Q \simeq \Lie_{k+1}.
\end{equation}
We also set 
$$
x_2 := \theta^{-1}(\{\gp{x}\}) - \log(\gp{x}) \ \in \MLie(\pi).
$$ 
Since $\theta(\gp{x}) = 1+ \{\gp{x}\} +(\deg \geq 2)$, 
$x_2$ belongs to $\widehat{\Gamma}_2 \MLie(\pi)$. 
Then, we have 
\begin{eqnarray*}
\theta \MLie(f_*) \theta^{-1}(\{\gp{x}\}) &=& \theta \MLie(f_*) ( \log(\gp{x}) + x_2) \\
&=& \theta ( \log (f_*(\gp{x})) + x_2 + x_{k+2}) \\
&=& \theta( \log (\gp{x}\cdot \gp{q}) + x_2 + x_{k+2}) \\
&=& \theta( \log(\gp{x}) + \log(\gp{q}) + x'_{k+2} + x_2 + x_{k+2}) \\
&=& \theta( \theta^{-1}(\{\gp{x}\}) +  \log(\gp{q}) ) +  \theta(x_{k+2} + x'_{k+2})\\
&=& \{\gp{x}\} +  \theta \log(\gp{q}) + \theta(x_{k+2} + x'_{k+2}).
\end{eqnarray*}
Here, $x_{k+2}$ and $x'_{k+2}$ are some elements of $\widehat{\Gamma}_{k+2} \MLie(\pi)$
and the fourth identity follows from the Baker--Campbell--Hausdorff formula. 
Since $\theta(x_{k+2} + x'_{k+2})$ belongs to $\Liehat_{\geq k+2}$, we deduce that 
$$
\tau^\theta_k(f)(\{\gp{x}\}) =  \{\theta \log(\gp{q})\} 
\ \in \Liehat_{\geq k+1}/\Liehat_{\geq k+2} \simeq \Lie_{k+1}.
$$
Since $\theta: \MLie(\pi) \to \Liehat$ induces the identity at the graded level, we conclude that
$$
\tau^\theta_k(f)(\{\gp{x}\}) = \{ \gp{q}\}
\stackrel{(\ref{eq:value_of_Johnson})}{=} \tau_k(f)(\{\gp{x}\}) 
\ \in \Gamma_{k+1}\pi/ \Gamma_{k+2}\pi \otimes \Q \simeq \Lie_{k+1}.
$$
\end{proof}

\begin{remark}
In fact, Kawazumi considers in \cite{Kawazumi} expansions which are not necessarily group-like,
so that he works with $\Tenshat(H)$ instead of $\Liehat(H)$.
He proves Theorem \ref{th:Kawazumi_to_Johnson} in this more general context.
\end{remark}

\subsection{Extension to the monoid of homology cylinders}

\label{subsec:extension}

As shown in \cite{GL}, the Johnson homomorphisms can be extended 
from the Torelli group $\I$ to the monoid of homology cylinders $\cyl$.
Indeed, by definition, a homology cylinder $C$ comes with a parametrization of its boundary
$$
c: \partial(\Sigma \times [-1,1]) \stackrel{\cong_+}{\longrightarrow} \partial C.
$$
This map $c$ splits into $c_+$ and $c_-$ where
$c_\pm := c|_{\Sigma \times \{\pm 1\} }: \Sigma \to C$.
The map $c_\pm$ is a homological equivalence so that, by Stalling's theorem \cite{Stallings}, 
it induces an isomorphism at the level of the $k$-th nilpotent quotient for every integer $k\geq 1$.
Thus, one gets a monoid homomorphism
$$
\rho_k: \cyl \longrightarrow \IAut_{\{\zeta\}}(\pi/\Gamma_{k+1} \pi), \
C \longmapsto (c_{-,*})^{-1} \circ c_{+,*}.
$$
Let $\cyl[k]$ be the kernel of $\rho_k$. The sequence of submonoids
$$
\cyl = \cyl[1] \supset \cyl[2] \supset \cyl[3] \supset \cdots 
$$
is called the \emph{Johnson filtration} of $\cyl$. Then, as in the case of the Torelli group, 
the map $\rho_{k+1}$ restricts to a monoid homomorphism
$$
\tau_k: \cyl[k] \longrightarrow H \otimes \Lie_{k+1}(H).
$$

The total Johnson map relative to a group-like expansion $\theta$ can also be extended from $\I$ to $\cyl$ in the following way. 
For all $k\geq1$, there is an ``infinitesimal'' version of the homomorphism $\rho_k$ 
\begin{equation}
\label{eq:infinitesimal_Stallings}
\varrho_k: \cyl \longrightarrow \IAut_{\log(\{\zeta\})} \left(\MLie(\pi/\Gamma_{k+1} \pi)\right)
\end{equation}
defined by $\varrho_k := \MLie \circ \rho_k$. Lemma \ref{lem:MLie_functorial} implies that 
\begin{equation}
\label{eq:Ker_rho_k}
\Ker(\varrho_k) = \cyl[k].
\end{equation}
By conjugating with the Lie algebra isomorphism $\theta$ from Corollary \ref{cor:Lie_nilpotent}, we obtain a monoid homomorphism
$$
\varrho_k^\theta: \cyl \longrightarrow \IAut_{\theta \log(\{\zeta\})}(\Lie/\Lie_{\geq k+1}).
$$
Equivalently, we can consider the map
$$
\tau^\theta_{[1,k[}: \cyl \longrightarrow \Hom(H,\Lie_{\geq 2}/\Lie_{\geq k+1}), \
C \longmapsto \left. \varrho^\theta_k(C)\right|_H - \Id_H.
$$
For all $l\geq k\geq 1$, we have the following commutative triangle
$$
\xymatrix{
{\cyl} \ar[rr]^-{\tau^\theta_{[1,l[}}  \ar[rrd]_-{\tau^\theta_{[1,k[}} & & 
{ \Hom(H,\Lie_{\geq 2}/\Lie_{\geq l+1})} \ar@{->>}[d]\\
& & {\Hom(H,\Lie_{\geq 2}/\Lie_{\geq k+1})} 
}
$$
where the vertical map is induced by the canonical projection 
$\Lie_{\geq 2}/\Lie_{\geq l+1} \to \Lie_{\geq 2}/\Lie_{\geq k+1}$.
Therefore, we can take the inverse limit as $k\to + \infty$ of the maps $\tau^\theta_{[1,k[}$ to obtain a map
$$
\tau^\theta: \cyl \longrightarrow \Hom(H,\Liehat_{\geq 2})
$$
whose restriction to $\I$ coincides with Kawazumi's total Johnson map.
Theorem \ref{th:Kawazumi_to_Johnson} and its proof can be extended without difficulty to homology cylinders.

As an alternative to $\tau^\theta$, we can equivalently consider the monoid homomorphism 
$$
\varrho^\theta: \cyl \longrightarrow \IAut_{\theta \log(\zeta)}(\Liehat)
$$
which sends any homology cylinder $C$ to the unique filtration-preserving automorphism of $\Liehat$ 
whose restriction to $H$ is $\Id_H + \tau^\theta(C)$.
If restricted to the Torelli group, this definition agrees with (\ref{eq:rho_theta}).

\subsection{Truncations of the total Johnson map}

In the next sections, we will be mostly interested
in certain truncations of the total Johnson map $\tau^\theta$.
These are introduced in the next statement.

\begin{proposition}
\label{prop:Kawazumi_homomorphism}
The degree $[k,2k[$ truncation of the total Johnson map $\tau^\theta$,
restricted to the $k$-th term of the Johnson filtration,
$$
\tau^\theta_{[k,2k[} := \sum_{m=k}^{2k-1} \tau^\theta_m:
\cyl[k] \longrightarrow \bigoplus_{m=k}^{2k-1} H \otimes \Lie_{m+1} 
\simeq \Hom(H,\Lie_{\geq k+1}/\Lie_{\geq 2k+1}) 
$$
is a monoid homomorphism. Moreover, its kernel is $\cyl[2k]$.
\end{proposition}

\begin{proof}
For all  $C \in \cyl[k]$, $\tau^\theta_{[k,2k[}(C)$ 
can be computed from $\varrho_{2k}^\theta(C)\in \Aut(\Lie/\Lie_{\geq 2k+1})$, and vice versa, thanks to the following equation:
\begin{equation}
\label{eq:equivalence}
\left.\varrho_{2k}^\theta(C)\right|_H 
= \Id_H + \tau_{[1,2k[}^\theta(C) 
= \Id_H + \tau_{[k,2k[}^\theta(C) 
\ \in \Hom(H, \Lie/\Lie_{\geq 2k+1}).
\end{equation}
Here, the second identity follows from the fact that $\varrho_k^\theta(C)$ is the identity, 
so that  $\tau_{[1,2k[}^\theta(C)$ starts in degree $k$. So, for all $C,D \in \cyl[k]$, we have
\begin{eqnarray*}
\left.\varrho_{2k}^\theta(D C)\right|_H  &=&  
\varrho_{2k}^\theta (D) \circ \left.\varrho_{2k}^\theta(C)\right|_H \\
&=&  \left.\varrho_{2k}^\theta(D)\right|_H + \varrho_{2k}^\theta(D) \circ \tau_{[k,2k[}^\theta(C)\\
&=& \Id_H + \tau_{[k,2k[}^\theta(D) + \varrho_{2k}^\theta(D) \circ \tau_{[k,2k[}^\theta(C)\\
&=& \Id_H + \tau_{[k,2k[}^\theta(D) + \tau_{[k,2k[}^\theta(C).
\end{eqnarray*}
Here, the last identity is an instance of the following elementary fact:
for all $c \in \Hom(H,\Lie_{\geq k+1}/\Lie_{\geq 2k+1})$ and for all 
$d\in \Aut(\Lie/\Lie_{\geq 2k+1})$ that reduces to the identity modulo $\Lie_{\geq k+1}/\Lie_{\geq 2k+1}$,
we have $d\circ c = c$. Thus, we conclude that
$$
\tau_{[k,2k[}^\theta(DC)
= \tau_{[k,2k[}^\theta(D)  + \tau_{[k,2k[}^\theta(C).
$$
Finally, we also deduce from (\ref{eq:equivalence}) that $\tau_{[k,2k[}^\theta(C)$ is trivial if and only if
$\varrho_{2k}^\theta(C)$ is the identity, which amounts to say that $C$ belongs to $\cyl[2k]$ by (\ref{eq:Ker_rho_k}).
\end{proof}

When the expansion $\theta$ is symplectic, the homomorphism $\tau^\theta_{[k,2k[}$ 
defined in Proposition \ref{prop:Kawazumi_homomorphism} has a diagrammatic description.

\begin{proposition}
\label{prop:Kawazumi_diagrammatic}
Assume that the expansion $\theta$ is symplectic.
Then, for all $C\in \cyl[k]$ and for all $j\in \{k, \dots, 2k-1\}$,
the Lie bracket of $\tau_j^\theta(C)\in H\otimes \Lie_{j+1}$ is trivial. Therefore, we have
$$
\eta^{-1} \tau_j^\theta(C) \in \T_{j}(H)
$$
where the diagrammatic space $\T_{j}(H)$  and the map $\eta$ 
have been introduced in \S \ref{subsec:fission} and \S \ref{subsec:free_nilpotent_Lie} respectively.
\end{proposition}

\begin{proof}
Let $(\alpha,\beta)$ be a system of meridians and parallels for the surface $\Sigma$,
and let $(a,b)$ be the corresponding basis of $H$.
Since $C$ acts trivially on $\MLie(\pi/\Gamma_{k+1}\pi)$, we have
$$
\varrho^\theta(C)(a_i) \equiv  a_i + \sum_{j=k+1}^{2k} u_i^{(j)} \mod \Liehat_{\geq 2k+1}
\quad \hbox{and} \quad
\varrho^\theta(C)(b_i) \equiv  b_i + \sum_{j=k+1}^{2k} v_i^{(j)} \mod \Liehat_{\geq 2k+1}
$$
where $u_i^{(j)}, v_i^{(j)} \in \Lie_{j}$ for all $j=k+1,\dots, 2k$.
Then, we obtain
$$
\varrho^\theta(C)(\omega) = \varrho^\theta(C)\left(\sum_{i=1}^g [a_i,b_i]\right) \equiv
\omega + \sum_{i=1}^g \sum_{j=k+1}^{2k} \left([a_i,v_i^{(j)}] + [u_i^{(j)},b_i] \right)
\mod \Liehat_{\geq 2k+2}.
$$
Since $\theta$ is symplectic, $\varrho^\theta(C)$ fixes $\omega$ and we deduce that
$$
\sum_{i=1}^g  \left([a_i,v_i^{(j)}] + [u_i^{(j)},b_i] \right) = 0 \in \Lie_{j+1},
\quad \forall j=k+1,\dots, 2k.
$$
Since $\tau_j^\theta(C) \in H\otimes \Lie_{j+1}$ is given by
$$
\tau_j^\theta(C) = 
\sum_{i=1}^g \left( -b_i \otimes u_i^{(j+1)} + a_i \otimes v_i^{(j+1)}\right),
$$
for all $j=k, \dots, 2k-1$, we conclude that its Lie bracket is zero.
\end{proof}

\vspace{0.5cm}

\section{Infinitesimal Morita homomorphisms}

\label{sec:Morita}

In this section, we define the Lie version $m_k$ of the $k$-th Morita homomorphism $M_k$,
and we show the equivalence between $m_k$ and $M_k$.
Next, we relate $m_k$ to the degree $[k,2k[$ truncation of the total Johnson map
and we deduce some properties for $m_k$.

\subsection{Definition of the infinitesimal Morita homomorphisms}

\label{subsec:infinitesimal_Morita}

For each integer $k\geq 1$, we define a monoid homomorphism
$$
m_k: \cyl[k] \longrightarrow H_3\left(\MLie(\pi/\Gamma_{k+1} \pi)\right)
$$
in a way very similar to the original definition of $M_k$ \cite{Morita,Sakasai},
the bar complex of a group being replaced by the Koszul complex of a Lie algebra. 
Details are as follows and need the following preliminary.

\begin{lemma}
\label{lem:H_trivial}
The linear map 
$$
H_2\left(\MLie(\pi/\Gamma_{n+1} \pi) \right) \longrightarrow H_2\left(\MLie(\pi/\Gamma_{m+1} \pi) \right),
$$
induced by the canonical Lie algebra homomorphism $\MLie(\pi/\Gamma_{n+1} \pi) \to \MLie(\pi/\Gamma_{m+1} \pi)$, 
is trivial for all $n>m$. Besides, the linear map
$$
H_3\left(\MLie(\pi/\Gamma_{n+1} \pi) \right) \longrightarrow H_3\left(\MLie(\pi/\Gamma_{m+1} \pi) \right)
$$
is trivial for all $n\geq 2m$.
\end{lemma}

\begin{proof}
Let $(\alpha,\beta)$ be a system of meridians and parallels for $\Sigma$, 
and let $(\gp{a},\gp{b})$ be the corresponding basis of $\pi$ 
which defines a group-like expansion $\theta_{(\gp{a},\gp{b})}$ of $\pi$ (as we saw in Example \ref{ex:basis}).
We also denote by $(a,b)$ the basis of $H$ defined by $(\gp{a},\gp{b})$.
According to Corollary \ref{cor:Lie_nilpotent}, 
$\theta_{(\gp{a},\gp{b})}$ induces for all $m\geq 1$ an isomorphism 
$$
 \Lie/\Lie_{\geq m+1} \stackrel{\simeq}{\longrightarrow} \MLie(\pi/\Gamma_{m+1} \pi)
$$
defined by $\{a_i\} \mapsto \log(\{\gp{a}_i\})$ and $\{b_i\} \mapsto \log(\{\gp{b}_i\})$.
Moreover, this isomorphism is compatible with the reduction maps
$$
\Lie/\Lie_{\geq n+1} \to \Lie/\Lie_{\geq m+1} \quad \hbox{and} \quad 
\MLie(\pi/\Gamma_{n+1} \pi) \to \MLie(\pi/\Gamma_{m+1} \pi)
$$
for all $n\geq m$.
Therefore, the lemma follows from (\ref{eq:Hopf_theorem}) and from Theorem \ref{th:Igusa-Orr}.
\end{proof}

We \emph{choose} $z \in \Lambda^2 \MLie(\pi/\Gamma_{2k+2} \pi)$
such that $\partial_2( z) = -\log(\{\zeta\})$ in the Koszul complex of the Lie algebra $\MLie(\pi/\Gamma_{2k+2} \pi)$. 
Such a $z$ exists since, by the Baker--Campbell--Hausdorff formula, we have
$$ 
\{\zeta\} \in \Gamma_2(\pi/\Gamma_{2k+2} \pi)
\Longrightarrow
\log(\{\zeta\}) \in \Gamma_2 \MLie(\pi/\Gamma_{2k+2} \pi).
$$
We denote by $\{z\}\in \Lambda^2 \MLie(\pi/\Gamma_{2k+1} \pi)$ the reduction of $z$. 
Let $C \in \cyl[k]$ for which we wish to define $m_k(C)$. We have 
$$
\partial_2\big(\{z\} - \varrho_{2k}(C)(\{z\}) \big)
= -\log \{\zeta\} + \varrho_{2k}(C)(\log\{\zeta\}) =0 \ \in \MLie(\pi/\Gamma_{2k+1} \pi).
$$
Thus, $\{z\} - \varrho_{2k}(C)(\{z\})$ is a $2$-cycle which, by Lemma \ref{lem:H_trivial},
is null-homologous. So, we can \emph{choose} a $t_C \in \Lambda^3 \MLie(\pi/\Gamma_{2k+1} \pi)$ such that
$$
\partial_3(t_C) = \{z\} - \varrho_{2k}(C)(\{z\}) \in \Lambda^2 \MLie(\pi/\Gamma_{2k+1} \pi).
$$
Observing that the reduction $\{t_C\} \in \Lambda^3 \MLie(\pi/\Gamma_{k+1} \pi)$ 
is a $3$-cycle since $\varrho_k(C)$ is the identity, we set
$$
m_k(C) := \left[ \{t_C\} \right] \in H_3\left(\MLie(\pi/\Gamma_{k+1} \pi)\right).
$$

\begin{lemma}
The above discussion defines a monoid homomorphism
$$
m_k: \cyl[k] \longrightarrow H_3\left(\MLie(\pi/\Gamma_{k+1} \pi)\right).
$$
\end{lemma}

\begin{proof}
First, assume that a different choice of $t_C$, say $t_C'$, has been done in the above discussion. 
Then, the difference $t_C-t_C' \in \Lambda^3 \MLie(\pi/\Gamma_{2k+1} \pi)$
is a $3$-cycle whose reduction 
$$
\{ t_C-t_C' \} = \{t_C\}-\{t'_C\} \ \in \Lambda^3 \MLie(\pi/\Gamma_{k+1} \pi)
$$
must be null-homologous by Lemma \ref{lem:H_trivial}. So, the choice of $t_C$ is irrelevant.
Next, assume that a different choice of $z$, say $z'$, has been done. 
The difference 
$$
\delta := z - z' \ \in \Lambda^2 \MLie(\pi/\Gamma_{2k+2} \pi)
$$
is then a $2$-cycle whose reduction
$\{\delta\} \in \Lambda^2 \MLie(\pi/\Gamma_{2k+1} \pi)$
must be null-homologous by Lemma \ref{lem:H_trivial}. 
Let $\varepsilon \in \Lambda^3 \MLie(\pi/\Gamma_{2k+1} \pi)$ be such that $\partial_3(\varepsilon) =\{\delta\}$.
The $3$-chain
$$
t'_C := t_C - \varepsilon + \varrho_{2k}(C)(\varepsilon) \ \in \Lambda^3 \MLie(\pi/\Gamma_{2k+1} \pi) 
$$
satisfies
$$
\partial_3(t_C') = \{z'\} - \varrho_{2k}(C)(\{z'\}) \in \Lambda^2 \MLie(\pi/\Gamma_{2k+1} \pi),
$$
and we have
$$
\{t'_C\} = \{t_C\} -\{\varepsilon\} + \varrho_k(C)(\{\varepsilon\}) = \{t_C\}
\ \in \Lambda^3 \MLie(\pi/\Gamma_{k+1} \pi).
$$
We conclude that $\left[ \{t_C\} \right]$ only depends on $C$, 
so that the map $m_k$ is well-defined.

Let $D\in \cyl[k]$ be another homology cylinder, for which 
we choose $t_D \in \Lambda^3 \MLie(\pi/\Gamma_{2k+1} \pi)$ satisfying
$$
\partial_3(t_D) = \{z\} - \varrho_{2k}(D)(\{z\}) \in \Lambda^2 \MLie(\pi/\Gamma_{2k+1} \pi).
$$
Thus, we have $m_k(D) = [\{t_D\}]$. 
The $3$-chain $t := t_C + \varrho_{2k}(C)(t_D) \in \Lambda^3 \MLie(\pi/\Gamma_{2k+1} \pi)$ satisfies
$$
\partial_3(t) = \partial_3(t_C) + \varrho_{2k}(C)\left( \partial_3(t_D)\right)
=  \{z\} - \varrho_{2k}(C \circ D)(\{z\}).
$$
Therefore, we have
$$
m_k(C\circ D)  = [\{t\}] = [\{t_C\} + \varrho_{k}(C)(\{t_D\})] = m_k(C) + m_k(D)
$$
and we conclude that the map $m_k$ is a monoid homomorphism.
\end{proof}

\subsection{Pickel's isomorphism}

Let $G$ be a finitely generated torsion-free nilpotent group.
Let us recall how Pickel relates the homology of $G$ to the homology of $\MLie(G)$ in \cite{Pickel}.
First, he shows that $\widehat{\Q}[G]$ is flat as a $\Q[G]$-module
and that, similarly, $\Uhat(\MLie(G))$ is flat as an $\U(\MLie(G))$-module.
Next, he deduces from \cite{Jennings} that the inclusion 
$\MLie(G) \subset \widehat{\Q}[G]$ induces an algebra isomorphism
$$
\Uhat(\MLie(G)) \simeq \widehat{\Q}[G].
$$
Finally, he considers, for all $n\geq 1$, the following sequence of isomorphisms:
$$
\xymatrix{
{\Tor_n^{\Q[G]}(\Q,\Q)} \ar[r]^-\simeq &
{\Tor_n^{\widehat{\Q}[G]}(\Q,\Q) \simeq   \Tor_n^{\Uhat(\MLie(G))}(\Q,\Q)} &
 {\Tor_n^{\U(\MLie(G))}(\Q,\Q)} \ar[l]_-\simeq  \\
{H_n(G)}  \ar@{=}[u] \ar@{-->}_-{\P} [rr] & & {H_n(\MLie(G)).} \ar@{=}[u]
}
$$
In dimension $n=3$ and for $G= \pi/\Gamma_{k+1} \pi$, 
Pickel's isomorphism links the $k$-th Morita homomorphism to its infinitesimal version.

\begin{proposition}
\label{prop:original_to_infinitesimal}
The following diagram is commutative:
$$
\xymatrix{
{\cyl[k]} \ar[r]^-{M_k} \ar[dr]_-{m_k} & {H_3(\pi/\Gamma_{k+1} \pi)} \ar[d]^-{\P}_-\simeq  \\
& {H_3\left(\MLie(\pi/\Gamma_{k+1} \pi)\right)}.  
}
$$
\end{proposition}

\begin{proof}
We need to make Pickel's isomorphism explicit at the chain level.
Let $G$ be a finitely generated torsion-free nilpotent group, for which  we set
$$
R:= \widehat{\Q}[G] = \widehat{\U}(\MLie(G)).
$$
Let $B \to \Q \to 0$ be a free resolution of $\Q$ as a $\Q[G]$-module,
and let $K \to \Q \to 0$ be a free resolution of $\Q$ as an  $\U(\MLie(G))$-module.
Then, by tensoring and using that  $\widehat{\Q}[G]$ is flat as $\Q[G]$-module,   
we get a free resolution of $\Q$ as a $\widehat{\Q}[G]$-module:
$$
\widehat{\Q}[G] \otimes_{\Q[G]} B \longrightarrow \Q \longrightarrow 0.
$$
Similarly, we obtain a free resolution of $\Q$ as an $\widehat{\U}(\MLie(G))$-module:
$$
\widehat{\U}(\MLie(G)) \otimes_{\U(\MLie(G))} K \longrightarrow \Q \longrightarrow 0.
$$
Thus, there exists a homotopy equivalence of chain complexes over the ring $R$
\begin{equation}
\label{eq:homotopy_equivalence}
f: \widehat{\Q}[G] \otimes_{\Q[G]} B \longrightarrow \widehat{\U}(\MLie(G)) \otimes_{\U(\MLie(G))} K
\end{equation}
(which is unique up to homotopy). Therefore, we get a homotopy equivalence
\begin{equation}
\label{eq:trivial_coef}
\Q \otimes_{R} f:
\Q \otimes_{\Q[G]} B = \Q \otimes_R \left(R \otimes_{\Q[G]} B\right)
\longrightarrow
\Q \otimes_R  \left( R \otimes_{\U(\MLie(G))} K\right) = \Q \otimes_{\U(\MLie(G))} K
\end{equation}
and Pickel's isomorphism is
$$
P = H_n\left( \Q \otimes_{R} f \right):
H_n(G)= H_n(\Q \otimes_{\Q[G]} B) \longrightarrow 
H_n(\Q \otimes_{\U(\MLie(G))} K) = H_n(\MLie(G)).
$$
Let us now assume that $B$ is the bar resolution for $\Q[G]$:
$$
\xymatrix{
\cdots \ar[r] & B_2 \ar[r]^-{\partial_2} & B_1 \ar[r]^-{\partial_1} 
& B_0 \ar[r]^-{\varepsilon} & \Q \ar[r] & 0 
}
$$
where $B_n = \Q[G] \cdot G^{\times n}$,
$\varepsilon$ is the augmentation of $\Q[G]$ and
$$
\partial_n\left(\gp{g}_1| \cdots | \gp{g}_n  \right) = \gp{g}_1 \cdot (\gp{g}_2 | \cdots | \gp{g}_n)
+ \sum_{i=1}^{n-1} (-1)^i \cdot (\gp{g}_1 | \cdots |\gp{g}_i \gp{g}_{i+1}| \cdots |\gp{g}_n)
+ (-1)^n \cdot (\gp{g}_1 |\cdots| \gp{g}_{n-1}). 
$$
We also assume that $K$ is the Koszul resolution for $\U(\MLie(G))$:
$$
\xymatrix{
\cdots \ar[r] & K_2 \ar[r]^-{\partial_2} & K_1 \ar[r]^-{\partial_1} 
& K_0 \ar[r]^-{\eta} & \Q \ar[r] & 0 
}
$$
where $K_n = \U(\MLie(G)) \otimes \Lambda^n \MLie(G)$,
$\eta$ is the augmentation of $\U(\MLie(G))$ and
\begin{eqnarray*}
\partial_n\left(1 \otimes g_1 \wedge \cdots \wedge g_n \right) &=& 
\sum_{i=1}^n (-1)^{i+1} g_i \otimes g_1 \wedge \cdots \widehat{g_i} \cdots \wedge g_n \\
&& + \sum_{1 \leq i<j \leq n} (-1)^{i+j} \otimes [g_i,g_j] 
\wedge g_1 \wedge \cdots \widehat{g_i} \cdots \widehat{g_j} \cdots \wedge g_n.
\end{eqnarray*}
For these choices of resolutions $B$ and $K$, Suslin and Wodzicki construct in \cite[\S 5]{SW} 
a homotopy equivalence $f$ of the form (\ref{eq:homotopy_equivalence}).
This chain map $f$ is derived from a contracting homotopy of the free resolution $\widehat{\U}(\MLie(G)) \otimes_{\U(\MLie(G))} K$ of $\Q$,
which is itself defined by means of the Poincar\'e--Birkhoff--Witt isomorphism.
Besides the fact that it is the identity of $R$ in degree $0$, 
we record two properties of the homotopy equivalence $f$: first, it is functorial in $G$ and, second,
it is given in degree $1$ by
$$
f(1 \otimes g) = \sum_{n\geq 1} \frac{1}{n!} \log(g)^{n-1} \otimes \log(g).
$$
We deduce that there exists a homotopy equivalence of the form (\ref{eq:trivial_coef})
between the bar complex of $G$ with trivial coefficients and the Koszul complex of $\MLie(G)$ 
with trivial coefficients, which is functorial in $G$ and is the log map in degree $1$.
This is exactly what we need to conclude that Morita's definition of $M_k$ \cite{Morita,Sakasai}
corresponds to our definition of $m_k$ through Pickel's isomorphism.
\end{proof}

\subsection{Properties of the infinitesimal Morita homomorphisms}

The $k$-th infinitesimal Morita homomorphism corresponds (up to a minus sign)
to the degree $[k,2k[$ truncation of the total Johnson map, relative to a symplectic expansion.

\begin{theorem}
\label{th:Kawazumi_to_Morita}
Let $\theta$ be a symplectic expansion of $\pi$.
Then, the following diagram is commutative:
$$
\xymatrix{
{\cyl[k]} \ar[d]_-{\tau^\theta_{[k,2k[}} \ar[r]^-{-m_k} & 
{H_3(\MLie(\pi/\Gamma_{k+1} \pi))} \ar[r]^-{\theta_*}_-\simeq & {H_3\left(\Lie/\Lie_{\geq k+1}\right)} \\
{\displaystyle \bigoplus_{j=k}^{2k-1} \D_{j+2}(H)} & & \ar[ll]_-\simeq^-{\eta} 
\ {\displaystyle  \bigoplus_{j=k}^{2k-1} \T_j(H)} \ar[u]_-{\Phi}^-{\simeq} 
 }
$$
Here, $\theta_*$ is induced by the Lie algebra isomorphism 
$\theta: \MLie(\pi/\Gamma_{k+1} \pi) \to \Lie/\Lie_{\geq k+1}$  from Corollary \ref{cor:Lie_nilpotent}.
\end{theorem}

\begin{proof}
Let $m_k^\theta: \cyl[k] \to H_3\left(\Lie/\Lie_{\geq k+1}\right)$ be the composition $\theta_* \circ m_k$,
and let $C\in \cyl[k]$. We are asked to show that
\begin{equation}
\label{eq:goal}
-m_k^\theta(C) = \Phi \eta^{-1} \tau^\theta_{[k,2k[}(C).
\end{equation}
Since the isomorphism $\theta: \MLie(\pi/\Gamma_{m+1} \pi) \to \Lie/\Lie_{\geq m+1}$ 
is compatible with the canonical projections 
$$
\Lie/\Lie_{\geq n+1} \to \Lie/\Lie_{\geq m+1}
\quad \hbox{and} \quad 
\MLie(\pi/\Gamma_{n+1} \pi) \to \MLie(\pi/\Gamma_{m+1} \pi)
$$
for all $n\geq m$, we can compute $m_k^\theta$  
directly from $\varrho^\theta_{2k} \in \Aut(\Lie/\Lie_{\geq 2k+1})$ in the following way. 
First, we set
$$
w:= \sum_{i=1}^g a_i \wedge b_i \ \in \Lambda^2 \Lie
$$
and we denote by $\{w\}$ its reduction to $\Lambda^2( \Lie/\Lie_{\geq 2k+1})$.
The $2$-chain
$$
\{w\} - \varrho_{2k}^\theta(C)(\{w\}) \in  \Lambda^2( \Lie/\Lie_{\geq 2k+1})
$$
is a $2$-cycle and, so, is a boundary by Lemma \ref{lem:H_trivial}.
Let $t_C  \in  \Lambda^3(\Lie/\Lie_{\geq 2k+1})$ be one of its antecedents by $\partial_3$. 
Then, the definition of $m_k$ given in \S \ref{subsec:infinitesimal_Morita} 
implies that
$$
m_k^\theta(C) = [\{t_C\}] \in  H_3\left(\Lie/\Lie_{\geq k+1}\right)
$$
where $\{t_C\}$ is the reduction of $t_C$ modulo $\Lie_{\geq k+1}/\Lie_{\geq 2k+1}$.

Next, since $C$ belongs to $\cyl[k]$, we can write
\begin{equation}
\label{eq:formulas_for_rho_2k}
\varrho^\theta_{2k}(C)(\{a_i\}) =  \{a_i\} + \sum_{j=k+1}^{2k} \left\{u_i^{(j)}\right\} 
\quad \hbox{and} \quad
\varrho^\theta_{2k}(C)(\{b_i\}) =  \{b_i\} + \sum_{j=k+1}^{2k} \left\{v_i^{(j)}\right\}
\end{equation}
where $u_i^{(j)}, v_i^{(j)} \in \Lie_{j}$ for all $j=k+1,\dots, 2k$.
Then, by  definition of $\tau^\theta$, we have 
$$
\tau^\theta_{[k,2k[}(C) =
\sum_{i=1}^g \left( -b_i \otimes \sum_{j=k}^{2k-1} u_i^{(j+1)} 
+ a_i \otimes \sum_{j=k}^{2k-1}  v_i^{(j+1)}\right).
$$
In the sequel, we set $a_C := \eta^{-1}  \tau^\theta_{[k,2k[}(C)$
and we consider the canonical embedding
$$
\digamma: \bigoplus_{j=k}^{2k-1} H\otimes \Lie_{j+1} \longrightarrow \Lambda^2( \Lie/ \Lie_{\geq 2k+1}),
\ u\otimes v \longmapsto \{u\} \wedge \{v\}.
$$ 
It follows from Lemma \ref{lem:d_3} and from (\ref{eq:eta}) that
$$
\digamma \eta(a_C) = \partial_3 \phi(a_C)
$$
where $\phi(a_C) \in \Lambda^3 (\Lie /\Lie_{\geq 2k+1})$ is obtained 
from the linear combination of trees $a_C$ by fission.
On the other hand, a direct computation based on (\ref{eq:formulas_for_rho_2k}) gives
$$
-\digamma  \tau^\theta_{[k,2k[}(C) = \{w\} - \varrho_{2k}^\theta(C)(\{w\}) + \varepsilon 
\quad \in \Lambda^2( \Lie/ \Lie_{\geq 2k+1})
$$
where $\varepsilon$ is the following $2$-cycle of degree  at least $2k+2$:
$$
\varepsilon := \sum_{i=1}^g \sum_{h=k}^{2k-1} \sum_{j=k}^{2k-1}
u_i^{(h+1)} \wedge v_i^{(j+1)}. 
$$ 
Since $H_2( \Lie/ \Lie_{\geq 2k+1})$ is concentrated in degree $2k+1$ according to (\ref{eq:Hopf_theorem}),
there exists an $e\in \Lambda^3( \Lie/ \Lie_{\geq 2k+1})$ of degree at least $2k+2$ such that $\partial_3(e)=\varepsilon$.
Then, we have
$$
\partial_3 \left(-\phi(a_C) - e\right) = -\digamma \eta (a_C) - \varepsilon  = - \digamma \tau^\theta_{[k,2k[}(C) - \varepsilon 
=  \{w\} - \varrho_{2k}^\theta(C)(\{w\}),
$$
which shows that $\left(-\phi(a_C)-e\right)$ can play the role of $t_C$. We deduce that
$$
m_k^\theta(C) = [\{ - \phi(a_C) - e \} ] =  [-\{ \phi(a_C) \} - \{e \}].
$$
But, $\{e \} \in \Lambda^3(\Lie/  \Lie_{\geq k+1} )$ is a $3$-cycle
(since $\{\varepsilon\} \in \Lambda^2(\Lie/  \Lie_{\geq k+1} )$ vanishes) of degree  at least $2k+2$,
and so, is null-homologous (since $H_3(\Lie /  \Lie_{\geq k+1} )$
is concentrated in degrees $[k+2,2k+1]$ by Theorem \ref{th:Igusa-Orr}).
Thus, we obtain (\ref{eq:goal}).
\end{proof}

As applications of Theorem \ref{th:Kawazumi_to_Morita}, 
we recover two important properties for the homomorphism $M_k$ by proving them for $m_k$.
First, the  $k$-th Morita homomorphism $M_k$ 
is known to determine the $k$-th Johnson homomorphism \cite{Morita}. 
To obtain a similar fact for $m_k$, 
we need the central extension of Lie algebras
\begin{equation}
\label{eq:extension_MLie}
0 \to  \Lie_{k+1}(H) \to \MLie(\pi/\Gamma_{k+2} \pi) \to \MLie(\pi/\Gamma_{k+1} \pi) \to 1
\end{equation}
whose first map is the composition 
$$
\Lie_{k+1}(H) \stackrel{\simeq}{\longrightarrow} 
(\Gamma_{k+1}\pi/ \Gamma_{k+2} \pi) \otimes \Q 
\stackrel{\log \otimes \Q}{\longrightarrow} \MLie(\pi/\Gamma_{k+2} \pi).
$$

\begin{corollary}
\label{cor:Morita_to_Johnson}
We have the following commutative diagram
$$
\xymatrix{
\cyl[k] \ar[r]^-{m_k} \ar[rd]_-{ -\tau_k} & 
H_3\left(\MLie(\pi/\Gamma_{k+1} \pi)\right) \ar[d]^-{d^2_{3,0}}\\
& H\otimes \Lie_{k+1}(H)
}
$$
where the homomorphism $d^2_{3,0}$ is the second-stage differential 
of the Hochschild--Serre spectral sequence associated to (\ref{eq:extension_MLie}).
\end{corollary}

\begin{proof}
Let $\theta$ be a symplectic expansion of $\pi$.  We consider the diagram
\begin{equation}
\label{eq:two_central_extensions}
\xymatrix{
0 \ar[r] &  {\Lie_{k+1}(H)} \ar[r] \ar@{=}[d] & {\MLie(\pi/\Gamma_{k+2} \pi)} \ar[d]^-{\theta}_-\simeq \ar[r] 
& {\MLie(\pi/\Gamma_{k+1} \pi)} \ar[r] \ar[d]^-{\theta}_-\simeq  & 1 \\
0 \ar[r] &  {\Lie_{k+1}(H)} \ar[r] & {\Lie/\Lie_{\geq k+2}} \ar[r] & {\Lie/\Lie_{\geq k+1}} \ar[r] & 1
}
\end{equation}
whose vertical isomorphisms $\theta$ are given by Corollary \ref{cor:Lie_nilpotent}.
The commutativity of the diagram is a consequence of the fact that these isomorphisms are induced
by the Lie algebra isomorphism $\theta:\MLie(\pi) \to \Liehat(H)$ from Corollary \ref{cor:Lie}.
By naturality of the Hochschild--Serre spectral sequence, the corollary is equivalent to 
the commutativity of the  diagram
$$
\xymatrix{
\cyl[k] \ar[r]^-{m_k^\theta} \ar[rd]_-{-\tau_k} & 
H_3\left(\Lie/\Lie_{\geq k+1}\right) \ar[d]^-{d^2_{3,0}}\\
& H\otimes \Lie_{k+1}(H),
}
$$
where $m_k^\theta$ denotes $\theta_*\circ m_k$ 
and where $d^2_{3,0}$ refers now to the central extension 
given by the second line of (\ref{eq:two_central_extensions}).
We conclude using some previous results:
$$
d^2_{3,0} \circ m_k^\theta 
\stackrel{(\ref{eq:eta_Phi})}{=} \eta_k \circ \Phi^{-1}  \circ m_k^\theta 
\stackrel{\operatorname{Thm\ \ref{th:Kawazumi_to_Morita}}}{=} - \tau^\theta_k 
\stackrel{\operatorname{Thm\ \ref{th:Kawazumi_to_Johnson}}}{=} - \tau_k.
$$
\end{proof}

Second, the kernel of $M_k$ is known to be the $2k$-th term of the Johnson filtration:
this has been proved by Heap \cite{Heap} in the case of the Torelli group, 
and by Sakasai \cite{Sakasai} in the general case.
Alternatively, we can deduce this from the following result.

\begin{corollary}
The kernel of $m_k$ is $\cyl[2k]$.
\end{corollary}

\begin{proof}
This follows immediately from Theorem \ref{th:Kawazumi_to_Morita} and
Proposition \ref{prop:Kawazumi_homomorphism}.
\end{proof}

\vspace{0.5cm}

\section{The tree-reduction of the LMO homomorphism}

\label{sec:LMO}

In this last section, we prove that the LMO functor introduced in \cite{CHM} 
defines a symplectic expansion of $\pi$. 
Next, we consider the  LMO homomorphism,
which is the restriction of the LMO functor to the monoid of homology cylinders.
We show that the total Johnson map relative to that particular expansion 
is equivalent to the tree-reduction of the LMO homomorphism.
We deduce that the degree $[k,2k[$ part of the tree-reduction of the LMO homomorphism 
coincides with the $k$-th infinitesimal Morita homomorphism.

We assume that the reader is familiar enough with the LMO invariant \cite{LMO,BGRT1,BGRT2,Ohtsuki}
and, more specifically, with the constructions of  \cite{CHM}.
(The surface $\Sigma$ is denoted by $F_g$ in \cite{CHM},  
and  the monoid of homology cylinders $\cyl$ is denoted there by $\mathcal{C}yl$.)

\subsection{The monoid of bottom knots in a thickened surface}

An essential ingredient to derive a symplectic expansion from the LMO functor
is the notion of  ``bottom knot'' in $\Sigma \times [-1,1]$. 
We fix two distinct points $p,q$ in the interior of $\Sigma$
and, at each of them, we fix a non-zero tangent vector which will be implicit in the sequel.

\begin{definition}
A \emph{bottom knot} in $\Sigma \times [-1,1]$ is a connected framed oriented tangle, 
which starts from $q\times (-1)$ and ends at $p \times (-1)$. 
Two bottom knots are considered to be the same if they differ
by an ambient isotopy of $\Sigma \times [-1,1]$ relative to the boundary.
\end{definition}

\noindent
An example of bottom knot is shown on Figure \ref{fig:bottom_knot}.
Another example is the \emph{trivial} bottom knot which, 
together with an interval in $\Sigma \times (-1)$ that connects $p\times (-1)$ to $q\times (-1)$, 
bounds an embedded disk in $\Sigma \times [-1,1]$.
We denote by $\bottomK(\Sigma)$, or simply by $\bottomK$, 
the set of  bottom knots in $\Sigma \times [-1,1]$.

\begin{figure}[h]
\begin{center}
\labellist \small \hair 2pt
\pinlabel {$p$} [t] at 136 49
\pinlabel {$q$} [t] at 196 49
\endlabellist
\includegraphics[scale=0.3]{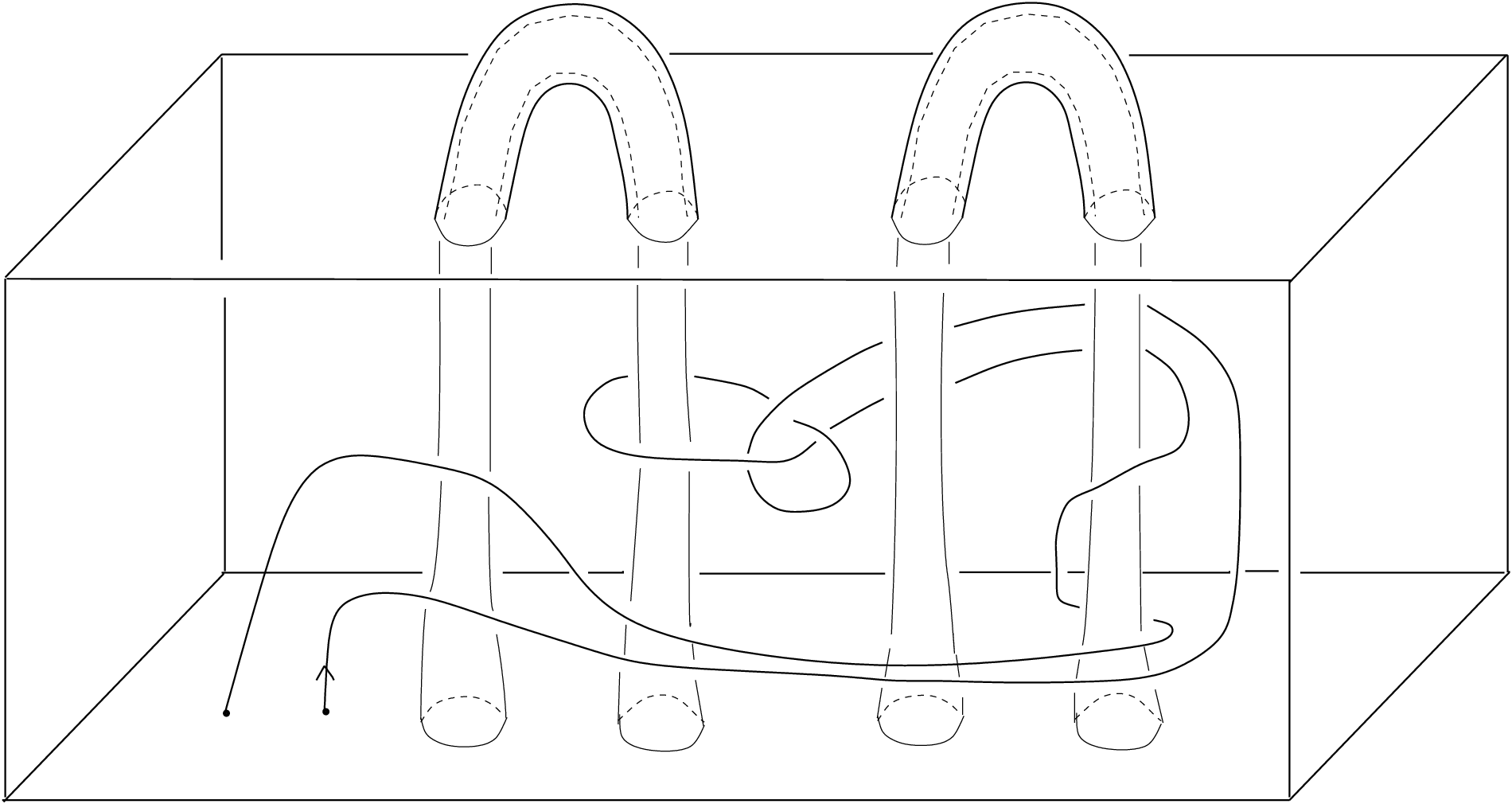}
\end{center}
\caption{An example of bottom knot in $\Sigma \times [-1,1]$ in genus $g=2$.
(The blackboard framing convention is used.)}
\label{fig:bottom_knot}
\end{figure}

The reader may be more familiar with the notion of \emph{string knot} in $\Sigma \times [-1,1]$,
which is a connected framed oriented tangle
starting from $q\times (-1)$ and ending at $q \times 1$.  
Let $\stringK(\Sigma)$ be the set of string knots in $\Sigma \times [-1,1]$.
There is a canonical bijection 
$$
b:\stringK(\Sigma) \stackrel{\simeq}{\longrightarrow} \bottomK(\Sigma)
$$
which is schematically defined by Figure \ref{fig:string_to_bottom}.
The set $\stringK(\Sigma)$ is a monoid, whose multiplication is defined by ``stacking'':
$$
K \cdot L  := \begin{array}{|c|}\hline L \\ \hline K\\ \hline \end{array}
\quad \quad \forall K,L \in \stringK(\Sigma)
$$
and whose identity element is the trivial string knot $(\Sigma \times [-1,1], q \times [-1,1])$. 
Therefore, the push-out by $b$ defines a monoid structure on $\bottomK(\Sigma)$,
whose identity element is the trivial bottom knot.

\begin{figure}[h]
\begin{center}
\labellist \small \hair 2pt
\pinlabel {$K$} [l] at 118 60
\pinlabel {$b(K)$} [l] at 32 110
\pinlabel {$\Sigma\times [-1,1]$} [tr] at 367 161
\pinlabel {$\Sigma\times [-1,1]$} [tr] at 412 208
\endlabellist
\includegraphics[scale=0.5]{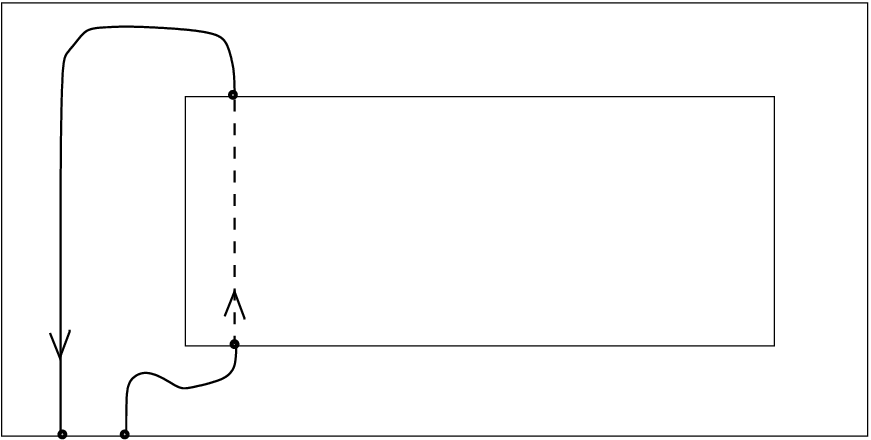}
\end{center}
\caption{How to transform a string knot into a bottom knot.}
\label{fig:string_to_bottom}
\end{figure}
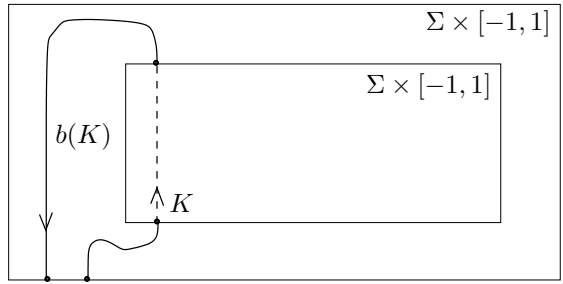

\begin{definition}
Two bottom knots $K$ and $K'$ in $\Sigma\times [-1,1]$ are \emph{homotopic} if
$K$ can be transformed to $K'$ by a framing change and a finite number of crossing changes.
\end{definition}

\noindent
The homotopy relation, which we denote by $\simeq_h$, 
is an equivalence relation on $\bottomK$ which is compatible with its multiplication.

\begin{lemma}
\label{lem:knot_to_loop}
There is a canonical monoid isomorphism
$$
\ell:\bottomK(\Sigma)/\!\simeq_h\ \stackrel{\simeq}{\longrightarrow} \pi
$$
defined by assigning to each bottom knot $K$ a based loop $\ell(K)$ in $\Sigma \times [-1,1]$,
as shown in Figure \ref{fig:knot_to_loop}, and by identifying 
$\pi=\pi_1(\Sigma,*)$ with $\pi_1(\Sigma\times [-1,1], *)$.
\end{lemma}

\begin{figure}[h]
\begin{center}
\labellist \small \hair 2pt
\pinlabel {$K$} [b] at 209 213
\pinlabel {$*$}  at 778 2
\pinlabel {$\ell(K)$} [b] at 739 13
\endlabellist
\includegraphics[scale=0.3]{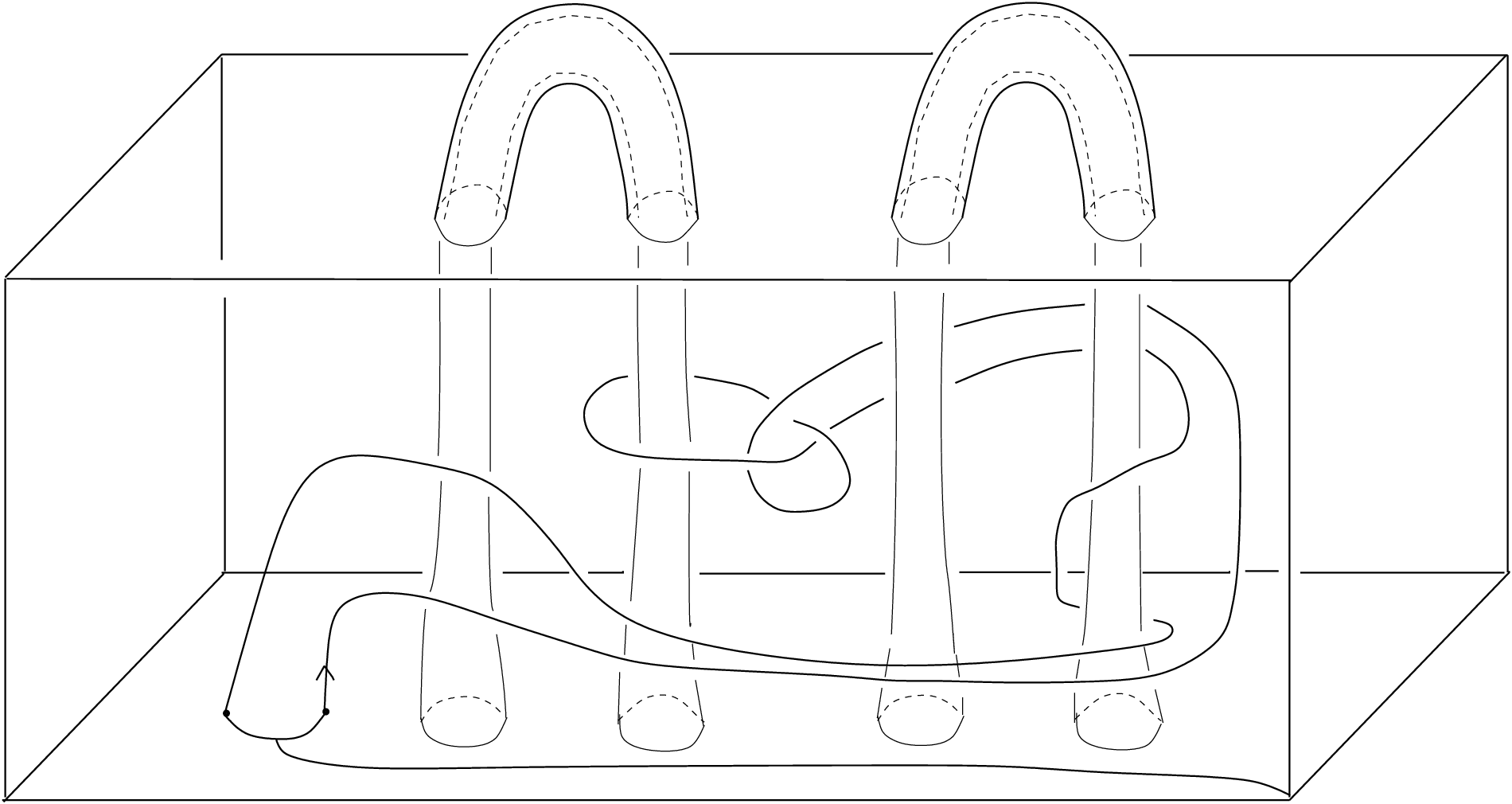}
\end{center}
\caption{How to transform a bottom knot into a based loop.}
\label{fig:knot_to_loop}
\end{figure}

\begin{proof}
The statement of the lemma certainly defines a map $\ell:\bottomK(\Sigma)/\!\simeq_h \to \pi$.
By composition with $b$, one gets a map
$\ell \circ b: \stringK(\Sigma)/\!\simeq_h\  \to \pi$
which is well-known to be a monoid isomorphism.
\end{proof}

Each bottom knot $K$ can be transformed into a cobordism by ``digging'' $\Sigma \times [-1,1]$ along $K$, 
and this cobordism is a ``Lagrangian'' cobordism from $F_g$ to $F_{g+1}$ in the sense of \cite{CHM}.  
To be more specific, the $1$-st handle of $F_{g+1}$ is identified with the boundary of 
a neighborhood of $K$ in $\Sigma \times [-1,1]$ while,  for all $i=2,\dots,g+1$,
the $i$-th handle of  $F_{g+1}$ corresponds to the $(i-1)$-st handle of $\Sigma \times (-1)$.
This construction is a special case of the way cobordisms are presented in \cite{CHM} 
in terms of ``bottom-top'' tangles. Thus, we get an inclusion
\begin{equation}
\label{eq:inclusion_bottom_knots}
\bottomK \subset \LCob(g,g+1).
\end{equation}
The monoid structure of $\bottomK$ can be defined in terms of the monoidal structure of $\LCob$.
For this, we recall that $\LCob$ is a subcategory of the category of cobordisms $\Cob$,
which is braided monoidal and for which the object $1$ is a Hopf algebra \cite{Kerler}.
Then, the multiplication of $\bottomK$ is given by 
\begin{equation}
\label{eq:product_bottom_knots}
\forall K,L \in \bottomK, \quad
K \cdot L  = (\mu \otimes \Id_g) \circ (\Id_1 \otimes K) \circ L,
\end{equation}
where $\mu \in \LCob(2,1)$ is the product of the Hopf algebra object $1$. 
The identity element of $\bottomK$ is then
\begin{equation}
\label{eq:identity_bottom_knots}
(\hbox{trivial bottom knot}) = \eta \otimes \Id_g
\end{equation}
where $\eta \in \LCob(0,1)$ is the unit of the Hopf algebra object $1$.

\subsection{The symplectic expansion defined by the LMO functor}

We now explain how the LMO functor defines 
a symplectic expansion of $\pi$ by considering bottom knots up to homotopy.
First, we recall the diagrammatic analogue of the homotopy relation introduced by Bar-Natan 
in the context of Milnor's $\mu$  invariants of string links \cite{Bar-Natan_homotopy}. 

Let $S$ be  a finite set, and let $r\in S$.
We denote by $\A(S)$ the space of Jacobi diagrams colored by $S$, and 
we denote by $\A(\uparrow^{S})$ the space of Jacobi diagrams 
based on the $1$-manifold $\uparrow^S$, 
which consists of one oriented interval $\uparrow^s$ for each element $s\in S$. 
Recall from \cite{Bar-Natan} that there is a diagrammatic analogue
of the Poincar\'e--Birkhoff--Witt isomorphism
$$
\chi: \mathcal{A}(S) \stackrel{\simeq}{\longrightarrow}  \mathcal{A}(\uparrow^S).
$$
Following \cite{Bar-Natan_homotopy}, we consider the subspace
$$
\homotopy(r) \subset  \A(S) 
$$
generated by Jacobi diagrams with at least one component that is looped
or that posseses at least two  univalent vertices colored by $r$.
Similarly, let 
$$
\homotopy(\uparrow^{r}) \subset  \A(\uparrow^{S}) 
$$
be the subspace generated by Jacobi diagrams with at least 
one dashed component that is looped or that posseses at least 
two univalent vertices attached to $\uparrow^{r}$.
The following statement is proved in \cite[Theorem 1]{Bar-Natan_homotopy}.

\begin{theorem}[Bar-Natan]
\label{th:homotopy}
For any finite set $S$ and for all $r \in S$, we have
$$
\chi\left(\homotopy(r)\right) = \homotopy(\uparrow^{r}).
$$
\end{theorem}

By the inclusion (\ref{eq:inclusion_bottom_knots}), bottom knots are Lagrangian cobordisms.
They are promoted to Lagrangian $q$-cobordisms (in the sense of \cite{CHM})
if we agree to equip each of them with $r_g$ on the top surface 
and with $r_{g+1}$ on the bottom surface, where
$$
r_g:= (\bullet \cdots (\bullet (\bullet \bullet)) \cdots )
$$
is the length $g$ right-handed non-associative word in the single letter $\bullet$. 
Then, the LMO functor $\Ztilde$ introduced in \cite{CHM} can be applied to bottom knots:
$$
\Ztilde: \bottomK \longrightarrow \tsA(g,g+1).
$$
Recall that $\tsA(g,g+1)$ is a subspace of 
$$
\A(\set{g}^+ \cup \set{g+1}^-)
$$
where $\set{g}^+$ denotes the finite set $\{1^+, \dots, g^+\}$ and $\set{g+1}^- $
stands for the finite set $\{1^-, \dots, (g+1)^-\}$.
In the construction of the LMO functor, the color $i^+$ refers to the $i$-th handle of the top surface,
while the color $i^-$ refers to the $i$-th handle of the bottom surface.
Thus, for those cobordisms arising from bottom knots, it is natural to rename the colors as follows:
\begin{equation}
\label{eq:change_variables}
1^- \longmapsto r
\quad \quad \hbox{and} \quad \quad
i^- \longmapsto (i-1)^-, \ \forall i=2,\dots,(g+1)
\end{equation}
so that the variable $r$ refers to the bottom knot.
After this change of variables (which is often tacit in the sequel), the LMO functor gives a map
$$
\Ztilde: \bottomK \longrightarrow \A(\set{g}^+ \cup \set{g}^-\cup \{r\}).
$$

\begin{lemma}
\label{lem:homotopy}
For any bottom knot $K$ in $\Sigma \times [-1,1]$, 
$\Ztilde(K)$ mod $\homotopy(r)$ only depends on the homotopy class of $K$.
\end{lemma}

\begin{proof}
Let us consider a crossing change (respectively, a framing change) $K \leadsto K'$.
The Lagrangian cobordisms defined by bottom knots are ``special'' in the sense of \cite{CHM}. 
So, we can apply \cite[Lemma 5.5]{CHM} to derive $\Ztilde(K)$ from
the Kontsevich integral $Z(L)$ of an appropriate framed oriented tangle $L$ in $[-1,1]^3$.
Similarly, we can reduce $\Ztilde(K')$ to the Kontsevich integral $Z(L')$ 
of an appropriate framed oriented tangle $L'$ in $[-1,1]^3$.
Furthermore, we can find such an $L$ and an $L'$ that only differ 
by a self-crossing change (respectively, by small kinks) of a component $\uparrow^{l_0}$. 
Then, the lemma follows from the fact that
$$
Z(L') - Z(L) \in \homotopy(\uparrow^{l_0}),
$$
which is well-known and is easily checked from the value of the Kontsevich integral 
on a crossing (respectively, on a small kink).  
\end{proof}

Let us now recall that, for any finite set $S$, $\A(S)$ is a Hopf algebra, 
whose product is given by the disjoint union $\sqcup$ and whose coproduct is defined by 
\begin{equation}
\label{eq:coproduct}
\Delta(D) = \sum_{D= D' \sqcup D''} D' \otimes D''
\end{equation}
for all diagrams $D$.  Since $\homotopy(r)$ is a Hopf ideal  of $\A(\set{g}^+ \cup \set{g}^-\cup \{r\})$,
the quotient space  $\A(\set{g}^+ \cup \set{g}^-\cup \{r\})/\homotopy(r)$ is a Hopf algebra. Besides,
the subspace of $\A(\set{g}^+ \cup \set{g}^-\cup \{r\})$ spanned by tree-shaped Jacobi diagrams,
with at most one $r$-colored vertex on each component, is a Hopf subalgebra of $\A(\set{g}^+ \cup \set{g}^-\cup \{r\})$.
This subspace is isomorphically mapped onto $\A(\set{g}^+ \cup \set{g}^-\cup \{r\})/\homotopy(r)$ 
by the canonical projection. Thus, in the sequel,  
$$
\A(\set{g}^+ \cup \set{g}^-\cup \{r\})/\homotopy(r)
$$  
will either denote a quotient Hopf algebra or a Hopf subalgebra of $\A(\set{g}^+ \cup \set{g}^-\cup \{r\})$.

Let $K$ be a bottom knot in $\Sigma \times [-1,1]$. By the previous paragraph,
$\Ztilde(K)$ mod $\homotopy(r)$ can be regarded as a series of tree-shaped 
Jacobi diagrams with at most one $r$-colored vertex on each component.
Using the functoriality of $\Ztilde$, it is easily checked that  
the subseries of $\Ztilde(K)$ mod $\homotopy(r)$ consisting only of  
tree-shaped Jacobi diagrams \emph{without} $r$-colored vertices, is $\Ztilde(\Sigma\times [-1,1])$.
This is the identity of the object $g$ in the category $\tsA$, namely
$$
\Id_g = \exp_\sqcup\left(\sum_{i=1}^g \strutgraph{i^-}{i^+}\right).
$$
Moreover, $\Ztilde(K)$ mod $\homotopy(r)$ is group-like since $\Ztilde(K)$ is so.
Thus, we deduce that
$$
\log_\sqcup\left(\Ztilde(K) \mod \homotopy(r)\right) - \sum_{i=1}^g  \strutgraph{i^-}{i^+}
$$
is a series of tree-shaped \emph{connected} Jacobi diagrams with \emph{exactly one} $r$-colored vertex. 
Such a  series can be interpreted as a Lie series in the variables $\set{g}^+\cup \set{g}^-$
as we did in (\ref{eq:tree_to_commutator}) and, so, as an element of  $\Liehat(H)$ after the following change of variables:
$$
i^+ \longmapsto b_i, \ \forall i=1,\dots,g \quad \quad  \hbox{and} \quad \quad i^- \longmapsto a_i, \ \forall i=1,\dots,g.
$$
Thus, we have an inclusion
\begin{equation}
\label{eq:inclusion_Lie}
\Liehat(H) \subset \A(\set{g}^+ \cup \set{g}^-\cup \{r\})
\end{equation}
and we can write
$$
\log_\sqcup\left(\Ztilde(K) \mod \homotopy(r)\right) - \sum_{i=1}^g  \strutgraph{i^-}{i^+} 
\quad \in \Liehat(H).
$$
Therefore, we can use Lemma \ref{lem:knot_to_loop} and Lemma \ref{lem:homotopy} to define a map
$$
\thetaLMO: \pi \longrightarrow \Tenshat(H)
$$
by the formula
\begin{equation}
\label{eq:theta_LMO}
\thetaLMO\left(\ell(K)\right) := 
\exp_\otimes 
\left({ \log_\sqcup\left(\Ztilde(K) \mod \homotopy(r)\right) - 
\sum_{i=1}^g \strutgraph{i^-}{i^+} }\right).
\end{equation}

\begin{proposition}
The map $\thetaLMO$ is  a symplectic expansion of $\pi$.
\end{proposition}

\begin{proof}
First, we prove that $\thetaLMO$ is a monoid homomorphism. 
According to (\ref{eq:identity_bottom_knots}), 
the $q$-cobordism corresponding to the trivial bottom knot is $\eta \otimes \Id_{r_{g}}$,
and we have 
$$
\Ztilde(\eta \otimes \Id_{r_{g}}) = \Ztilde(\eta) \otimes \Ztilde(\Id_{r_{g}}) 
= \varnothing \otimes \Id_g = \exp_\sqcup\left(\sum_{i=1}^g \strutgraph{i^-}{i^+}\right).
$$
We deduce that $\thetaLMO(1) = \exp_\otimes(0) =1$. 
To prove the multiplicativity of $\thetaLMO$,
we consider two bottom knots $K$ and $L$. According to (\ref{eq:product_bottom_knots}),
the Lagrangian $q$-cobordism corresponding to $K\cdot L$ can be decomposed as
$$
K \cdot L = (\mu \otimes \Id_{r_g}) \circ P_{\bullet,\bullet,r_g} 
\circ (\Id_\bullet \otimes K) \circ L.
$$
Here, $P_{\bullet,\bullet,r_g}$ is the $q$-cobordism 
$(\bullet(\bullet r_g)) \to ((\bullet\bullet) r_g)$
whose associated cobordism is the identity of $(g+2)$. So, we obtain
$$
\Ztilde(K \cdot L) = (\Ztilde(\mu) \otimes \Id_g) \circ \Ztilde(P_{\bullet,\bullet,r_g}) 
\circ (\Id_1 \otimes \Ztilde(K)) \circ \Ztilde(L).
$$
An application of \cite[Lemma 5.5]{CHM} shows that $\Ztilde(\mu)$ is congruent modulo $\homotopy(1^-)$ to
$$
p:= \chi^{-1} \left(\begin{array}{c}
{\labellist \small \hair 2pt
\pinlabel {$1^-$} [l] at 182 10
\pinlabel {$1^+$} [b] at 33 150
\pinlabel {$2^+$} [b] at 153 150
\endlabellist
\includegraphics[scale=0.55]{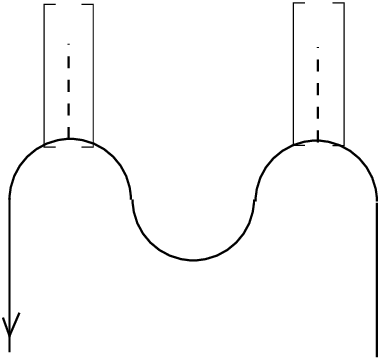}}
\end{array}\ \right)
$$
where the brackets $[-]$ stand for exponentials. Thus, we obtain that
$$
\Ztilde(K \cdot L) \equiv (p \otimes \Id_g) \circ \Ztilde(P_{\bullet,\bullet,r_g}) 
\circ (\Id_1 \otimes \Ztilde(K)) \circ \Ztilde(L) \mod \homotopy(1^-).
$$
Using the functoriality of $\Ztilde$, it is easily checked that
$$
\left(\Ztilde(P_{\bullet,\bullet,r_g})\left| 1^- \mapsto 0\right.\right)
= \varnothing \otimes \Id_{g+1}
\quad \hbox{and} \quad 
\left(\Ztilde(P_{\bullet,\bullet,r_g})\left| 2^- \mapsto 0\right.\right)
= \Id_1 \otimes \varnothing \otimes \Id_{g}
$$
from which it follows that
$$
\left(\ZtildeY(P_{\bullet,\bullet,r_g})\left| 1^- \mapsto 0\right.\right) = \varnothing 
\quad \hbox{and} \quad 
\left(\ZtildeY(P_{\bullet,\bullet,r_g})\left| 2^- \mapsto 0\right.\right) = \varnothing.
$$
Since $\ZtildeY(P_{\bullet,\bullet,r_g})$ is group-like, 
we deduce that each diagram of this series displays 
the color $1^-$ as well as the color $2^-$, on each of its components. Therefore, we have
\begin{eqnarray}
\label{eq:K_L}
&& \Ztilde(K \cdot L) \mod \homotopy(1^-)\\
\nonumber  &=& (p \otimes \Id_g) \circ (\Id_1 \otimes \Ztilde(K)) \circ \Ztilde(L) \mod \homotopy(1^-) \\
\nonumber &=&  (p \otimes \Id_g) \circ (\Id_1 \otimes (\Ztilde(K) \mod \homotopy(1^-))) 
\circ (\Ztilde(L) \mod \homotopy(1^-)) \mod \homotopy(1^-). 
\end{eqnarray}
For all $B\in \bottomK$, we set
$$
\varkappa(B):= \log_\sqcup\left(\Ztilde(B) \mod \homotopy(r)\right) - \sum_{i=1}^g  \strutgraph{i^-}{i^+}
 \quad \in \Liehat(H) 
$$
where $\Ztilde(B)$ is seen as an element of $\A(\set{g}^+ \cup \set{g}^-\cup \{ r\})$
by the change of variables (\ref{eq:change_variables})
and $\Liehat(H)$ is seen as a subspace of $\A(\set{g}^+ \cup \set{g}^-\cup \{ r\})$ by (\ref{eq:inclusion_Lie}).
By definition of $\thetaLMO$ we have
\begin{equation}
\label{eq:kappa_to_theta}
\thetaLMO(\ell(B)) = \exp_\otimes ({\varkappa(B)}).
\end{equation}
By expliciting the composition law in the category $\tsA$, we derive from (\ref{eq:K_L}) the following identity:
$$
\Ztilde(K \cdot L)  \mod \homotopy(r)
= \begin{array}{|c|} 
\hline
\left(\exp_\sqcup\varkappa(L)| r \mapsto 1^*\right) \sqcup 
\left(\exp_\sqcup\varkappa(K)| r \mapsto 2^*\right)    \\
\hdashline 
\{1^*,2^*\} \\
\hdashline 
\left(p\left| \begin{array}{l}1^+ \mapsto 1^*\\ 2^+ \mapsto 2^*\\ 1^-\mapsto r \end{array} \right.\right) \\
\hline
\end{array} \sqcup \Id_g.
$$
Here, the array means that the bottom row is ``contracted'' to the top row
with respect to the set of their common variables, which appears in the middle row. 
(This contraction is denoted by $\langle - , - \rangle_{\{1^*,2^*\}}$ in \cite{CHM}.) Therefore, we have
\begin{eqnarray*}
\Ztilde(K \cdot L) \mod \homotopy(r)&=& 
\chi^{-1}_r\left(\chi_r \exp_{\sqcup}\varkappa(K) \cdot \chi_r \exp_{\sqcup}\varkappa(L)  \right)\sqcup \Id_g  \\
& = & \chi^{-1}_r\left(\exp_{\cdot} \chi_r \varkappa(K) \cdot \exp_{\cdot} \chi_r \varkappa(L)  \right)\sqcup \Id_g 
\end{eqnarray*} 
where, in the last two terms, the dot $\cdot$ denotes the multiplication in $\A(\uparrow^{r},\set{g}^+\cup \set{g}^-)$ along $\uparrow^r$,
and $\chi_r$ is the  Poincar\'e--Birkhoff--Witt isomorphism from $\A(\set{g}^+\cup \set{g}^-\cup \{r\})$ 
to $\A(\uparrow^{r},\set{g}^+\cup \set{g}^-)$. We deduce that
$$
\exp_\sqcup \varkappa(K\cdot L) = 
\chi^{-1}_r\left(\exp_{\cdot} \chi_r \varkappa(K) \cdot \exp_{\cdot} \chi_r \varkappa(L)  \right)
$$
or, equivalently, that
$$
\exp_\cdot \chi_r \varkappa(K\cdot L) = 
\chi_r \exp_\sqcup \varkappa(K\cdot L) = \exp_{\cdot} \chi_r \varkappa(K) \cdot \exp_{\cdot} \chi_r \varkappa(L).
$$
By (\ref{eq:kappa_to_theta}), we conclude that
$$
\thetaLMO\left(\ell(K) \cdot \ell(L)\right) = \thetaLMO(\ell(K)) \otimes \thetaLMO(\ell(L))
$$ 
and that $\thetaLMO$ is a monoid homomorphism.

Next, for any bottom knot $B$, we deduce from \cite[Lemma 4.12]{CHM} that 
$$
\Ztilde(B) = \exp_\sqcup\left(\strutgraph{r}{b}+\sum_{i=1}^g\strutgraph{i^-}{i^+}\right) \sqcup \ZtildeY(B)
$$
where $b$ is the homology class of $B$ in $\Sigma \times [-1,1]$ written as
$$
H \ni  b = \sum_{i=1}^g (x_i \cdot a_i + y_i \cdot b_i) =
\sum_{i=1}^g (x_i \cdot i^- + y_i \cdot i^+)
$$ 
for some $x_1,y_1,\dots,x_g,y_g\in \Q$. So, we have
$$
\varkappa(B) = \strutgraph{r}{b} + (\ideg\geq 1) \mod \homotopy(r)
$$
and we deduce that
$$
\thetaLMO(\ell(B)) = \exp_\otimes\left(\varkappa(B)\right) = 1 + b + (\deg \geq 2).
$$

\begin{figure}[h]
\begin{center}
\labellist \small \hair 2pt
\pinlabel {$B_g$} [tr] at 129 55
\endlabellist
\includegraphics[scale=0.3]{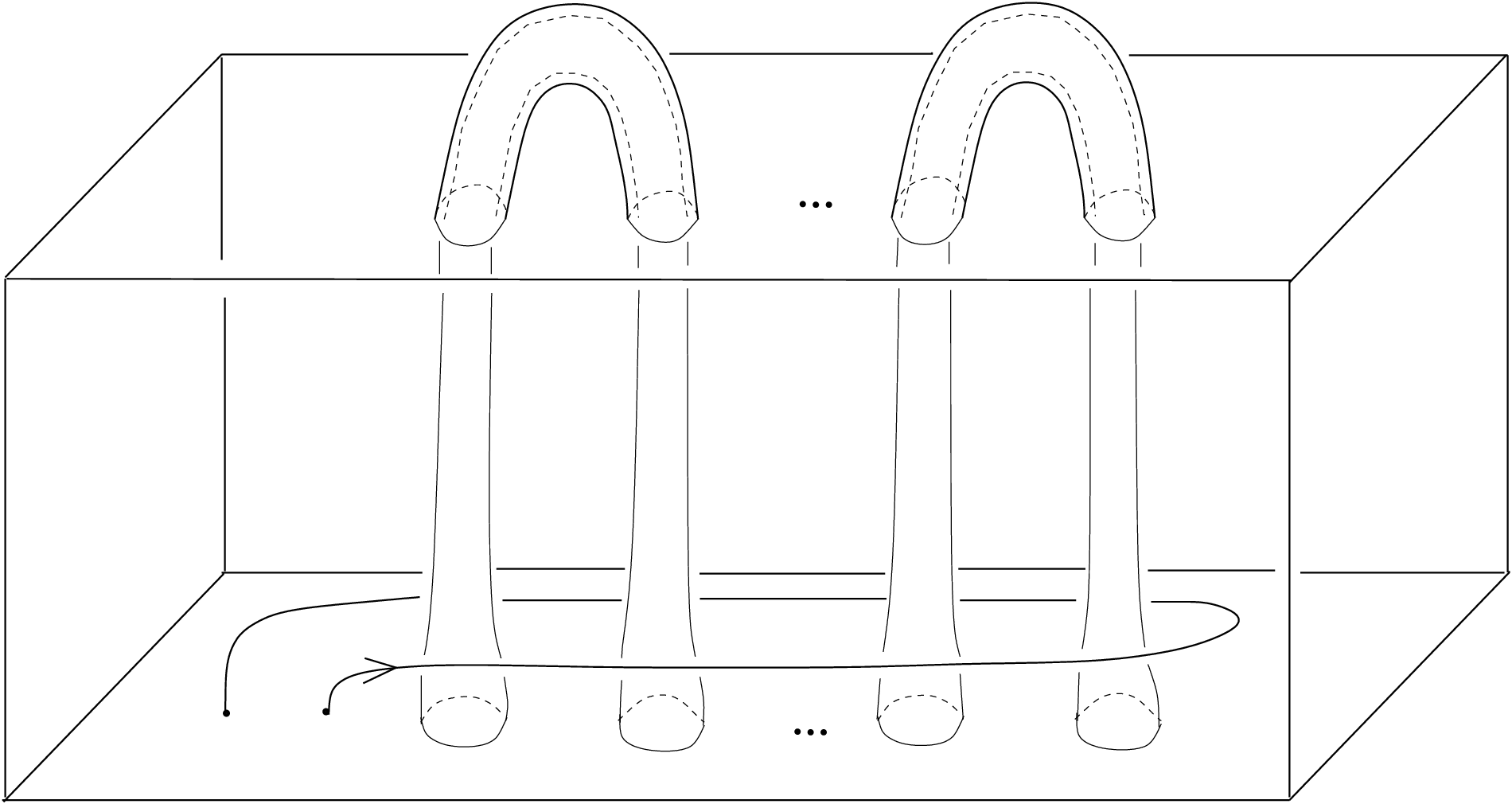}
\end{center}
\caption{A bottom knot $B_g$ such that $\ell(B_g)=\zeta$.}
\label{fig:boundary}
\end{figure}

It remains to prove that $\thetaLMO(\zeta)= \exp(-\omega)$. 
The element  $\zeta\in \pi$ can be represented by the bottom knot $B_g$ shown on Figure \ref{fig:boundary}.
Next,  an application of \cite[Lemma 5.5]{CHM} gives 
$$
\Ztilde(B_g) = \chi^{-1}\left( \begin{array}{c}
\labellist \small \hair 2pt
\pinlabel{$Z(L_g)$} at 156 50
\pinlabel {$\cdots$} at 158 128
\pinlabel {$1^+$}  at 58 227
\pinlabel {$g^+$}  at 258 227
\endlabellist
\includegraphics[scale=0.4]{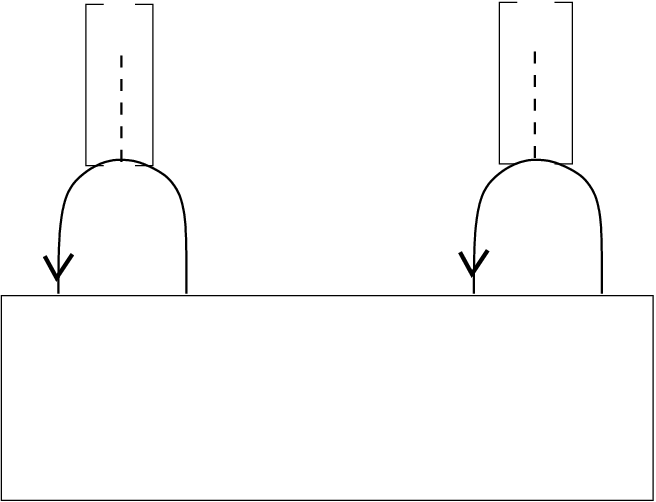}
\end{array}\right)
$$
where $Z(L_g)$ is the Kontsevich integral of the framed oriented $q$-tangle $L_g$ shown on Figure \ref{fig:Lg}.
The same figure states that the $q$-tangle  $L_g$ is obtained 
by cabling the vertical component of a certain $q$-tangle $L$. So, we have
$$
Z(L_g) = \Delta^+_{(r_g|\bullet \mapsto (+-))} \left(Z(L)\right)
$$ 
where, for any non-associative word $w$ in the letters $\{+,-\}$,
$\Delta^+_w$ denotes the usual ``doubling''/``orientation-reversal'' map as recalled in \cite[Notation 3.13]{CHM}.

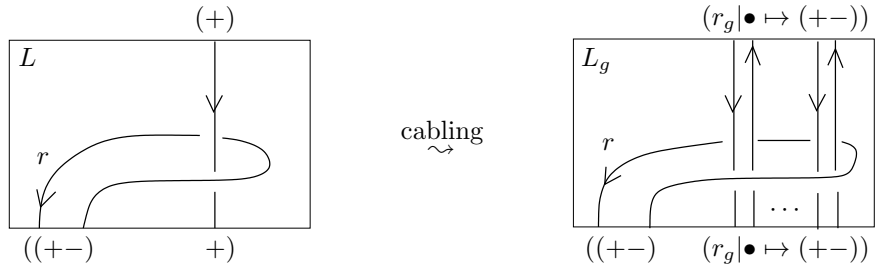
\begin{figure}[h]
\begin{center}
\labellist \small \hair 2pt
\pinlabel {$\stackrel{\hbox{cabling}}{\leadsto}$} at 327 71
\pinlabel{$\cdots$} at 587 16
\pinlabel{$r$} [br] at 34 48
\pinlabel{$r$} [br] at 460 58
\pinlabel{$L$} [tl] at 5 141
\pinlabel{$L_g$} [tl] at 429 141
\pinlabel{$((+-)$} [t] at 38 0 
\pinlabel{$+)$} [t] at 160 0
\pinlabel{$(+)$} [b] at 156 146
\pinlabel{$((+-)$} [t] at 462 0 
\pinlabel{$(r_g|\bullet \mapsto (+-))$} [t] at 586 0 
\pinlabel{$(r_g|\bullet \mapsto (+-))$} [b] at 586 146
\endlabellist
\includegraphics[scale=0.5]{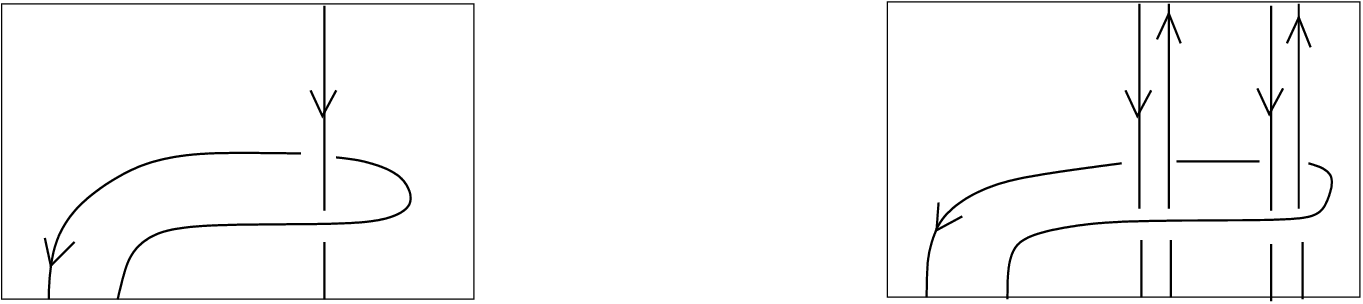}
\end{center}
\caption{The framed oriented $q$-tangle $L_g$ as a cabling.}
\label{fig:Lg}
\end{figure}

\noindent
By decomposing $L$ into elementary $q$-tangles, it is easily checked that
$$
Z(L) \equiv \begin{array}{c}
\labellist \small \hair 2pt
\pinlabel{$\shortmid$} at 161 120
\pinlabel {$r$} [r] at 0 48
\endlabellist
\includegraphics[scale=0.4]{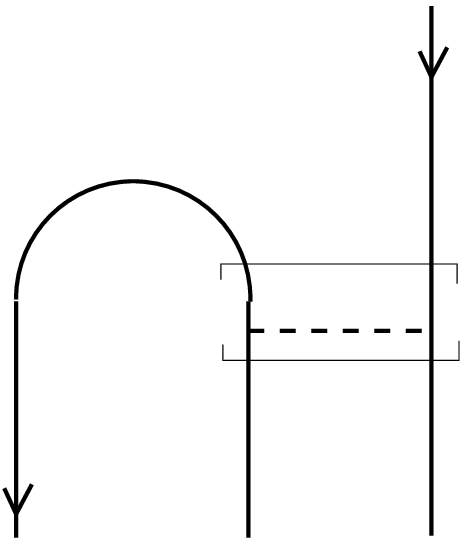}
\end{array}
\mod \homotopy(\curvearrowleft^r),
$$
where we notify of a (vertical) minus sign inside the exponential $[-]$ . We deduce that
\begin{equation}
\label{eq:}
\Ztilde(B_g) \equiv \chi^{-1}\left(\begin{array}{c}
\labellist \small \hair 2pt
\pinlabel{$\shortmid$} at 310 75
\pinlabel {$\cdots$}  at  310 180
\pinlabel {$r$} [t] at  66 0
\pinlabel{$1^-$} [t] at 237 3
\pinlabel{$g^-$} [t] at 436 3
\pinlabel{$1^+$} at 206 328
\pinlabel{$g^+$} at 406 328
\endlabellist
\includegraphics[scale=0.4]{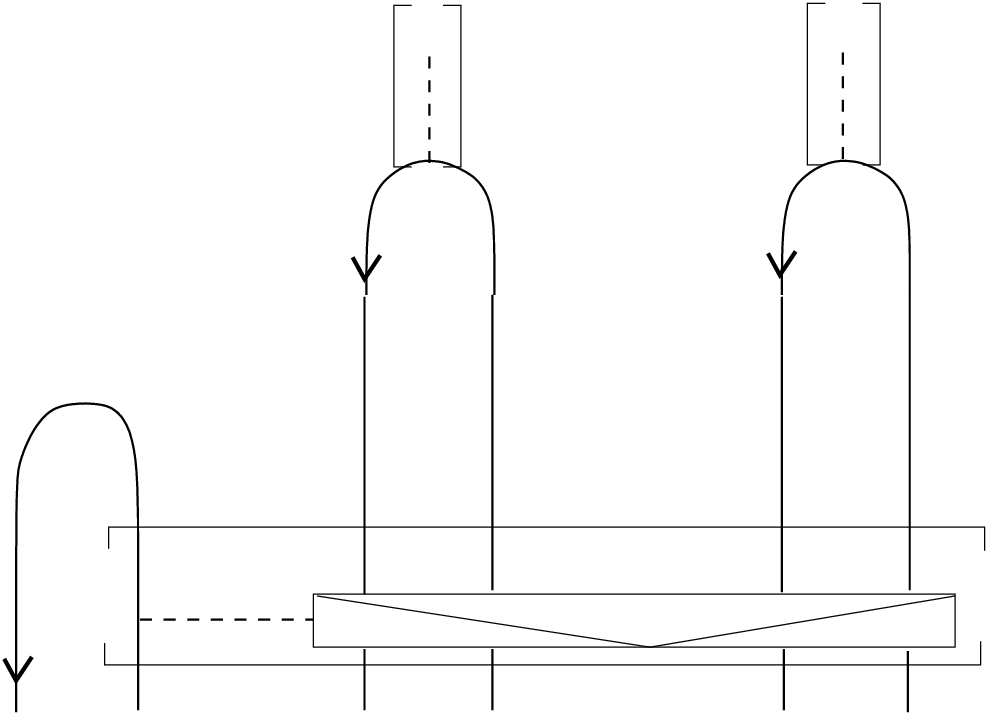}
\end{array}\right)
\mod \homotopy(r).
\end{equation}
By developing the $(g+1)$ exponentials and by applying $\chi^{-1}$, we obtain\\

$$
\Ztilde(B_g) \equiv  \sum_{m} \frac{1}{m!}  \quad \quad \begin{array}{c}
\labellist \small \hair 2pt
\pinlabel{$\vdots$} at 92 140
\pinlabel {$\cdots$}  at  228 66
\pinlabel {$\cdots$}  at  228 208
\pinlabel {$\cdots$}  at  163 58
\pinlabel {$\cdots$}  at  295 58
\pinlabel{$1^-$} [t] at 146 45
\pinlabel{$1^-$} [t] at 185 45
\pinlabel{$g^-$} [t] at 278 45
\pinlabel{$g^-$} [t] at  318 45
\pinlabel{$1^+$} [b] at 146 229
\pinlabel{$1^+$} [b] at 185 229 
\pinlabel{$g^+$} [b] at 278 230 
\pinlabel{$g^+$} [b] at 316 230 
\pinlabel{$r$} [r] at  55 171 
\pinlabel{$r$} [r] at 55 98 
\pinlabel{$m_0$} [r] at 0 138 
\pinlabel{$m_1$} [t] at 162 0
\pinlabel{$m_g$} [t] at 292 0
\endlabellist
\includegraphics[scale=0.4]{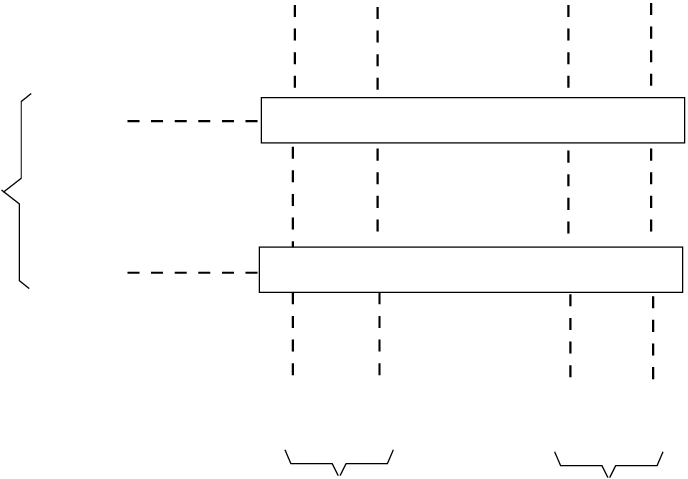}
\end{array} \mod \homotopy(r)
$$
\vspace{0.1cm}

\noindent
where the sum is over all $(g+1)$-uplets of non-negative integers $m=(m_0,m_1,\dots,m_g)$.
This can be reduced modulo $\homotopy(r)$ to
$$
\sum_m  \frac{1}{m!} \sum_{\substack{n_1\leq m_1,\dots,   n_g\leq m_g\\ n_1 + \cdots + n_g = m_0}}
\prod_{i=1}^g n_i! \cdot \prod_{i=1}^g \binom{m_i}{n_i} \cdot \binom{m_0}{n_1,\dots,n_g} 
\cdot \bigsqcup_{i=1}^g {\Ygraphbottoptop{r}{i^+}{i^-}}^{_\sqcup n_i} \sqcup 
\bigsqcup_{i=1}^g {\strutgraph{i^-}{i^+}}^{\ \sqcup (m_i-n_i)}.
$$
Thus, we obtain
$$
\Ztilde(B_g) \equiv \exp_\sqcup\left(\sum_{i=1}^g \strutgraph{i^-}{i^+}\right) 
\sqcup \exp_\sqcup\left(\sum_{i=1}^g \Ygraphbottoptop{r}{i^+}{i^-}\right)  \mod \homotopy(r),
$$
which is equivalent to
$$
\varkappa(B_g) = \sum_{i=1}^g \Ygraphbottoptop{r}{i^+}{i^-}
= \sum_{i=1}^g [b_i,a_i] = - \omega \ \in \Liehat(H). 
$$
This concludes the proof of the symplecticity of $\thetaLMO$.
\end{proof}

\begin{remark}
The construction of the LMO functor $\Ztilde$ in \cite{CHM} assumes two choices:
one has to fix a system of meridians and parallels $(\alpha,\beta)$ on the surface $\Sigma$ 
and, since the definition of $\Ztilde$ is based on the Kontsevich integral, one has to specify a Drinfeld associator.
Therefore, the symplectic expansion $\thetaLMO$ should depend on those two choices.
\end{remark}

\subsection{The total Johnson map defined by the LMO functor}

The total Johnson map relative to the symplectic expansion $\thetaLMO$ of $\pi$ is denoted by
$$
\tauLMO: \cyl \longrightarrow \Hom(H,\Liehat_{\geq 2})
$$
and, as explained at the end of \S \ref{subsec:extension}, this map is equivalent to a monoid homomorphism
$$
\rhoLMO : \cyl \to \IAut_{\omega}(\Liehat).
$$
Besides, the \emph{LMO homomorphism} is defined\footnote{ 
Homology cylinders over $\Sigma$ are equipped in \cite{CHM} with left-handed non-associative words
$\left(\cdots \left( \left(\bullet \bullet\right) \bullet \right) \cdots \bullet\right)$. 
Here, we prefer to equip them with  right-handed non-associative words, 
which affects none of the properties shown in \cite{CHM} for $\ZtildeY$.}
in \cite{CHM} as the ``Y-part'' of the LMO functor restricted to $\cyl$, and this is a monoid homomorphism
$$
\ZtildeY : \cyl \longrightarrow \A^{Y}(\set{g}^+ \cup \set{g}^-).
$$
We recall that the product $\star$ on $\A^{Y}(\set{g}^+ \cup \set{g}^-)$ 
is defined on Jacobi diagrams $D,E$ by
\begin{equation}
\label{eq:star}
D \star E := \left( \begin{array}{c} 
\hbox{sum of all ways of connecting some $i^+$-colored vertices}\\
\hbox{of $D$ to some $i^-$-colored vertices of $E$, for all $i=1,\dots,g$}
\end{array}\right).
\end{equation}
The values of $\ZtildeY$ are group-like and, 
so, they are invertible for the multiplication $\star$.
Actually, we will only need the tree-reduction of $\ZtildeY$
$$
\ZtildeYt : \cyl \longrightarrow \A^{Y,t}(\set{g}^+ \cup \set{g}^-)
$$
with values in the quotient of $\A^{Y}(\set{g}^+ \cup \set{g}^-)$
by the subspace generated by looped Jacobi diagrams.
This subspace being a Hopf ideal,  $\A^{Y,t}(\set{g}^+ \cup \set{g}^-)$ is a quotient Hopf algebra.

The next result shows that the tree-reduction of the LMO homomorphism determines 
the total Johnson map relative to the symplectic expansion $\thetaLMO$.

\begin{theorem}
\label{th:LMO_to_Kawazumi}
Let $C \in \cyl$ and let $Z := \ZtildeYt(C)$. 
We  denote by $Z^{-1}$ the inverse of $Z$ 
with respect to the multiplication $\star$. Then, for all $y \in \Liehat(H)$, we have
\begin{equation}\label{eq:LMO_to_Kawazumi}
\rhoLMO(C)(y)= \log_{\sqcup} \left( Z \star \exp_{\sqcup}(y) \star Z^{-1} \mod \homotopy(r) \right) \ \in \Liehat(H). 
\end{equation}
Here, $\Liehat(H)$ is seen as a subspace of $\A(\set{g}^+ \cup \set{g}^-\cup \{ r\})$ by inclusion (\ref{eq:inclusion_Lie}), 
and $\star$ is the multiplication on $\A(\set{g}^+ \cup \set{g}^-\cup \{ r\})$ defined by (\ref{eq:star}).
\end{theorem} 

\noindent
This result is inspired by the work of Habegger and Masbaum  \cite[\S 12]{HM}.
In this work, they first show that the Kontsevich integral defines an expansion 
of the free group of rank $n$, with respect to which one can compute Milnor's $\mu$ invariants 
of an $n$-strand string link $\beta$. 
Then, they prove a ``global formula'' giving Milnor's $\mu$ invariants of $\beta$ in terms of its Kontsevich integral.
Theorem \ref{th:LMO_to_Kawazumi} is very close in spirit to that formula.

\begin{proof}[Proof of Theorem \ref{th:LMO_to_Kawazumi}]
We start with the case when $C$ is the mapping cylinder of an $h\in \I$, which is  simpler.
For all bottom knots $K$ in $\Sigma \times [-1,1]$, we have the identity  
$$
(\Id_1 \otimes h) \circ K \circ h^{-1} = h(K) \circ h \circ h^{-1} = h(K)
$$
in the category of cobordisms $\LCob$, hence the identity 
$$
\Ztilde(h(K)) = (\Id_1 \otimes \Ztilde(h)) \circ \Ztilde(K) \circ \Ztilde(h^{-1})
$$
in the category of diagrams $\tsA$.  By expliciting the composition law of  $\tsA$, this writes
$$
\Ztilde(h(K)) =
\begin{array}{|c|}
\hline 
\left(\Ztilde(h^{-1})\left| j^- \mapsto j^*, \forall j=1,\dots,g \right.\right) \\
\hdashline 
\set{g}^*\\
\hdashline
\left(\Ztilde(K)\left| \begin{array}{l} r \mapsto s \\ 
 j^+ \mapsto j^*, \forall j=1,\dots,g \\
  j^- \mapsto j^\triangle , \forall j=1,\dots,g
\end{array} \right.\right)  \\
\hdashline
\{s\} \cup \set{g}^\triangle\\
\hdashline
\exp_\sqcup\left( \strutgraph{r}{s}\right) \sqcup 
\left(\Ztilde(h)\left|j^+ \mapsto j^\triangle, \forall j=1,\dots,g \right.\right)\\
\hline
\end{array}.
$$
By reducing modulo $\homotopy(r)$, we obtain
\begin{equation}
\label{eq:Z_h_K_reduced}
\Ztilde(h(K)) \mod \homotopy(r) =
\begin{array}{|c|}
\hline 
\exp_\sqcup\left({\displaystyle \sum_{j=1}^g \strutgraph{j^*}{j^+}}\right) \sqcup
\left(Z^{-1}\left| j^- \mapsto j^*, \forall j=1,\dots,g \right.\right) \\
\hdashline 
\set{g}^*\\
\hdashline
\left(\Ztilde(K) \mod \homotopy(r)\left| \begin{array}{l} r \mapsto s \\ 
 j^+ \mapsto j^*, \forall j=1,\dots,g \\
  j^- \mapsto j^\triangle , \forall j=1,\dots,g
\end{array} \right.\right)  \\
\hdashline
\{s\} \cup \set{g}^\triangle\\
\hdashline
\exp_\sqcup\left( \strutgraph{r}{s}+ {\displaystyle \sum_{j=1}^g \strutgraph{j^-} {j^\triangle}}\right) \sqcup 
\left(Z \left|j^+ \mapsto j^\triangle, \forall j=1,\dots,g \right.\right)\\
\hline
\end{array}.
\end{equation}
As before, we associate to each bottom knot $B$ the Lie series
$$
\varkappa(B):= \log_\sqcup\left(\Ztilde(B) \mod \homotopy(r)\right) - \sum_{i=1}^g  \strutgraph{i^-}{i^+}
 \quad \in \Liehat(H) \subset \A(\set{g}^+ \cup \set{g}^-\cup \{ r\}),
$$
which satisfies $\varkappa(B) = \thetaLMO \log \ell(B) \in \Liehat(H) \subset \Tenshat(H)$.  We have 
\begin{eqnarray*}
\varkappa h(K) &= &\thetaLMO \log \ell(h(K))\\
& =& \thetaLMO \log h_* \ell(K)\\
&=&  \thetaLMO \MLie(h_*) \log \ell(K) =  \rhoLMO(h) \thetaLMO  \log \ell(K).
\end{eqnarray*}
Then, equation (\ref{eq:Z_h_K_reduced}) gives 
the following identity where $y$ is the Lie series $\thetaLMO \log \ell(K)$ 
and where $C=h$ is assumed to belong to $\I$:

\begin{eqnarray}
\label{eq:log_pi}
&&\exp_\sqcup\left(\sum_{j=1}^g \strutgraph{j^-}{j^+}\right) 
\sqcup \exp_\sqcup \left(\rhoLMO(C)(y) \right) \mod \homotopy(r) \\
\nonumber & = &
\begin{array}{|c|}
\hline 
\exp_\sqcup\left({\displaystyle \sum_{j=1}^g \strutgraph{j^*}{j^+}}\right) \sqcup
\left(Z^{-1}\left| j^- \mapsto j^*, \forall j=1,\dots,g \right.\right) \\
\hdashline 
\set{g}^*\\
\hdashline
\exp_\sqcup\left({\displaystyle \sum_{j=1}^g \strutgraph{j^\triangle}{j^*}}\right)
\sqcup \left(\exp_\sqcup(y) \left| \begin{array}{l} r \mapsto s \\ 
 j^+ \mapsto j^*, \forall j=1,\dots,g \\
 j^- \mapsto j^\triangle , \forall j=1,\dots,g
\end{array} \right.\right)  \\
\hdashline
\{s\} \cup \set{g}^\triangle\\
\hdashline
\exp_\sqcup \left( \strutgraph{r}{s}+ {\displaystyle \sum_{j=1}^g \strutgraph{j^-} {j^\triangle}}\right) \sqcup 
\left(Z \left|j^+ \mapsto j^\triangle, \forall j=1,\dots,g \right.\right)\\
\hline
\end{array}
\end{eqnarray}
\vspace{0.1cm}

Let us now prove formula (\ref{eq:log_pi}) for any homology cylinder $C$.
Our arguments are based on Goussarov and Habiro's calculus of claspers, 
which allows us to proceed as before \emph{up to} a fixed degree $d \geq 1$.
Fundamental in their works is the notion of $Y_d$-equivalence for $3$-manifolds
\cite{Goussarov_clovers,Habiro,Goussarov_note,GGP}.
In particular, they prove that there exists a homology cylinder $\overline{C}$ (depending on $d$)
such that $\overline{C} \circ C$ and $C \circ \overline{C}$ are $Y_d$-equivalent to the trivial cylinder. 
Thus, there exists a disjoint union $G$  of  connected tree claspers with $d$ nodes in $\Sigma \times [-1,1]$,
such that  surgery along $G$ transforms $\Sigma \times [-1,1]$ to  $C \circ \overline{C}$.
Let us consider the Lagrangian cobordism
\begin{equation}
\label{eq:def_L}
L := (\Id_1 \otimes C) \circ K \circ \overline{C}.
\end{equation}
If we regard $L$ as a connected framed oriented tangle  in the homology cylinder $C \circ \overline{C}$,
we see that there exists a bottom knot $L'$ in $\Sigma \times [-1,1]$ disjoint from $G$
such that  surgery along $G$ transforms $L'$ to $L$.
On the one hand, a property of the LMO functor shown in \cite{CHM}
implies that $\Ztilde(L')$ differs from $\Ztilde(L)$ by a series of Jacobi diagrams with internal  degree $\geq d$.
So, we have
\begin{equation}
\label{eq:one}
\kappa(L') \equiv \log_\sqcup\left(\Ztilde(L) \mod \homotopy(r)\right) - \sum_{i=1}^g  \strutgraph{i^-}{i^+}
\mod \Liehat_{\geq d+1}.
\end{equation}
On the other hand, $L' \circ C$ is $Y_d$-equivalent to $L\circ C$
and so is $Y_d$-equivalent to $(\Id_1 \otimes C) \circ K$, from which we deduce that
$$
\{ \ell(L') \} = \rho_d(C)\big(\{ \ell(K)\}\big) \in \pi/\Gamma_{d+1} \pi.
$$
Then, a short computation gives
\begin{equation}
\label{eq:other}
\kappa(L') \equiv \rhoLMO(C) \thetaLMO \log \ell(K) \mod \Liehat_{\geq d+1}.
\end{equation}
By comparing (\ref{eq:one}) to (\ref{eq:other}), we get
\begin{equation}
\label{eq:Z_L_reduced}
\Ztilde(L) \mod \homotopy(r) = 
\exp_{\sqcup}\left(\sum_{j=1}^g \strutgraph{j^-}{j^+}\right) 
\sqcup \exp_\sqcup\left( \rhoLMO(C) \thetaLMO \log \ell(K) \right) + (\ideg \geq d).
\end{equation}
If we now come back to (\ref{eq:def_L}) and apply to it the LMO functor, we get
\begin{equation}
\label{eq:Z_L}
\Ztilde(L) = \left(\Id_1 \otimes \Ztilde(C)\right) \circ \Ztilde(K) \circ \Ztilde(\overline{C}).
\end{equation}
Since $\overline{C}$ is inverse to $C$ up to $Y_d$-equivalence,
$\ZtildeY(\overline{C})$ is inverse to $\ZtildeY(C)$ with respect to the $\star$ product up to some terms of internal degree at least $d$.
So, $\ZtildeYt(\overline{C})$ is equal to $Z^{-1}$ modulo  diagrams of internal degree  at least $d$.
Then, by reducing (\ref{eq:Z_L}) modulo $\homotopy(r)$ and by using (\ref{eq:Z_L_reduced}),
we obtain equation (\ref{eq:log_pi})  modulo diagrams of internal degree at least $d$.
By passing to the limit $d\to + \infty$,
we conclude that (\ref{eq:log_pi}) is valid for any homology cylinder $C$ and for $y=\thetaLMO \log \ell(K)$.

Next, since the bottom knot $K$ is arbitrary in the above discussion,
formula (\ref{eq:log_pi}) holds true for any $C \in \cyl$ and for any $y\in \thetaLMO \log(\pi)$.
Besides, we observe that both sides of (\ref{eq:LMO_to_Kawazumi}) are linear in $y$,
the right side being the connected part of $Z\star y\star Z^{-1} \mod \homotopy(r)$.
We deduce from this observation and  the next claim 
that it is enough to prove (\ref{eq:LMO_to_Kawazumi})  for any $y \in \thetaLMO \log(\pi)$.

\begin{claim}
\label{claim:generation_Malcev_Lie_algebra}
Let $F$ be a finitely generated free group. 
Then, any element $x \in \MLie(F)$ can be written as
$$
x= \sum_{i=1}^{+ \infty} q_i x_i \quad
\hbox{where $x_i \in \widehat{\Gamma}_i \MLie(F) \cap \log(F)$ and $q_i\in \Q$}.
$$
\end{claim}

\begin{proof}[Proof of Claim \ref{claim:generation_Malcev_Lie_algebra}]
Let $(\gp{b}_1,\dots,\gp{b}_n)$ be a basis of $F$. By Example \ref{ex:basis}, 
$\MLie(F)$ is the complete free Lie algebra generated by $({b}_1,\dots,{b}_n)$ where $b_i:= \log(\gp{b}_i)$.
The Baker--Campbell--Hausdorff formula  implies that, for all $r\geq 1$ and
for all $i_1,\dots,i_r \in \{1,\dots,n\}$, 
$$
[b_{i_1}, [b_{i_2},[ \dots, b_{i_r}] \cdots ]] \equiv 
\log\left([\gp{b}_{i_1}, [\gp{b}_{i_2},[ \dots, \gp{b}_{i_r}] \cdots ]] \right)
\mod \widehat{\Gamma}_{r+1} \MLie(F).
$$
The claim is easily deduced from that fact.
\end{proof}

We now finish the proof by developing (\ref{eq:log_pi}):
$$
\exp_\sqcup\left(\sum_{j=1}^g \strutgraph{j^-}{j^+}\right) 
\sqcup \exp_\sqcup \left(\rhoLMO(C)(y) \right) \mod \homotopy(r)
$$
$$
= \begin{array}{|c|}
\hline 
\exp_\sqcup\left({\displaystyle \sum_{j=1}^g \strutgraph{j^*}{j^+}}\right) \sqcup
\left(Z^{-1}\left| j^- \mapsto j^*, \forall j=1,\dots,g \right.\right) \\
\hdashline 
\set{g}^*\\
\hdashline
\exp_{\sqcup}\left({\displaystyle \sum_{j=1}^g \strutgraph{j^\Delta}{j^*} }\right)
\sqcup
\left(\exp_\sqcup(y) \left| \begin{array}{l}
 j^+ \mapsto j^*, \forall j=1,\dots,g \\
 j^- \mapsto j^\triangle , \forall j=1,\dots,g
\end{array} \right.\right)  \\
\hdashline
 \set{g}^\triangle\\
\hdashline
\exp_\sqcup \left(  {\displaystyle \sum_{j=1}^g \strutgraph{j^-}{j^\triangle}}\right) \sqcup 
\left(Z \left|j^+ \mapsto j^\triangle, \forall j=1,\dots,g \right.\right)\\
\hline
\end{array}
$$
$$
= \exp_\sqcup\left({\displaystyle \sum_{j=1}^g \strutgraph{j^-}{j^+}}\right)
\sqcup \left(Z \star \exp_{\sqcup}(y) \star Z^{-1} \right).
$$
The last identity is obtained by an easy combinatorial argument \cite[Example 4.5]{CHM}.
We conclude that 
$\rhoLMO(C)(y)= \log_{\sqcup}\left(Z \star \exp_{\sqcup}(y) \star Z^{-1} \mod \homotopy(r)  \right)$.
\end{proof}

We can now prove that, on the submonoid $\cyl[k]$, the degree $[k,2k[$ part of $\ZtildeYt$
coincides after fission with the $k$-th infinitesimal Morita homomorphism.
For this, let us note that the space $\T(H)$ defined in \S \ref{subsec:fission}
embeds into  $\A^{Y,t}(\set{g}^+ \cup \set{g}^-)$ by 
$$
T \longmapsto \left(T \left| 
\begin{array}{l} 
b_i \mapsto i^+, \ \forall i=1,\dots,g \\ 
a_i \mapsto i^-, \ \forall i=1,\dots,g
\end{array}
\right. \right).
$$
In that way, $\T(H)$ is identified with the subspace of $\A^{Y,t}(\set{g}^+ \cup \set{g}^-)$  
spanned by \emph{connected} tree-shaped Jacobi diagrams.

\begin{theorem}
\label{th:LMO_to_Kawazumi_truncated}
Let $C \in \cyl$ and let $k\geq 1$ be such that $\ZtildeYt_i(C)= 0$ for all $i\in [1,k[$.
Then, $C$ belongs to $\cyl[k]$ and we have
$$
\eta\left( \ZtildeYt_{[k,2k[}(C) \right) = - \tauLMO_{[k,2k[}(C)
\ \in \bigoplus_{i=k}^{2k-1} H \otimes \Lie_{i+1}.
$$
Equivalently, we have
$$
\Phi\left( \ZtildeYt_{[k,2k[}(C) \right)
=   \thetaLMO_*\left(m_k(C)\right) \ \in H_3(\Lie/\Lie_{\geq k+1})
$$
where $\thetaLMO_*$ is induced by the isomorphism $\thetaLMO: \MLie(\pi/\Gamma_{k+1}\pi) \to \Lie/\Lie_{\geq k+1}$ in homology.
\end{theorem}

\begin{proof}
We set $Z := \ZtildeYt(C)$, which writes
$$
Z = \varnothing + Z_k + \cdots + Z_{2k-1} + (\ideg\geq 2k).
$$
Since $\ZtildeY(C)$ is group-like, $Z$ is group-like and, so, is the exponential with respect to $\sqcup$  of a primitive element.
Since $2k-1$ is strictly less than $k+k$,  we deduce that $Z_k, \dots, Z_{2k-1}$ only consist of connected diagrams, 
so that the map $\eta$ or $\Phi$ can indeed be applied to $Z_{[k,2k[}$.
The same  argument shows that the inverse of $Z$ with respect to $\star$ can be written as follows:
$$
Z^{-1} = \varnothing - Z_k - \cdots - Z_{2k-1} + (\ideg\geq 2k).
$$
According to Theorem \ref{th:LMO_to_Kawazumi} , we  have
\begin{eqnarray*}
  &&  {\tauLMO(C)(b_i)} \mod \Liehat_{\geq 2k+1}\\
  & = &  \log_{\sqcup}\left( Z \star \exp_{\sqcup}\left(\strutgraph{r}{i^+}\right) \star Z^{-1}\right) -b_i + (\ideg\geq 2k) \mod \homotopy(r)\\
 & = & {\log_\sqcup\left( Z \star (Z^{-1} \vert i^- \mapsto i^- +r)\right)  + (\ideg\geq 2k) \mod \homotopy(r)}\\
& = & (Z_k+\cdots + Z_{2k-1}) - (Z_k + \cdots + Z_{2k-1}|i^- \mapsto i^- + r ) \mod \homotopy(r)  \\
 & = & - (Z_k + \cdots + Z_{2k-1}|i^- \mapsto  r \hbox{ \dots exactly one time !} ) 
\end{eqnarray*}
and, similarly, we have
\begin{eqnarray*}
 && {\tauLMO(C)(a_i)} \mod \Liehat_{\geq 2k+1}\\
& = & {\log_\sqcup\left( (Z \vert i^+ \mapsto  i^+ +r) \star Z^{-1}\right) + (\ideg\geq 2k) \mod \homotopy(r)}  \\
 & = &  (Z_k+\cdots + Z_{2k-1}|i^+\mapsto i^+ +r) - (Z_k + \cdots + Z_{2k-1}) \mod \homotopy(r) \\
 & = & (Z_k + \cdots + Z_{2k-1}|i^+ \mapsto  r \hbox{ \dots exactly one time !} ).
\end{eqnarray*}
Thus, we conclude that 
$$
\tauLMO(C) \mod H \otimes \Liehat_{\geq 2k+1} = - \eta\left( Z_{[k,2k[} \right).
$$
In particular, this shows that $C$ belongs to $\cyl[k]$.
The second statement follows from Theorem \ref{th:Kawazumi_to_Morita}.
\end{proof}

As a consequence of Theorem \ref{th:Kawazumi_to_Johnson}, 
we recover the following result from \cite{CHM}.

\begin{corollary}
\label{cor:LMO_to_Johnson}
Let $C\in \cyl$. The lowest degree non-trivial term of $\ZtildeYt(C)$
coincides with the opposite of the first non-trivial Johnson homomorphism of $C$.
\end{corollary}

\noindent
The proof of Corollary \ref{cor:LMO_to_Johnson} given in \cite{CHM} is indirect: 
based on Habegger's correspondence
between Johnson homomorphisms and Milnor's $\mu$ invariants  \cite{Habegger}, 
it uses the connection between the latter invariants and the Kontsevich integral \cite{HM}.

To conclude this paper, we give a restatement of Theorem \ref{th:LMO_to_Kawazumi}
in which it is clear that the total Johnson map relative to $\thetaLMO$ 
tantamounts to the tree-reduction of the LMO homomorphism. 
For this, we need the following statement, 
which is well-known and is true for any Lie algebra equipped with a complete filtration.

\begin{proposition}
\label{prop:derivations_and_automorphisms}
There is a canonical bijection between the set of 
filtration-preserving derivations of $\Liehat(H)$ with values in $\Liehat_{\geq 2}(H)$,
and the set of filtration-preserving automorphisms of $\Liehat(H)$ 
that induce the identity at the graded level:
$$
\Der(\Liehat,\Liehat_{\geq 2})
\overset{\exp_\circ}{\underset{\log_\circ}{\longrightleftarrows}} \IAut(\Liehat).
$$
\end{proposition} 

\begin{proof}
We follow \cite{Praagman} which deals with the (commutative)  associative algebra case.
Let $\psi \in  \IAut(\Liehat)$. Then, $\widetilde{\psi}:= \psi - \Id$ is a $\Q$-linear map $\Liehat \to \Liehat$
which sends $\Liehat_{\geq n}$ to $\Liehat_{\geq n+1}$. 
Thus, the series $\sum_{n\geq 1} (-1)^{n+1}/n \cdot \widetilde{\psi}^n(x)$
 converges in $\Liehat$ for all $x \in \Liehat$, so that 
$$
\log_\circ (\psi) := \sum_{n\geq 1} \frac{(-1)^{n+1}}{n} \cdot \widetilde{\psi}^n
$$
defines a filtration-preserving $\Q$-linear map $\Liehat \to \Liehat$ valued into $\Liehat_{\geq 2}$. 
For any integer $n \geq 1$ and for all $i,j \in \Z$, we consider the integers $a_{i,j}^n \geq 0$ 
defined inductively by the relation $a_{i+1,j+1}^{n+1}= a_{i+1,j}^n + a_{i,j+1}^n + a_{i,j}^n$ starting with
$$ 
a^{1}_{i,j} = \left\{\begin{array}{ll}
1 & \hbox{if } (i,j)\in \{(1,0),(0,1),(1,1)\},\\
0 & \hbox{otherwise}.
\end{array}\right.
$$
These  integers $a_{i,j}^n$ satisfy
$$
\widetilde{\psi}^n\left([x,y]\right) = \sum_{i,j\in \Z}
a_{i,j}^n \cdot \left[\widetilde{\psi}^i(x),\widetilde{\psi}^j(y)\right] \quad \forall x,y \in \Liehat.
$$
In addition, the integers $a_{i,j}^n$ can be defined by the formula
$$
(X+Y+XY)^n = \sum_{i,j\in \Z} a_{i,j}^n \cdot  X^i Y^j \ \in \Q[X,Y].
$$
Then, the identity $\log(1+X+Y+XY)= \log(1+X) + \log(1+Y)$ 
implies certain linear relations among the trinomial coefficients $a_{i,j}^n$ with $(i,j)$ fixed.
We deduce from these relations that $\log_\circ(\psi)$ is a derivation.

Conversely, let $\delta \in \Der(\Liehat,\Liehat_{\geq 2})$. 
Then, $\delta$ sends $\Liehat_{\geq n}$ to $\Liehat_{\geq n+1}$.
Thus, the series $\sum_{n\geq 0} 1/n! \cdot \delta^n(x)$ converges in $\Liehat$ 
for all $x \in \Liehat$, so that 
$$
\exp_\circ (\delta) := \sum_{n\geq 0} \frac{1}{n!} \cdot \delta^n
$$
defines a filtration-preserving $\Q$-linear map $\Liehat \to \Liehat$ which is the identity at the graded level
and, so, is an isomorphism.
It is easily checked that $\exp_\circ(\delta)$ preserves the Lie bracket.
\end{proof}

\noindent
Consequently, we have a one-to-one correspondence
$$
\Der_\omega(\Liehat,\Liehat_{\geq 2})
\mathop{\longrightleftarrows}^{\exp_\circ}_{\log_\circ}
\IAut_\omega(\Liehat)
$$
between derivations that vanish on $\omega$ and automorphisms that fix $\omega$. 
Moreover, the canonical isomorphism
$$
\Der(\Liehat,\Liehat_{\geq 2}) \simeq \Hom(H,\Liehat_{\geq 2}) \simeq H \otimes \Liehat_{\geq 2}
$$
sends $\Der_\omega(\Liehat,\Liehat_{\geq 2})$ to the kernel of the Lie bracket.
Then, using the isomorphism $\eta$ defined at (\ref{eq:eta}), we obtain an isomorphism 
$$
\eta: \T(H) \stackrel{\simeq}{\longrightarrow} \Der_\omega(\Liehat,\Liehat_{\geq 2})
$$
which appears in Kontsevich's work \cite{Kontsevich_Gelfand,Kontsevich_ECM}
and where $\T(H)$ stands here for its degree completion.
So, given any symplectic expansion $\theta$ of $\pi$, we can consider the composition
\begin{equation}
\label{eq:tree_invariant}
\cyl \stackrel{\varrho^\theta}{\longrightarrow} \IAut_\omega(\Liehat)
\mathop{\longrightarrow}^{\log_\circ}_\simeq \Der_\omega(\Liehat,\Liehat_{\geq 2})
\mathop{\longrightarrow}^{\eta^{-1}}_\simeq \T(H).
\end{equation}

Besides, the LMO homomorphism takes values in the group-like part 
of the Hopf algebra $\A^Y(\set{g}^+\cup \set{g}^-)$, whose product $\star$ is defined at (\ref{eq:star}) 
and whose coproduct $\Delta$ is defined at (\ref{eq:coproduct}).
The same is true for the tree-reduction of the LMO homomorphism.
Thus, we can consider the composition
$$
\cyl \stackrel{\ZtildeYt}{\longrightarrow} \GLike( \A^{Y,t}(\set{g}^+\cup \set{g}^-) ) 
\mathop{\longrightarrow}^{\log_\star}_\simeq  
\Prim( \A^{Y,t}(\set{g}^+\cup \set{g}^-) ) =  \T(H).
$$
Our last result asserts that this map $\cyl \to \T(H)$ is an instance of (\ref{eq:tree_invariant}).

\begin{theorem}
\label{th:Kawazumi_to_LMO}
For all $C \in \cyl$, we have
$$
\log_\circ \rhoLMO (C) = - \eta \log_\star \ZtildeYt(C) \ \in \Der(\Liehat,\Liehat_{\geq 2}).
$$
\end{theorem}

\noindent
Thus $\ZtildeYt$ is essentially the same invariant of homology cylinders 
as the infinitesimal Dehn--Nielsen representation $\rhoLMO$ or, equivalently, as the total Johnson map $\tauLMO$.

\begin{proof}[Proof of Theorem \ref{th:Kawazumi_to_LMO}] 
We set $Z:= \ZtildeYt(C)$ and $z:=\log_\star(Z)$.
Let $y\in \Liehat$ which we can regard as an element of $\A(\set{g}^+ \cup \set{g}^- \cup \{r\})$.
Then, Theorem \ref{th:LMO_to_Kawazumi} gives
\begin{eqnarray*}
\rhoLMO(C)(y) & = &  \log_{\sqcup} \left(  \exp_\star(z) \star \exp_{\sqcup}(y) \star  \exp_\star(-z) \mod \homotopy(r) \right) \\
&=&  \hbox{\scriptsize non-empty connected part of }\ \exp_\star(z) \star \exp_{\sqcup}(y) \star \exp_\star(-z) \mod \homotopy(r) \\
&=&  \hbox{\scriptsize connected part of }\  \exp_\star(z) \star y \star \exp_\star(-z) \mod \homotopy(r)\\
&=&  \hbox{\scriptsize  connected part of }\  \sum_{n=0}^{+ \infty} \frac{1}{n!}
 \sum_{i=0}^n \binom{n}{i}\ {z_+}^{\star i} \star y \star {z_-}^{\star(n-i)} \mod \homotopy(r),
\end{eqnarray*}
where we have denoted $z_\pm := \pm z$.
Observe that there are $ \binom{n}{i}$ ways 
of parenthesizing the $n$-iterated product ${z_+}^{\star i} \star y \star {z_-}^{\star(n-i)}$ ``starting from the inside'' with $y$.
For instance, for $n=4$ and $i=2$, these $\binom{4}{2}=6$ ways are
$$
\big(((z_+ \star (z_+ \star y)) \star z_-) \star z_-\big), \
\big((z_+ \star ((z_+ \star y) \star z_-)) \star z_-\big), \
\big(z_+ \star (((z_+ \star y) \star z_-) \star z_-)\big), \
$$
$$
\big((z_+ \star (z_+ \star (y \star z_-))) \star z_-\big), \
\big(z_+ \star ((z_+ \star (y \star z_-)) \star z_-)\big), \
\big(z_+ \star (z_+ \star ((y \star z_-) \star z_-))\big).
$$
Besides,  the derivation $\eta(z)$ can be written as the sum of two derivations $\eta_+(z)$ and $\eta_-(z)$ 
by using the decomposition
$$
\Der(\Liehat,\Liehat_{\geq 2}) \simeq \Hom(H,\Liehat_{\geq 2}) \simeq H \otimes \Liehat_{\geq 2}
= (H_+ \otimes \Liehat_{\geq 2}) \oplus (H_- \otimes \Liehat_{\geq 2}),
$$
where $H_+$ and $H_-$ denote the subspaces of $H$ spanned by the longitudes
$b_1,\dots,b_g$ and the meridians $a_1,\dots,a_g$ respectively.
We deduce that, for any integer $n\geq 0$,
\begin{eqnarray*}
(-\eta(z))^{\circ n}(y) &=& \left(-\eta_+(z) -\eta_-(z) \right)^{\circ n}(y)\\
&=&(-1)^n \cdot \! \! \! \! \! \!  \! \! \!  \sum_{p:\{1,\dots,n\}\to\{+,-\}}\eta_{p(1)}(z)\circ  \cdots \circ \eta_{p(n)}(z) (y)\\
&=& \hbox{\scriptsize connected part of }\ \sum_{i=0}^n \binom{n}{i}\ {z_+}^{\star i} \star y \star {z_-}^{\star(n-i)} \mod \homotopy(r)
\end{eqnarray*}
where the last identity follows from the previous observation and  the definition of the multiplication $\star$.
We conclude that $\rhoLMO(C) = \exp_\circ(-\eta(z))$. 
\end{proof}

\vspace{0.5cm}

%
%
%
%
%
%

\bibliographystyle{abbrv}

\bibliography{IMH}

\end{document}